\documentclass[aos]{imsart}
\RequirePackage[colorlinks,citecolor=blue,urlcolor=blue]{hyperref}

\usepackage{listings}
\usepackage[most]{tcolorbox} 

\usepackage[normalem]{ulem}       

\usepackage{algorithmic}
\usepackage{algorithm}
\usepackage{array}
\usepackage{subcaption}
\usepackage{textcomp}
\usepackage{stfloats}
\usepackage{url}
\usepackage{verbatim}
\usepackage{graphicx,tabularx}
\usepackage[normalem]{ulem}

\RequirePackage{amsthm,amsmath,amsfonts,amssymb}
\RequirePackage[numbers,sort&compress]{natbib}
\RequirePackage{graphicx}

\usepackage{float}
\usepackage{amsmath}
\usepackage{amsthm}
\usepackage{amssymb}
\usepackage{booktabs}

\usepackage{xcolor}
\usepackage{stackengine}

\startlocaldefs
\theoremstyle{plain}
\newtheorem{thm}{Theorem}[section]
\newtheorem{lem}[thm]{Lemma}
\newtheorem{corollary}[thm]{Corollary}
\newtheorem{prop}[thm]{Proposition}
\theoremstyle{remark}
\newtheorem{definition}[thm]{Definition}
\newtheorem{rem}[thm]{Remark}
\newtheorem*{assumptiontrain*}{Assumption [Train]} 
\newtheorem*{assumptiontest*}{Assumption [Test]} 

\endlocaldefs

\newcommand*{\vertbar}{\rule[-1ex]{0.5pt}{3.5ex}}

\newcommand{\bdradd}[1]{#1}
\newcommand{\bdrdel}[1]{}
\newcommand{\bdrdelp}[1]{}
\newcommand{\bdrminor}[1]{}
\newcommand{\vll}[1]{}
\newcommand{\malinas}[1]{}

\newcommand{\exkurt}{\gamma_2} 
\newcommand{\mynone}{n_2}

\newcommand{\pit}{w}

\newcommand{\bW}{\mathbf{W}}
\newcommand{\la}{\left\langle}
\newcommand{\ra}{\right\rangle}
\newcommand{\lv}{\left\Vert}
\newcommand{\rv}{\right\Vert}

\newcommand{\fH}{\mathcal{H}}

\newcommand{\mydot}{\,\cdot\,}
\newcommand{\asp}{\phi}
\newcommand{\nasp}{\phi_n}
\newcommand{\mysig}{\Sigma}
\newcommand{\bfr}{\boldsymbol{\mysig}}
\newcommand{\bfrh}{\hat{\bfr}}

\newcommand{\myh}{h}
\newcommand{\myH}{H}

\newcommand{\bX}{\mathbf{X}}

\newcommand{\bbR}{\mathbb{R}}

\newcommand{\bx}{\mathbf{x}}

\newcommand{\bS}{\mathbf{S}}

\newcommand{\tr}{\mathrm{tr}}

\newcommand{\bu}{\mathbf{u}}
\newcommand{\resolv}{\mathbf{R}}

\newcommand{\buni}{\bu_{ni}}
\newcommand{\lamni}{\lambda_{ni}}

\newcommand{\bz}{\mathbf{z}}

\newcommand{\by}{\mathbf{y}}

\newcommand{\PHT}{SRHT}

\newcommand{\bbE}{\mathbb{E}}

\newcommand{\OPtilde}{O_{\tilde{P}}}
\newcommand{\oas}{o_{P}} 

\newcommand{\oOPof}[1]{O_\prec\left(#1\right)}
\newcommand{\oasneta}{O_\prec\left(\frac{1}{n\eta} \right)}

\newcommand{\myTheta}{\Theta}
\newcommand{\uTheta}{\myTheta_{\infty}}

\newcommand{\aspdom}{(0,1)}

\newcommand{\ox}{\overline{\bx}}
\newcommand{\oxn}{\ox_n}

\newcommand{\sigtil}{\tilde{\Sigma}}
\newcommand{\mythresh}{\tau}
\newcommand{\tps}{'}

\newcommand{\nind}{{ni}}
\newcommand{\lamnind}{\lambda_\nind}
\newcommand{\unind}{\bu_\nind}

\newcommand{\mun}{\mu_n}

\newcommand{\MP}{Mar\v{c}enko--Pastur}
\newcommand{\lp}{Ledoit--P\'{e}ch\'{e}}
\newcommand{\bSn}{\bS_n}
\newcommand{\fSn}{f(\bSn)}
\newcommand{\shrinker}{f_n(\bSn)}
\newcommand{\bXn}{\bX_n}
\newcommand{\bXdmn}{\bX_n}
\newcommand{\bR}{\mathbf{R}}

\newcommand{\uthetaonez}{\myTheta_\infty(z)}

\newcommand{\lwEig}{\tilde{d}_n}

\newcommand{\lwEigX}{\tilde{d}_n(x)}
\newcommand{\lpqfnilong}{\buni\tps\bfrn\buni}

\newcommand{\bfrhalf}{\bfr^{1/2}}
\newcommand{\bfrhalfn}{\bfrhalf_n}
\newcommand{\bmun}{\bmu_n}

\newcommand{\bZ}{\mathbf{Z}}
\newcommand{\bZn}{\bW_n}

\newcommand{\bfrn}{\bfr_n}
\newcommand{\popeig}{\tau}

\newcommand{\popsm}{\pi}
\newcommand{\popsmn}{\popsm_n}
\newcommand{\upopsm}{\popsm_\infty}

\newcommand{\msample}{m}
\newcommand{\msamplen}{\msample_n}
\newcommand{\um}{\msample_\infty}
\newcommand{\brevem}{\breve{\msample}}
\newcommand{\Cplus}{\mathbb{C}^+}

\newcommand{\bxnind}{\bx_{ni}}
\newcommand{\strans}{\mathcal{S}}
\newcommand{\myW}{W}
\newcommand{\myw}{w}
\newcommand{\mywx}{\myw(x)}
\newcommand{\suppumu}{F}

\let\Im\relax
\DeclareMathOperator{\Im}{Im}

\newcommand{\umu}{\mu_\infty}

\newcommand{\LW}{Ledoit--Wolf}
\newcommand{\bI}{\mathbf{I}}
\newcommand{\bRn}{\bR_n}
\newcommand{\bRj}{\bR^{(j)}(z)}
\newcommand{\bRk}{\bR^{(k)}(z)}
\newcommand{\bRjk}{\bR^{(jk)}(z)}

\newcommand{\hatm}{\hat{m}}

\newcommand{\myddiff}{d}

\newcommand{\lpmeas}{\nu}
\newcommand{\lpmeasn}{\lpmeas_n}
\newcommand{\ulpmeas}{\lpmeas_\infty}
\newcommand{\thetaone}{\Theta}
\newcommand{\thetaonenz}{\thetaone_n(z)}

\newcommand{\mywt}{\tilde{\myw}}
\newcommand{\mywtn}{\mywt_n}
\newcommand{\mywtnx}{\mywtn(x)}

\newcommand{\mysmall}{\Delta}
\newcommand{\myfunc}{a}
\newcommand{\myfuncsmall}{\myfunc_{\mysmall}}

\newcommand{\wmu}{\mywt}
\newcommand{\Hwmu}{\fH\wmu}
\newcommand{\Hw}{\fH\myw}
\newcommand{\Hwt}{\fH\mywt}
\newcommand{\Hwtn}{\Hwt_n}
\newcommand{\Hwtnx}{\Hwtn(x)}

\newcommand{\deltil}{\tilde{d}_n}
\newcommand{\wdel}{\myw_\mysmall}

\newcommand{\bc}{\mathbf{c}}

\newcommand{\RMold}[1]{}

\newcommand{\myclx}{v}

\newcommand{\tilhilb}{\mathcal{K}}

\newcommand{\tsig}{\tilde{\sigma}}
\newcommand{\tsign}{\tsig_n}

 \newcommand{\byn}{\by_n}
 \newcommand{\bzn}{\bz_n}
\newcommand{\bmuyn}{\bmu_n}
\newcommand{\mean}{m^o}
\newcommand{\meann}{\mean_n}
\newcommand{\tmean}{\tilde{m}}
\newcommand{\tmeann}{\tmean_n}
\newcommand{\sig}{\sigma}
\newcommand{\sign}{\sig_n^o}
\newcommand{\Zo}{Z^o}

\newcommand{\Zorign}{\Zo_n}

\newcommand{\Zanalytic}{\tilde{Z}}
\newcommand{\Zanalyticn}{\Zanalytic_n}
\newcommand{\hs}{{\text{HS}}}

\newcommand{\mysigma}{\sigma_{\infty}}
\newcommand{\myTsquared}{\tilde{T}^2_n}

\newcommand{\myg}{g}
\newcommand{\myG}{G}

\newcommand{\myGt}{\myG}

\newcommand{\mygtnx}{\myg_{n}(\lambda_{nj})}
\newcommand{\myGtnx}{\overline{\myGt}_n(\lambda_{nj})}
\newcommand{\myhnj}{\overline{h}_n(\lambda_{nj})}

\newcommand{\fU}{\mathcal{U}_\infty}
\newcommand{\fUn}{\mathcal{U}_n}
\newcommand{\omegainfty}{\omega_\infty}

\newcommand{\fopt}{f^*}

\newcommand{\myGamma}{\Gamma}
\newcommand{\myGamman}{\tilde{\myGamma}_n}
\newcommand{\myT}{T}
\newcommand{\myTn}{\tilde{\Gamma}_n}

\newcommand{\fstar}{f_*}

\newcommand{\Tprimeinv}{(\myT \tps)^{-1}}

\newcommand{\magnitude}{\gamma}

\newcommand{\htrans}{\mathcal{H}}
\newcommand{\bsalg}{\textsc{BS96}}
\newcommand{\cqalg}{\textsc{CQ10}}
\newcommand{\lalg}{\textsc{LAPPW20}}
\newcommand{\lwalg}{\textsc{LW22}}
\newcommand{\cwhalg}{\textsc{CWH11}}
\newcommand{\bmu}{\boldsymbol{\mu}}
\newcommand{\myOmega}{\boldsymbol{\Omega}}
\newcommand{\myXi}{\boldsymbol{\Xi}}

\newcommand{\altm}{\underline{m}_\infty}
\newcommand{\altmNaught}{\underline{m}_0}

\newcommand{\Htildew}{\tilde{\fH}_w}
\newcommand{\Hsubw}{\fH_w}

\newcommand{\omegafcn}{h}
\newcommand{\mya}{a}

\newcommand{\blue}[1]{#1}

\AtBeginDocument{

}

\begin{document}

\numberwithin{equation}{section}

\begin{frontmatter}
\title{Spectrally robust covariance shrinkage for Hotelling's \(T^2\) in high dimensions}
\runtitle{Robust Shrinkage for Hotelling's \(T^2\)}

\begin{aug}

\runauthor{Robinson and Latimer}



\author[A]{\fnms{Benjamin D.}~\snm{Robinson} \ead[label=e1]{srht.research370@passmail.net}\orcid{0000-0002-9391-45}}
\and
\author[B]{\fnms{Van}~\snm{Latimer}\ead[label=e2]{van.latimer@radialrad.com}\orcid{0000-0001-9789-77}}
\address[A]{Information and Networks Division, Air Force Office of Scientific Research \printead[presep={ ,\ }]{e1}}

\address[B]{
Radial Research and Development, Inc. \printead[presep={,\ }]{e2}}

\end{aug}

\begin{abstract}
We investigate covariance shrinkage for Hotelling's $T^2$ in the regime where the data dimension $p$ and the sample size $n$ grow in a fixed ratio---without assuming that the population covariance matrix is spiked or well-conditioned.  
When $p/n\to\phi \in (0,1)$, we propose a practical finite-sample shrinker that, for a maximum-entropy signal prior and any fixed significance level,  (a) asymptotically optimizes power under Gaussian data, and (b) asymptotically saturates the Hanson--Wright lower bound on power in the more general sub-Gaussian case.  
Our approach is to formulate and solve a variational problem characterizing the optimal limiting shrinker, and to show that our finite-sample method consistently approximates this limit via extensions of recent local random matrix laws.
Empirical studies on simulated and real-world data, including the \textsc{Crawdad} UMich/RSS data set, demonstrate up to a $50\%$ gain in power over leading linear and nonlinear competitors at a significance level of $10^{-4}$.
\end{abstract}



\received{June 2025}
\revised{July 2026}

\end{frontmatter}

\maketitle

\section{Introduction}

Comparing the means of two samples with a shared covariance matrix is a fundamental problem in multivariate statistics, and Hotelling's \(T^2\) test is a popular procedure when the sample size \(n\) dominates the dimension \(p\) \cite{anderson1963asymptotic,muirhead2009aspects}.  
However, in the more modern high-dimensional regime of \cite{marvcenko1967distribution} where \(n\) scales linearly with \(p\), the test is inconsistent even when well-defined \cite{bai1996effect}.  
Several authors have proposed addressing this issue by shrinking the sample covariance matrix \cite{kai2015high,li2020adaptable,robinson2022improvement}: applying a linear or even nonlinear function---called a \textit{shrinker}---to the eigenvalues while keeping the eigenvectors fixed \cite{stein1975estimation,ledoit2004well,mestre2005finite,chen2010shrinkage,ledoit2020analytical}.  
While such methods are effective under strong assumptions on model order or condition number, they can fall far short of optimality when many strongly correlated nuisance signals are present \cite{robinson2022improvement}, a common scenario in finance, genomics, and other applications \cite{bergin2002high,johnstone2009consistency,ledoit2017nonlinear}.
This limitation motivates the development of covariance shrinkage techniques that remain robust across a significantly broader range of spectral distributions in high dimensions.

Covariance shrinkage under general spectral conditions is challenging for several reasons.  
First, linear shrinkers \cite{tikhonov1943stability,hoerl1970ridge,haff1980empirical,carlson1988covariance,mestre2005finite,chen2010shrinkage,couillet2015robust,bodnar2016direct} often fail to capture the complexity of a general population spectrum since they have at most two degrees of freedom.
Previous work has indeed shown that in many contexts optimal  shrinkers belong to the much richer \textit{nonlinear} class, which consists of functions with \(k\) continuous derivatives for some \(k\) and thus has infinitely many degrees of freedom \cite{stein1975estimation,dey1985estimation,nadakuditi2014optshrink,donoho2018optimal,ledoit2018optimal}.  Second, implementing optimal nonlinear shrinkers can be difficult: the first practical versions, devised by Ledoit and Wolf, are still relatively new \cite{ledoit2018optimal,ledoit2020analytical,ledoit2022quadratic}.
Third, deriving the asymptotic null distribution of the shrinkage-modified Hotelling \(T^2\) statistic is delicate---even in the case of linear shrinkage \cite{pan2011central,li2020adaptable}.  
Finally, identifying and optimizing an effective detection criterion for the test under fully general spectral conditions remains an open problem \cite{namdari2021high}.

In this paper, we develop a method of covariance shrinkage for Hotelling's \(T^2\) that is robust to large condition numbers, high model orders, and deviation from Gaussianity.   
Building on the anisotropic local-law framework of Knowles and Yin \cite{knowles2017anisotropic}, we adopt a random matrix model that accommodates many families of matrices with such spectral features.
This model also accommodates Johnstone's spiked model \cite{johnstone2001distribution} and many of the generalized spiked models in \cite{bai2012sample} under finite-rank perturbations, which recent literature suggests leave the relevant asymptotic limits unaffected in analogous settings \cite{ding2024eigenvector}.
Moreover, our method is robust to sub-Gaussian departures from the Gaussian model, as long as the data have independent components in a fixed basis. 

 In parallel with various recent works on eigenvector overlaps \cite{ding2024eigenvector,lin2026eigenvector}, we first establish a stochastic convergence rate associated with Ledoit and Wolf's nonlinear shrinkage eigenvalues from \cite{ledoit2020analytical} in Theorem~\ref{thm:local-lw}---rates that provide the foundation for our remaining results. Focusing on the regime of \(p\to\infty\) with \(p/n\to\asp\in (0,1)\), our main contributions to statistical theory are as follows: 
\begin{enumerate} 
    \item In Theorem~\ref{thm:significance}, we establish asymptotic control of size for the Hotelling \(T^2\) test modified with covariance shrinkage.
    \item In Theorem~\ref{thm:deterministic-limit}, (a) we present a variational problem characterizing the shrinker that asymptotically optimizes a sub-Gaussian lower bound on power, 
    under a maximum-entropy signal-prior assumption [\textsc{Test}]; and (b) we express this optimal shrinker in terms of Hilbert transforms.
    \item In Theorem~\ref{thm:approx-fopt}, we present a practical finite-sample approximation that achieves the same asymptotic detection performance as the optimal shrinker.
\end{enumerate}
Taken together, these results give the first usable shrinker that asymptotically optimizes Hotelling's \(T^2\) test---up to sub-Gaussian concentration bounds---across a wide range of spectral complexity.

We present these results and their preliminaries in the following sections.  Section~\ref{sec:background} defines the shrinkage-modified Hotelling \(T^2\) statistic and provides a preview of the proposed shrinker.  Section~\ref{sec:RMT} introduces our sub-Gaussian data model and reviews the relevant local laws from universal random matrix theory.  In Section~\ref{sec:nonlinear-shrinkage}, we develop the nonlinear shrinkage framework and present our convergence-rate analysis of Ledoit and Wolf's shrinkage eigenvalues. Section~\ref{sec:EDs} finishes introducing our hypothesis-testing model and contains our three main theoretical results. 
Finally, Section~\ref{sec:empirics}  presents empirical results on simulated and real-world data, including the \textsc{Crawdad} UMich/RSS data set \cite{c7r30h-22}, 
and demonstrates up to a $50\%$ gain in power over leading competitors at a significance level of \(10^{-4}\).  Additional technical details, simulations, and algorithms can be found in the Supplementary Material following this paper's main body. Code for this paper's simulations can be found at \url{https://github.com/bnrbnsn/SRHT}.

\section{Problem and preview of proposed shrinker} \label{sec:background}
\blue{Anomaly detection} in this paper, as addressed by Hotelling's \(T^2\), will refer to the problem of \blue{determining whether the unknown mean of a vector observation is equivalent to the unknown mean} of a reference data set having the same covariance matrix.
Formulated as a hypothesis test, the question is, given i.i.d$.$ random vectors 
 \(\bx_1, \bx_2, \dots, \bx_{n} \sim P\) and a random vector \(\by \sim Q\) with \(\operatorname{cov}(\bx_1) = \operatorname{cov}(\by)\), to decide between \(\fH_0: \bbE[\bx_1]=\bbE[\by]\) and \(\fH_1: \bbE[\bx_1]\ne\bbE[\by]\).

\blue{ 
Although in the traditional setting two-sample testing requires that multiple i.i.d$.$ copies of both $\bx_1$ and $\by$ are available, we focus primarily on the challenging regime of one ``test'' sample $\by$.  This setting is of fundamental importance in statistical signal processing applications---both in online and offline analysis.
Some examples include distributed motion detection in the presence of complex noise and interference \cite{chen2011robust}, data censoring in radar space-time adaptive processing \cite{chen1999screening, adve1999transform}, and out-of-distribution detection for deep neural classifiers \cite{lee2018simple}.  In these settings, detection-delay costs or nonstationarity constraints may mean that it is often not sensible to accumulate test data beyond a singleton before a decision is made.
}

When the common covariance matrix \(\bfr\) is known, a typical test statistic for \blue{anomaly} detection is the quadratic-form detector given by
\begin{equation} \label{eq:mahalanobis}
(\by-\ox)\tps\bfr^{-1}(\by-\ox) \underset{\fH_0}{\overset{\fH_1}{\gtrless}} \mythresh,
\end{equation}
where \(\ox\) is  the sample mean of \(\{\bx_i\}\), \(\mythresh\) is some detection threshold, and \((\,\cdot\,)\tps\) denotes the transpose operation.  
However, unless \(n\) is large---\textit{i.e.}, much larger than \(p\)---it is not realistic to assume that one can approximate all the \(p^2\) entries of \(\bfr\) with much accuracy.  In this case,
it is standard to replace \(\bfr\) by some estimator \(\bfrh\).
The focus of this paper is on
how to choose \(\bfrh\) in the \textit{high-dimensional regime,} where \(n\) and  \(p\) go to infinity and \(p/n \to \asp\in(0,1)\).

Very often \(\bfrh\) is chosen to be the Bessel-corrected sample covariance matrix of \(\{\bx_i\}\), defined to be
\begin{equation} \label{eq:scm}
\bSn := \frac{1}{n-1}\sum_{i=1}^n (\bx_i-\ox)(\bx_i-\ox)\tps = \frac{1}{n-1}\bXdmn\left(\mathbf{I}_n-\frac{1}{n}\mathbf{1}_n \mathbf{1}_n'\right)\bXdmn\tps,
\end{equation}
where \(\bXdmn\) is the \(p\times n\) matrix 
\([\bx_1, \bx_2, \dots, \bx_n ]\), \(\mathbf{I}_n\) is the \(n\times n\) identity matrix, and \(\mathbf{1}_n\) is the \(n\times 1\) vector \((1,1,\dots, 1)'\).
The resulting plug-in quadratic-form detector can be identified with Hotelling's \(T^2\) statistic \cite{muirhead2009aspects} for large \(n\).  However, although this choice minimizes asymptotic detection loss compared to \eqref{eq:mahalanobis} when \(n\gg p\), there is no reason to believe it has the same property in the  high-dimensional regime.
Indeed, Hotelling's \(T^2\) is inconsistent in this regime \cite{bai1996effect}.

A natural alternative approach is to replace the sample precision matrix \(\bSn^{-1}\), with eigen-decomposition \(\sum_{i=1}^p \lambda_i^{-1}\bu_i\bu_i\tps\), by a \textit{shrinkage estimator} \(\fSn\) with eigen-decomposition \( \sum_{i=1}^p f(\lambda_i)\bu_i\bu_i\tps\), where \(f\) is some non-negative deterministic function we call the \textit{shrinker}.  
More in line with the literature, we overload the term shrinker to mean not only deterministic shrinkers but also more general functions \(f_n\) whose values could depend on random \(n\)-dependent quantities, such as the full spectrum of \(\bSn\) \cite{stein1975estimation}.  
In any case, we call the resulting plug-in test statistic a \textit{shrinkage-regularized Hotelling \(T^2\) statistic}, or \textit{\PHT{}}, defined by
\begin{equation} \label{eq:eed}
\myTsquared \equiv \myTsquared(f_n) := (\byn-\oxn)\tps \shrinker(\byn-\oxn).
\end{equation}

The most common type of shrinker  \(f_n\) corresponds to \textit{linear} shrinkage and takes the form \(f_n(x)=1/(a_n x+b_n)\) for some scalars \(a_n, b_n > 0\).  
This choice corresponds to the ridge-regularized Hotelling \(T^2\) test statistics of \cite{li2020adaptable}.
Linear shrinkage is extremely well-studied, but for all the effort devoted to it, virtually all procedures for finding \(a_n\) and \(b_n\)  appear to depend on prior knowledge, such as a tight bound on the minimum eigenvalue of \(\bfr\) or on its condition number or model order, which may not be practical to obtain.
We thus assume \(f_n\) is not necessarily of the linear class, and belongs to a more general class of \textit{nonlinear} shrinkers.

In the related contexts of financial portfolio optimization and radar space-time adaptive processing, a common optimal nonlinear shrinkage estimator has been identified in \cite{ledoit2020analytical} and \cite{robinson2021space}, respectively, under certain signal-prior assumptions and the high-dimensional regime in question. Under certain additional conditions relating to the covariance spectrum  and moments of the data, one choice for this shrinkage procedure replaces \(\bfrh\) with \(\delta(\bSn)\), where 
\begin{equation} \label{eq:delta}
\delta(x) := \frac{x}{[1-\asp  - \pi\asp x\Hw(x)]^2 + \pi^2 \asp^2 x^2 w(x)^2}, \qquad (x\in \bbR)
\end{equation}
  \(w(x)\) is the \MP{} density \cite{marvcenko1967distribution}, and \(\Hw(x)\) is the Hilbert transform of \(w(x)\), given by the principal-value integral
\begin{equation} \label{eq:hilb-def}
\fH f(x) \equiv \fH[f](x)   :=  \mathrm{p.v.} \frac{1}{\pi} \int \frac{f(t)}{t-x}\, dt. \qquad (f \in L^2(\bbR))
\end{equation}
We note that the function \(\delta(x)\) is a key mapping appearing throughout this paper that can be approximated for finite sample sizes using
Theorem~\ref{thm:lw}.

In this paper we present an asymptotically optimal choice of the function \(f_n\) appearing in the statistic \(\tilde{T}^2_n\) of \eqref{eq:eed} for a sub-Gaussian linear model given by assumptions [\textsc{Train}] and [\textsc{Test}], which will be presented in the forthcoming sections.  To make the problem well-posed, we assume a signal prior of the form \(\mathcal{N}(0, \myOmega)\) and that the quantities \(\bu_{i}\tps \myOmega\bu_{i}w(\lambda_i)\) approximate \(h_n(\lambda_i)\) for some sufficiently regular function \(\myh\).  Letting \(a(x) = x w(x)\delta(x)\), an optimal shrinker is the function
\[
f^* = \frac{g^2 \myh + g G \myH}{a} - \mathcal{H}\left[ \frac{G^2 \myH + Gg \myh}{a}\right],
\]
where 
\begin{equation} \label{eq:gG}
g(x) := 1-\asp - \asp \pi x\Hw(x)\qquad\text{and}\qquad G(x):= -\asp \pi x w(x)
\end{equation}
and \(\myH := \mathcal{H}h\).  The shrinker \(f^*\) can be obtained by solving a singular integral equation arising from the variational problem that appears in Theorem~\ref{thm:deterministic-limit}.  Under the model to be presented, this shrinker asymptotically optimizes power for any chosen significance level if the data happen to be Gaussian, and  otherwise asymptotically saturates the Hanson--Wright lower bounds on power at any chosen significance level if the data have sub-Gaussian components. 
Since \(f^*\) is generally unobservable, a suitable finite-sample approximation  \(f_n\) to \(f^*\) is given in Theorem~\ref{thm:approx-fopt}.

In what follows, we will ground these ideas in a firm theoretical foundation.  To do this, our first task will be to summarize and develop key results in random matrix and nonlinear shrinkage theory over the next two sections. 

\begin{rem}
 It is worth noting before proceeding that Hotelling's \(T^2\)
is not always used to compare a singleton sample to another sample, but is frequently used to compare two samples of cardinality greater than one.
We expect our results to hold in this case as well, but some additional technical arguments are required in the non-Gaussian case to handle dependencies between sample means and covariance estimators.  For ease of exposition, we omit the details of this extension except to note that the analysis of \cite[Section~4]{li2020adaptable} appears to provide a blueprint.
\end{rem}

\section{Background on random matrix theory} \label{sec:RMT}

The starting point of nonlinear shrinkage theory is the celebrated \MP{} Theorem \cite{marvcenko1967distribution,silverstein1995strong}, which is foundational in high-dimensional inference.  In this section, we will state this theorem in one of its modern forms.   

  We emphasize that the samples \(\bxnind\) and sample means \(\oxn\), both of increasing dimension, depend on \(n\) (and \(p\)) by denoting them as \(\bxnind\) and \(\oxn\).  We write the sample matrix again in block form as \(\bXn = [\bx_{n1}, \bx_{n2}, \dots, \bx_{nn}]\), which can be visualized as in the following diagram:
\def\tmpA{
\begin{bmatrix}
\vertbar & \vertbar & ~ & \vertbar \\
\bx_{n1} & \bx_{n2} & \cdots & \bx_{nn} \\
\vertbar & \vertbar & ~ & \vertbar
\end{bmatrix}
}
\def\tmpB{
\begin{bmatrix}
\vertbar & \vertbar & ~ & \vertbar \\
\oxn & \oxn & \cdots & \oxn \\
\vertbar & \vertbar & ~ & \vertbar
\end{bmatrix}
}
\[ 
\stackMath\def\stackalignment{r}%
  \stackon{%
    \ensuremath{\normalsize p} \left\{\tmpA \right.%
  }{%
    \overbrace{\phantom{\smash{\tmpA\mkern -36mu}}}^{\scriptstyle \ensuremath{n}} \mkern 25mu%
  }
\quad
\]
Given that \(\bSn = (n-1)^{-1}\bXn(\mathbf{I}_n-n^{-1}\mathbf{1}_n\mathbf{1}_n\tps)\bXn\tps\), as before, the norm difference of \(\bSn\) and \(\tilde{\bS}_n = n^{-1} \bXn\bXn\tps\) goes to zero when the population has mean zero. As a result of this and shift invariance of Hotelling's \(T^2\), we will frequently assume the mean vanishes and identify \(\tilde{\bS}_n\) with \(\bSn\), commenting on any complications that arise in the case of nonzero mean as necessary.

Except when otherwise noted (\textit{e.g.}, in the nonzero-mean case), we will assume throughout the rest of the paper the following data model for the \(\bxnind\)'s.

\begin{assumptiontrain*} 
\leavevmode
\begin{itemize}
    \item {[\textsc{Train1}]} \(\bXn\) 
can be expressed as \( \bfrhalfn \bZn\),  
 where \(\bZn\) is a matrix of i.i.d$.$ random variables \(\sim W\) with zero mean, unit variance, and sub-Gaussian tails: 
 \[
 \bbE[e^{sW}]\le e^{s^2/2} \qquad \text{for all real \(s\).}
 \]
     \item {[\textsc{Train2}]} \(\bfrn\) is a deterministic symmetric positive-definite \(p\times p\) matrix with eigenvalues \( \popeig_{n1}, \popeig_{n2}, \dots \popeig_{np}\).
    \item {[\textsc{Train3}]} There is \(\asp \in \aspdom \) such that \(|p/n - \asp| = O(1/p)\) as \(n, p\to\infty\).
    \item {[\textsc{Train4}]} As \(n,p\to\infty\), (a) defining \(\delta_\tau\) to be the Dirac mass at \(\tau\), the population spectral distribution \(\popsmn = p^{-1}\sum_{i=1}^p \delta_{\popeig_{ni}}\) of \(\bfrn\) differs in 1-Wasserstein distance by \(O(1/p)\) from some limiting distribution
    $\upopsm$ with compact positive support, 
and (b) all population eigenvalues eventually lie in a positive compact interval.
    \item {[\textsc{Train5}]} The sequence
    $\{\pi_n\}$ is \textit{regular} in the sense of \cite[Definition~2.7]{knowles2017anisotropic}---for example, when $\pi_\infty$ has the property that \(d\upopsm(x)/dx\) is bounded above and below and has a positive compact interval for support.
    \end{itemize}
\end{assumptiontrain*}

The conditions [\textsc{Train}] strengthen somewhat but align closely with the assumptions in the main body of \cite{ledoit2020analytical}.  Although many of these conditions can be relaxed, doing so would substantially complicate the proofs of our main theoretical results. Nevertheless, we briefly mention some significant potential relaxations. 
\begin{itemize}
\item The moment condition  in {[\textsc{Train1}]}, also known as \textit{one-sub-Gaussianity}, can be relaxed to a sub-Weibull-type condition without much conceptual difficulty, but the proofs become messier notationally.
\item  The assumption of regularity in {[\textsc{Train5}]} appears in a wide range of practical models, includes the identity matrix as a special case \cite[Example~2.8]{knowles2017anisotropic}, and ensures the highly desirable property of square-root edge-behavior of the \MP{} density, but regularity technically excludes spiked covariances.  However, it turns out we lose no generality by relaxing {[\textsc{Train5}]} to allow finite-rank perturbations of \(\bfrn\)---including spiked and many generalized spiked models.  See \cite[Section~2.1]{ding2024eigenvector} and subsequent sections for the details of this extension, which we omit here.  
\item A third relaxation involves the assumption in {[\textsc{Train3}]} that \(\asp \in (0,1)\).  We discuss the delicate ``singular'' case of $\asp \in (1,\infty)$ further in Remark~\ref{rem:sample-starved} 
and Section~S.10
of the Supplemental Material.
\item We make the convergence rate assumption [\textsc{Train3}] explicit for convenience, but from a practical point of view, it is not essential.  The reason: in practice, the data matrix $\bXn$ is only given for a single $n$.  In other words, the limiting quantity $\asp$ is not observed, and can be taken to approximate $p/n$ as closely as desired, or even equal it. Similarly, the convergence rate in [\textsc{Train4}] may seem restrictive, but for the same reason, does not impose any practical difficulty. 
To expand, since only one sample is given, one may simply redefine $\upopsm$ to be $\pi_n$ for sufficiently large $n$, after which the convergence rate in 1-Wasserstein distance is arbitrarily fast.  The value of these formulations is that (a) having an infinite-dimensional limit enables the use of a large body of existing random matrix theory, and (b) having an order of convergence enables downstream convergence-rate guarantees to be made.
\end{itemize}

We note that the assumption {[\textsc{Train3}]} means that we will be able to use the condition ``as \(n\to\infty\)'' as a substitute for the equivalent but more verbose condition ``\(n,p\to\infty\) and 
\(p/n = \asp + O(1/p)\).''  
Even more briefly, when we express a convergence relation, such as  \(\left\Vert \bfrn - \bI_p\right\Vert \to 0\) without identifying a limiting regime, it will be understood that this convergence occurs as \(n\) (and thus \(p\)) goes to infinity.
\blue{ 
The only exception will be inside theorem environments, so that the statements of main results remain as self-contained as possible.
}

Our primary notion of convergence rate in probability will be that of \textit{stochastic domination}, defined as follows as in  \cite{knowles2017anisotropic} :
\begin{definition}
    Let \(X_n(u)\) and \(Y_n(u)\)  be random variables,  parameterized by a family of real vectors \(u \in U\). We say \(X_n\) is stochastically dominated by \(Y_n\)  uniformly in \(u\)  if for all \(\epsilon, D > 0\) we have
    \[
    \sup_{u\in U} \Pr\left[ X_n(u) < n^\epsilon Y_n(u) \right] \le n^{-D}
    \]
    for all \(n\) greater than some \( n_0(\epsilon, D)\).  More briefly, we write \(|X_n|\prec Y_n\) or \(X_n = O_\prec(Y_n)\). More generally, if for every \(D>0\) there is \(n\) such that a property holds with probability at least \(1-n^{-D}\) for all larger \(n\), we say the property holds \textit{with high probability}.  
    
    Quantities converging to probability in the usual sense will often be denoted by  \(o_P(1)\).
\end{definition}
\begin{rem}
    The underlying set \(U\) parameterizing random variables of interest will often not be explicitly identified; instead the notation \(O_\prec(\mydot)\) will be implicitly taken to mean uniform stochastic domination over all parameters (\textit{e.g.}, matrix indices or complex spectral parameters) not expressly deemed constant.
\end{rem}

We note that the sample covariance matrix can be expressed just as before, except with terms indexed by \(n\) to indicate their increasing dimension and dependence on \(n\):
\[\bSn = (n-1)^{-1}\sum_{i=1}^n (\bxnind - \ox_n)(\bxnind - \ox_n)\tps.
\]
Further, we will frequently make use of the sample eigen-decomposition:
\begin{equation*}
    \bSn = \sum_{i=1}^p \lamni \buni\buni\tps,
\end{equation*}
where the unit eigenvectors \(\buni\) are almost surely uniquely defined up to a sign since the probability of repeated sample eigenvalues is zero. 

The \MP{} Theorem is essentially a statement about the limiting behavior of the resolvent \(\bR_n(z)\) of \(\bSn\), given by 
\begin{equation*}
    \bR_n(z) = (\bSn-z\bI_p)^{-1},
\end{equation*}
and the spectral measure of \(\bSn\): namely,
\begin{equation*}
    \mun = p^{-1}\sum_{i=1}^p \delta_{\lamni}. 
\end{equation*}

\vspace{1em} 

A strong form of the \MP{} Theorem due to Knowles and Yin is reproduced below.
\begin{thm}[Local \MP{} Laws \cite{knowles2017anisotropic}] \label{thm:local-mp}
    Assume \emph{[\textsc{Train}]} and that \(\mu_n\) is the spectral distribution of the sample covariance matrix \(\bSn\).  Let 
    \[
    \msamplen(z) := p^{-1}\tr \bR_n(z) .
    \]
    Then for all bounded \(z = x+i\eta\) in the complex upper half plane \(\Cplus\) with \(\eta \ge n^{-1+\epsilon}\) for some positive \(\epsilon\), 
    we have 
    \begin{equation} \label{eq:avgd-local-mp}
    \msamplen(z) = \um(z) + \oasneta 
    \end{equation}
    as \(n,p\to\infty\),
    where \(\um(z)\) is the unique solution to
    \begin{equation} \label{eq:for-umu}
    \um(z) = \int_0^\infty \frac{d\upopsm(\popeig)}{\popeig(1-\asp-\asp z \um(z))-z}.
    \end{equation}
    Furthermore, (the ``small-scale'' laws) there is a nonrandom measure \(\umu\) such that
    \begin{equation} \label{eq:small-scale-mp}
       \sup_{(a,b)\subset \mathbb{R}} \left|\mun(a,b) - \umu(a,b)\right| \prec \frac{1}{n}
    \end{equation}
    and, uniformly in \(f\in C^2\),
    \begin{equation}
        \int f\, d\mu_n = \int f\, d\umu + O_\prec\left(\frac{\lv f\rv_1}{n} + \frac{\lv f'\rv_1}{n} + \frac{\lv f''\rv_1}{n^2} \right),
    \end{equation}
    as \(n,p\to\infty\).
\end{thm}
\begin{proof}[Proof Sketch]
    The statement spanning \eqref{eq:avgd-local-mp} and \eqref{eq:for-umu} follows directly from the cited work.
    The small-scale laws then follow from well-established arguments by Helffer and Sj\"{o}strand,
   which are reproduced
    in Section~S.2
    of the Supplementary Material for convenience. 
\end{proof}

\begin{rem}
 Pursuant to the comments following \cite[Equation~2.9]{knowles2017anisotropic}, we note that the convergence rates in Theorem~\ref{thm:local-mp} are not affected by the condition number of \(\bfrn\).  
\end{rem}

The first known version of \eqref{eq:avgd-local-mp} and \eqref{eq:small-scale-mp} in the literature is due to \cite{marvcenko1967distribution} but does not include an explicit almost-sure convergence rate or uniformity over intervals \((a,b)\) and pertains to the case of diagonal \(\bfrn\).  For example, when \(\bfrn\) is the \(p\times p\) identity matrix for all \(n\), it is well-known that the density \(d\umu(x)/dx\) can  be expressed in closed form as 
\[
 \frac{1}{2\pi \asp x}\sqrt{(x-(1-\sqrt{\asp})^2)((1+\sqrt{\asp})^2-x
)}.
\]
   Numerical evidence of the small-scale law \eqref{eq:small-scale-mp} for the case of identity covariance can be found in
 Figure~\ref{fig:mp}.
    In another major development, \cite{silverstein1995strong} extended the work of \cite{marvcenko1967distribution} to the case of matrices \(\bfrn\) that are not necessarily diagonal. 
    However, again no convergence rate is provided. Theorem~\ref{thm:local-mp}, by contrast, provides a convergence rate that is essentially sharp, \textit{i.e.}, up to factors of \(n^\epsilon\).
    
\begin{definition} \label{def:MP-defs}
Regarding notation, throughout this paper we will let \(\suppumu:=\operatorname{supp}\umu \), let  \(w(x) := d\umu(x)/dx\), and let \(\kappa(x)\) be defined as the distance from \(x\in \bbR\) to  the boundary of \( F\).   We will also say that if \(a_n\) and \(b_n\) are sequences of complex numbers, \(a_n \lessapprox  b_n\) will mean \(a_n = O(b_n)\), and \(a_n \asymp b_n\) will mean both \(a_n \lessapprox b_n\) and \(b_n \lessapprox a_n\).
Tables~\ref{tab:RMT} and \ref{tab:RMT2} list the random matrix notation appearing throughout this paper for convenience.
\end{definition}

\vspace{1em}

Some important facts about these quantities are that 
 \(\suppumu\) is the union of finitely many disjoint closed intervals of nonzero length \cite{bai1998no}, and that \(w(x)\)
 is a smooth function except near the edges of these intervals, where \(\mywx\asymp\kappa(x)^{1/2}\) for \(x\in F\) \cite{knowles2017anisotropic}.

Define the Stieltjes transform of a positive measure \(\rho\)  by \(\strans[d\rho](z):= \int (t-z)^{-1}\, d\rho(t)\), so that \(\msamplen(z) = \strans[d\mun](z)\). 
It turns out that \(\umu\) is defined by the relation \(\um(z) = \mathcal{S}[d\umu](z)\) for \(z \in \Cplus\), and that extending \(\um(z)\) and related Stieltjes transforms to the real line will be of interest.
 In \cite{silverstein1995analysis}, it is shown that for \(x \in \bbR\) the following limit exits
\[	
\lim_{z\in \Cplus \to x} \um(z) =: \brevem(x),
\]
and, further \(\Im[\brevem(x)] = \pi w(x)\). 
Another fact is that the real part of \(\brevem(x)\) is given by \( \pi \Hw(x)\) \cite[Supplement~A.1]{ledoit2020analytical}.
Combining these two results about \(\brevem(x)\), we get
\begin{equation} \label{eq:brevem}
\brevem(x) = \pi \Hw(x) + i\pi w(x). \qquad(x\in \bbR)
\end{equation}
It follows in much the same way that if \(f \in C(F)\), we may write
\begin{equation} \label{eq:utheta}
    \lim_{z\in \Cplus\to x} \strans[f\, d\umu](z) = \pi\htrans[f w](x) + i\pi f(x)w(x)
\end{equation}
for all \(x\in \bbR\).  

\vspace{.25cm}
\begin{minipage}{\linewidth}
\centering
\setlength{\tabcolsep}{10pt}
\renewcommand{\arraystretch}{1.1}
\captionof{table}{Key random matrix symbols\label{tab:RMT}}
\begin{tabularx}{\linewidth}{>{\raggedright\arraybackslash}p{0.45\linewidth}X}
\toprule
\textbf{Symbol} & \textbf{Meaning / definition} \\
\midrule
$\asp$ & $\lim_n \frac{p}{n} \in (0,1)$ \\
$\bX_n=[\bx_{n1},\dots,\bx_{nn}]$ & Real $p\times n$ data matrix ($\bx_{ni}$ i.i.d.) \\
$\ox_n = \frac{1}{n}\sum_{i=1}^n \bx_{ni}$ & Sample mean \\
$(\mydot)\tps$, $(\mydot)^*$ &  Matrix transpose/conjugate-transpose\\
$\bS_n=\tfrac1{n-1}\sum_{i=1}^{n}(\bx_{ni}-\bar{\bx}_n)(\bx_{ni}-\bar{\bx}_n)\tps$ & Sample covariance matrix \\
$\delta_\lambda$ & Dirac mass at \(\lambda\in\bbR\) \\
$\{\lambda_{ni},\bu_{ni}\}_{i=1}^{p}$ & Eigen-pairs of $\bS_n$ \\
$\mu_n=\frac1p\sum_{i=1}^{p}\delta_{\lambda_{ni}}$ & Spectral distribution of \(\bSn\)  \\
$\bR_n(z)=(\bS_n-z\bI_p)^{-1}$ & Resolvent of $\bS_n$ at \(z\in\mathbb{C}\backslash\bbR\) \\
$\msamplen(z)=\frac1p\operatorname{tr}\bR_n(z)$ & Stieltjes transform of $\mu_n$ at \(z\)\\
\bottomrule
\end{tabularx}
\end{minipage}

\vspace{.25cm}

\section{Nonlinear shrinkage background and refinements} \label{sec:nonlinear-shrinkage}

In a seminal paper extending the classical \MP{} theorem of \cite{silverstein1995strong}, Ledoit and P\'{e}ch\'{e} \cite{ledoit2011eigenvectors} derive an almost sure deterministic limit for generalized resolvents of the form \(\myTheta^g_n(z) = p^{-1}\tr(g(\bfrn) \bRn(z))\), where \(g\) is a bounded piece-wise continuous function applied spectrally to \(\bfrn\).
These resolvents are important in covariance and precision shrinkage, as well as in the estimation of high-dimensional eigenvector biases, which appear, for example, in principal-components analysis.

We present in Subsection~\ref{sec:lp-theory} a refinement of the \lp{} theory that provides explicit convergence rates---analogous to Theorem~\ref{thm:local-mp}'s refinement of \cite{silverstein1995strong}. 
Then, in Subsection~\ref{sec:lw-theory}, we present a similar refinement of Ledoit and Wolf's work \cite{ledoit2020analytical}, which can be seen as the empirical version of Ledoit and P\'{e}ch\'{e}'s law.
These refinements will enable us to standardize the \PHT{}s of interest in Section~\ref{sec:EDs}.

\vspace{.25cm}
\begin{minipage}{\linewidth}
\centering
\setlength{\tabcolsep}{10pt}
\renewcommand{\arraystretch}{1.1}
\captionof{table}{More key random matrix symbols\label{tab:RMT2}}
\begin{tabularx}{\linewidth}{>{\raggedright\arraybackslash}p{0.35\linewidth}X}
\toprule
\textbf{Symbol} & \textbf{Meaning / definition} \\
\midrule
$\umu$,\;$w(x)=d\umu(x)/dx$ & The MP distribution and density \\
$\um(z)$ & Stieltjes transform of \(\umu\), limit of \(\msamplen(z)\)  \\
$\brevem(x)$ & \eqref{eq:brevem}, \(\lim\um(z)\) as \(z\in\Cplus\to x\in \bbR\) \\
$F=\operatorname{supp}\umu$ & The MP spectrum \\
$\kappa(x)=\operatorname{dist}(x, \partial F)$ & Distance to spectral edge \\
$\eta$ & $\Im z \in [n^{-1+\epsilon}, 1]$, some $\epsilon > 0$ \\
$O_\prec(a_n)$ & Stochastically dominated by \(a_n\) \\
$o_P(1)$ & Converges to zero in probability \\
$ a_n \lessapprox b_n$ & $|a_n/b_n|$ eventually bounded above \\
$ a_n \asymp b_n$ & $a_n \lessapprox b_n$ and $b_n \lessapprox a_n$  \\
$\htrans[f](x) = \htrans f(x)$ & \eqref{eq:hilb-def}, Hilbert transform of \(f\), evaluated at \(x\) \\
$\lv A \rv_{HS}$ & Hilbert-Schmidt norm \(\tr(AA^*)^{1/2}\) of matrix \(A\) \\
\bottomrule
\end{tabularx}
\end{minipage}
\vspace{.25cm}

\subsection{\lp{} Theory} \label{sec:lp-theory}
The functions \(\Theta_n^g(z)\) are of particular interest when \(g(x)=1\), just discussed in Theorem~\ref{thm:local-mp}, and when \(g(x)=\mathrm{id}(x) \equiv x\), which is the subject of this section.  When \(g = \mathrm{id}\), we will simplify notation by writing \(\thetaonenz\) instead of \(\myTheta^{\mathrm{id}}_n(z)\). Then the following analogue of the classical \MP{} theorem holds.
\begin{thm}[Theorem~4 of \cite{ledoit2011eigenvectors}]
\label{thm:lp}
Assume \emph{[\textsc{Train}]} and \(z\in\Cplus\).
Let
    \[
    \thetaonenz := p^{-1}\tr(\bfrn\bRn(z)).
    \]
    Then \(\thetaonenz =  \strans[d\lpmeasn](z)\), where 
\[
\lpmeasn:=\frac{1}{p}\sum_{i=1}^p \lpqfnilong \delta_{\lamni},
\]
  and
     there exists a nonrandom analytic function \(\uthetaonez\)  such that
    \begin{equation} \label{eq:avgd-local-lp}
    \thetaonenz = \uthetaonez + o_{a.s.}(1)
    \end{equation} 
    as \(n, p\to\infty\).
Furthermore, if 
 the function \(\delta:\mathbb{R}_+\to\mathbb{R}_+\) is defined using \eqref{eq:delta},
then
\begin{equation} \label{eq:lp}
d\lpmeasn \to \delta\, d\umu =: d\ulpmeas
\end{equation}
weakly almost surely as \(n, p\to\infty\).
\end{thm}
\begin{proof}
   The first statement follows from diagonalization of \(\bRn(z)\) and the circular-permutation property of trace.  Equation \eqref{eq:avgd-local-lp} follows from \cite[Theorem~2]{ledoit2011eigenvectors}.  Equation \eqref{eq:lp} follows from \cite[Theorem~4]{ledoit2011eigenvectors}.  
   
\end{proof}

 The choice \(g=\mathrm{id}\) studied in this section is important for several reasons.  First, the coefficients of the measure \(\lpmeasn\) resulting from \(\Theta^g_n(z)\) optimize  various loss functions of statistical interest, such as Frobenius loss \cite{ledoit2011eigenvectors}, reverse Stein loss \cite{ledoit2018optimal}, the Sharpe ratio appearing in financial portfolio optimization \cite{ledoit2017nonlinear}, and the Reed--Mallet--Brennan figure of merit in radar detection \cite{robinson2021space}, assuming an ignorant signal prior for the latter two.  A potential wrinkle is that \(\myw(x)\) has no known closed form expression except in a handful of special cases.  
 However, fortunately, Ledoit and Wolf \cite{ledoit2020analytical} have developed observable substitutes for \(\myw(x)\), \(\Hw(x)\), and \(\delta(x)\) that exhibit various forms of consistency, as we will discuss in Theorem~\ref{thm:lw}.

For the applications just mentioned, more machinery is not needed, but for the application in this paper, an explicit rate of convergence is needed between \(\myw(x)\) and its substitute.  We will address this rate in the next subsection, but first we need to refine Theorem~\ref{thm:lp} so that it includes precise convergence rates.

\begin{thm}[Local \lp{} Laws] \label{thm:local-lp}
Assume \emph{[\textsc{Train}]} and \(z\) is bounded in \(\Cplus\) with with \(\eta = \mathrm{Im}z \ge n^{-1+\epsilon}\) for some positive \(\epsilon\). Then, we have the following analogue of \eqref{eq:avgd-local-mp} from Theorem~\ref{thm:local-mp}:
\begin{equation} \label{eq:theta-lim}
    \thetaonenz = \uthetaonez + \oOPof{\frac{1}{n\eta}}
\end{equation}
as \(n,p\to\infty\). Furthermore (the ``small-scale laws,'') we have the following analogue of \eqref{eq:small-scale-mp} from Theorem~\ref{thm:local-mp}:
    \begin{equation} \label{eq:small-scale-lp}
       \sup_{(a,b)\subset \mathbb{R}} \left|\lpmeasn(a,b) - \ulpmeas(a,b)\right| \prec \frac{1}{n}
    \end{equation}
        and, uniformly in \(f\in C^2\),
    \begin{equation} \label{eq:portmanteau-lp}
        \int f\, d\nu_n = \int f\, d\ulpmeas + O_\prec\left(\frac{\lv f\rv_1}{n} + \frac{\lv f'\rv_1}{n} + \frac{\lv f''\rv_1}{n^2} \right),
    \end{equation}
    as \(n,p\to\infty\).
\end{thm}
\begin{proof}[Proof Sketch] Equation~\eqref{eq:theta-lim} can be obtained using the same general techniques used to establish \eqref{eq:avgd-local-mp} in \cite{knowles2017anisotropic}.  For details, see \cite{latimer2023local} or the more recent and general work of \cite{ding2024eigenvector}.  The small-scale laws of Equations~\eqref{eq:small-scale-lp} and \eqref{eq:portmanteau-lp} follow from the arguments in Section~S.2
of the Supplementary Material.
\end{proof}

\subsection{\LW{} Approximation} \label{sec:lw-theory}

In this subsection, we present the approximations devised by Ledoit and Wolf \cite{ledoit2017direct} for the \MP{} density \(\myw(x)\) and its Hilbert transform.  Upon substitution in the expression for \(\delta(x)\), these approximations give rise to a shrinkage function \(\deltil(x)\) that exhibits uniform convergence in probability to \(\delta(x)\) on \(\suppumu\).  As discussed, this mode of convergence is sufficient to establish asymptotic optimality in the applications of \cite{ledoit2020analytical,robinson2021space}.

Our main refinement of \cite{ledoit2020analytical} is to develop an explicit rate of convergence for \(\deltil(x)\) to \(\delta(x)\), although  we study convergence in \(L^1(d\mun)\) instead of in \(L^{\infty}(d\umu)\).  With this essentially more relaxed convergence mode, we are able to obtain convergence at a rate of \(O_\prec(n^{-2/3})\), which will be fast enough 
to enable standardization and optimization of the \PHT{}s that will be our focus starting in the next section.

We now describe the main approximation result of \cite{ledoit2020analytical}. As mentioned before, the idea is to replace \(\myw(x)\) and \(\htrans\myw(x)\)  in \(\delta(x)\) by suitable substitutes.  First, one replaces \(\myw(x)\) everywhere by
\begin{equation} \label{eq:mywt}
\mywtnx := p^{-1}\sum_{i=1}^p k_{ni}(x),
\end{equation}
where
\begin{equation*} \label{eq:myk}
k_{ni}(x) := \dfrac{1}{n^{-1/3}\lamni} k\left(\dfrac{x-\lamni}{n^{-1/3}\lamni}\right)
\end{equation*}
for a suitable bump-like kernel function \(k\) with known Hilbert transform \(K\).  Second, they replace \(\htrans\myw(x)\) in \(\delta(x)\) by 
\begin{equation} \label{eq:myHwt}
\Hwtnx := p^{-1}\sum_{i=1}^p K_{ni}(x),
\end{equation}
where
\begin{equation*} \label{eq:myK}
K_{ni}(x) := \dfrac{1}{n^{-1/3}\lamni} K\left(\dfrac{x-\lamni}{n^{-1/3}\lamni}\right).
\end{equation*}
using invariance properties of \(\htrans\).  

The main theoretical result of \cite{ledoit2020analytical} is then the following, which states that if \(k\) is chosen appropriately 
and the above replacements are made in \(\delta(x)\), 
the resulting function consistently estimates \(\delta(x)\) in probability.
 \begin{thm}[Theorem~4.1 of \cite{ledoit2020analytical}] \label{thm:lw} 
     Assume \emph{[\textsc{Train}]} and \(k(x) = \frac{1}{2\pi}{\sqrt{(4-x^2)_+}}\) and \(K(x) = (2\pi)^{-1}(-x+\operatorname{sign}(x)\sqrt{(x^2-4)_+})\), where \((\mydot)_+=\max\{\mydot,0\}\).  Then, with \(\mywt_n(x)\) as in \eqref{eq:mywt} and \(\Hwt_n(x)\) as in \eqref{eq:myHwt}, the function
     \begin{equation} \label{eq:lweig}
     \lwEigX = \frac{x}{[1-\frac{p}{n}-\frac{p}{n}\pi x\Hwtnx]^2 + \frac{p^2}{n^2}\pi^2 x^2 \mywtnx^2}
     \end{equation}
     has the property that
     \[
     \sup_{x\in \suppumu} \lvert\lwEigX - \delta(x)\rvert = o_P(1)
     \]
     as \(n,p\to\infty\).
 \end{thm} 

 \begin{rem} 
     The only significant departure of our statement from \cite[Theorem~4.1]{ledoit2020analytical} is that Ledoit and Wolf actually choose \(k(x)\) to be the so-called \textit{Epanechnikov} kernel \(3(1-x^2/5)_+/(4\sqrt{5})\) in the main body of their paper, but remark in \cite[Supplement~D.5]{ledoit2020analytical} that the semicircular kernel above works equally well in theory.  We choose the latter since the Epanechnikov kernel appears to require much more care in its implementation to avoid numerical instabilities.
     As a reminder of this and other notation introduced in this section, Table~\ref{tab:Shrink} summarizes key symbols related to nonlinear shrinkage that will be used throughout the paper.
 \end{rem}

As indicated above, we do not analyze rates of stochastic convergence in the rather strong \(L^\infty(\umu)\) norm that appears in the above theorem since \(L^1(\mun)\) convergence will suffice for our purposes.  The following theorem provides an explicit rate of \(O_\prec(n^{-2/3})\) for this type of convergence---a dramatic improvement on the \(o_P(1)\) rate in \(L^1(\mun)\) that follows immediately from Theorem~\ref{thm:lw}---and is the culmination of the work presented so far. 
 
 \begin{thm}[Error Analysis of \LW{} Eigenvalues in \(L^1(\mun)\)] \label{thm:local-lw}
Assume \emph{[\textsc{Train}]}.  Then we have
\begin{equation*}
    p^{-1}\sum_{i=1}^p \left|\lwEig(\lamni) - \delta(\lamni)\right| \prec \frac{1}{n^{2/3}}
\end{equation*}
as \(n,p\to\infty\).
 \end{thm}
 \begin{proof}
     See Section~S.3
     of the Supplemental Material.
 \end{proof}

With this result in hand, we may now return to our application of \blue{anomaly} detection.

\vspace{.25cm}
\begin{minipage}{\linewidth}
\centering
\setlength{\tabcolsep}{6pt}
\renewcommand{\arraystretch}{1.1}
\captionof{table}{Nonlinear shrinkage notation\label{tab:Shrink}}
\begin{tabularx}{\linewidth}{>{\raggedright\arraybackslash}p{0.35\linewidth}X}
\toprule
\textbf{Symbol} & \textbf{Meaning / definition} \\
\midrule
$f_n(\mydot)$ & Bounded spectral shrinkage function, possibly random \\
$\shrinker = \sum_{i=1}^p f_n(\lambda_{ni})\bu_{ni}\bu_{ni}\tps$ & Precision-shrinkage estimator ($p\times p$) \\
$\delta(x)\  (\ne \text{ the measure } \delta_x)$ & \eqref{eq:delta}, Deterministic limit from \lp{} \\
$\tilde d_n(x)$ & \eqref{eq:lweig}, \LW{} observable approximation of $\delta(x)$ \\
$k(x)$ & Semicircular kernel \((2\pi)^{-1}\sqrt{(4-x^2)_+}\) \\
$K(x)$ & \(\htrans k(x) = (2\pi)^{-1}(-x+\operatorname{sign}(x)\sqrt{(x^2-4)_+})\) \\
$\lpmeasn$ & \lp{} measure $p^{-1}\sum_{i=1}^p \lpqfnilong \delta_{\lamni}$ \\
$\Theta_n(z)$ & Stieltjes transform of \(d\lpmeasn\) \\
$\ulpmeas(a,b)$ & \(\int_a^b \delta(x)\, d\umu(x)\), limit of \(\lpmeasn(a,b)\)  \\ 
$\uTheta(z)$ & Stieltjes transform of \(d\ulpmeas\) and limit of \(\Theta_n(z)\)\\
\bottomrule
\end{tabularx}
\end{minipage}
\vspace{.25cm}

\section{Theoretical contributions}   \label{sec:EDs}

In order to discuss the asymptotic properties of our test statistic \eqref{eq:eed}, we must now introduce a model for our test datum \(\by=\byn\).  This random vector must, like the \(\bxnind\)'s, of course be \(p\times 1\) and increasing in dimension as \(n\) increases. We assume that \(\byn = \bmuyn + \bfrhalfn \bzn\) for some unknown  (possibly random) vector \(\bmuyn\), where the components of \(\bzn\) are i.i.d$.$ \(\sim W\) and independent of \(\bxnind\). In other words, we assume \(\byn - \bmuyn\) and \(\{\bxnind\}_{i=1}^n\) are jointly i.i.d.

In Subsection~\ref{sec:standardization}, we control asymptotic significance levels for the \PHT{} (Theorem~\ref{thm:significance}).
Then, in Subsection~\ref{sec:detection-power}, we find a surrogate functional for detection performance (Theorem~\ref{thm:deterministic-limit}(a)) and identify an (unobservable) shrinkage function that optimizes it in the limit (Theorem~\ref{thm:deterministic-limit}(b)).  Finally, we present a practical empirical shrinkage function that approximates the deterministic optimum (Theorem~\ref{thm:approx-fopt}).  Throughout this section we will need the following notion of a regular sequence of functions.
    \begin{definition}
    Suppose \(f_n \in C^2(\bbR)\) be a sequence for which \(|f_n|\) is uniformly bounded for large enough \(n\),  \(|f'_n|\prec n^{1/3}\),  \(|f''_n|\prec n^{2/3}\), and \(\lv \bfrn \shrinker\rv_{HS}^2 \asymp p\) in probability.  Then we say \(f_n\) is \textit{regular}.
\end{definition}

\subsection{Significance level}
\label{sec:standardization}

In this subsection, we fix a regular sequence \(f_n\) and determine a valid asymptotic significance level for the statistic \(\myTsquared = \myTsquared(f_n)\), defined in \eqref{eq:eed}.  Suppose we shift and rescale by empirical quantities \(\tmeann p \) and \(\tsign \sqrt{p}\), where \(\tmeann,\tsign \) are order-one in probability, obtaining
    \begin{equation} \label{eq:Zon}
    \Zanalyticn := \frac{\myTsquared - \tmeann p}{\tsign \sqrt{p}}.
    \end{equation}
    Expanding the quadratic form \(\myTsquared\), it can be shown that the mean and variance of \(\myTsquared\) are both almost surely of order \(p\) and that the quantity \(\ox_n\) can be  replaced by \(\bmun\) without affecting the limiting distribution of \(\Zanalyticn\).  Thus, if we define the following unobservable quantity
        \begin{equation*} \label{eq:Ztildiff}
    \Zanalytic^o_n := \frac{\bz_n\tps\tilde{\bfr}_n\bz_n - \tmeann p}{\tsign\sqrt{p}},
    \end{equation*}
    then \(\Zanalyticn - \Zanalytic^o_n\) converges in distribution to \(\delta_0\),
    where \(\tilde{\bfr}_n = \bfrhalfn \shrinker \bfrhalfn\). 
    
    Let \(\Zorign\) be the additional unobservable quantity defined as follows
    \begin{equation} \label{eq:zorig}
   \Zorign := \frac{\bz_n\tps\tilde{\bfr}_n\bz_n - \meann p}{\sign\sqrt{p}},
    \end{equation}
    where 
    \begin{equation*}
        \meann \equiv \meann(f_n) := p^{-1}\sum_{i=1}^p f_n(\lamnind) \unind\tps \bfrn \unind
    \end{equation*} and 
    \begin{equation*}
    (\sign)^2 \equiv \sign(f_n)^2  := p^{-1} \tr(f_n(\bSn) \bfrn f_n(\bSn) \bfrn).
    \end{equation*}
    In the special case where the fourth moment of \(W\) matches a standard normal's, this statistic is asymptotically normal by conditioning on \(\bXn\), applying Berry--Esseen \cite{berry1941accuracy,esseen1942liapunoff} , and taking the expectation over the distribution of \(\bXn\). Comparing \(\Zorign\), \(\Zanalytic^o_n\), and \(\Zanalyticn\), this special case thus motivates choosing  \(\tmeann\) and \(\tsign\) to consistently approximate \(\meann\) and \(\sign\) in probability.  Suitable approximations will be given in Lemma~\ref{thm:empirical-standardization}, but first we need an intermediate approximation of \((\sign)^2\), in terms of a deterministic quadratic functional that we denote \(\sigma_\infty^2\).

\begin{lem} \label{lem:det-lim-sigma}
    Assume \emph{[\textsc{Train}]} and let \(\tilhilb:= L^2(\suppumu)\). Let \(\mysigma^2(f)\) be a quadratic functional on \(\tilhilb  \) defined by 
    \begin{equation} \label{eq:sigmainfty}
        \mysigma^2(f) = \int_{\suppumu} [\myGamma f(x)]^2 x\delta(x)\, d\umu(x),
    \end{equation}
    where \(\myGamma: \tilhilb \to \tilhilb\) is defined by
    \[
    \myGamma f(x) = f - \asp \pi \left(\htrans[ f \delta w] - \htrans[\delta w]f\right). \qquad (x\in \suppumu)
        \]
    Then, if \(f_n\) is regular we have
    \begin{equation*}
        \sign(f_n) = \mysigma(f_n) + \oas(1)
    \end{equation*}
    as \(n,p\to\infty\).
\end{lem}
\begin{proof}
    See Section~S.4
    of the Supplemental Material.
\end{proof}

With these formulas in mind, we now present
our choices of \(\tsign\) and \(\tmeann\), as well as the resulting asymptotic equivalence of \(\Zanalyticn\) and \(\Zorign\).
\begin{lem} \label{thm:empirical-standardization}
    Assume \emph{[\textsc{Train}]} and that \(f_n\) is regular. Let
    \[
     \tmeann \equiv \tmeann(f_n) := p^{-1}\sum_{i=1}^p f_n(\lamnind) \lwEig(\lamni)
    \]
    and
       \begin{equation*}
      \myGamman f(x) :=  f(x)-\frac{1}{n}\sum_{j=1}^p(f(\lambda_{nj})-f(x))\lwEig(\lambda_{nj}) K_{nj}(x)
   \end{equation*}
   and
   \begin{align} \label{eq:sigtiln}
   \tsign^2 \equiv \tsign^2(f_n) :=  p^{-1}\sum_{i=1}^p  \left[ \myGamman f_n(\lamnind)\right]^2 \lamni \lwEig(\lamni) .
   \end{align}
Then, with \(\Zanalyticn\) and \(\Zorign\) as defined in \eqref{eq:Zon} and \eqref{eq:zorig},
    \(\Zanalyticn - \Zorign\)
    converges in distribution to \(\delta_0\) as \(n,p\to\infty\).
\end{lem}
\begin{proof}
It follows from Theorems~\ref{thm:local-lp} and \ref{thm:local-lw} together is that we have
\[
\frac{|\meann(f_n) - \tmeann|p}{\sqrt{p}} = \sqrt{p} |\meann(f_n) - \tmeann| \prec \sqrt{p}\left(\frac{\lv f_n\rv_1}{n} + \frac{\lv f_n'\rv_1}{n} + n^{-4/3} \right)\prec \frac{\sqrt{p}}{n^{2/3}} \prec \frac{1}{n^{1/6}}
\]
as \(n\to\infty\).  Thus, as long as \(\tsign\) is a consistent approximation of \(\sign(f_n)\) our choice of \(\tmeann\) is suitable. 
The consistency of \(\tsign\) and \(\sign(f_n)\) follows from the formal similarity of \eqref{eq:sigtiln} and \eqref{eq:sigmainfty}, the fact that \(\sign(f_n)^2\asymp p^{-1}\lv \bfrn \shrinker\rv_{HS}^2\asymp 1\), and the analysis of Section~S.5
of the Supplemental Material.
\end{proof}

It remains to calculate the asymptotic significance level of  \(\Zorign\). First, we note that the null distribution of \(\Zorign\) converges conditionally almost surely (and thus, unconditionally) to a mean-zero univariate Gaussian by Berry--Esseen again and the boundedness assumption on \(f_n\).  (See Figure~\ref{fig:standard}.) However, the variance of this Gaussian depends on both \(f_n\) and the \textit{excess kurtosis} of \(W\), defined as \(\bbE[W^4]-3\).  

\begin{rem} \label{rem:R2Maj4}
We note that the asymptotic normality of $\tilde{T}_n^2(f)=\by_n\tps f(\bS_n)\by_n$, for differentiable $f$ can be considered as a parallel result to \cite[Theorem~2]{pan2011central}, which establishes the asymptotic normality of the closely related one-sample statistic \(\overline{\bx}_n\tps f(\bSn) \overline{\bx}_n / \Vert\overline{\bx}_n\Vert^2\) in the special case that  $\bfr_n=\mathbf{I}_p$ (and $f$ is analytic). 
Making this assumption on $\bfr_n$, the results differ primarily in that the one-sample statistic is a quadratic form in the sample mean, normalized by a sample mean factor, whereas $\tilde{T}^2_n(f)$ is a quadratic form in an un-normalized singleton observation. Because  the one-sample statistic evaluates a sample mean, the contribution of the kurtosis of $W$ vanishes asymptotically. 

\ Consequently, to compare the two results, we look at the kurtosis-independent terms of the asymptotic variances. The  variance of the one-sample statistic is shown in \cite{pan2011central} to approach $(2/\asp)(\int f^2\, d\umu - (\int f\,d\umu)^2)$, whereas the kurtosis-independent term of $\mathrm{var}(\tilde{T}_n^2(f))$ in \eqref{eq:null-variance} is structurally similar, though scaled differently and lacking the $f$-dependent corrective term: $2\Vert\tilde{\bfr}_n\Vert_{HS}^2 = 2\mathrm{tr}(f(\bS_n)^2)$, which is asymptotic to $2p\int f^2\, d\umu$.
\end{rem}

Even without an exact estimate of size, we can obtain an asymptotic significance level that is independent of \(f_n\) using sub-Gaussian concentration bounds as follows.
The Hanson--Wright inequality \cite{rudelson2013hanson} implies that, almost surely for \(\tau > 0\),
\begin{equation} \label{eq:HWnull}
    \begin{split}
         \Pr\left[ \lvert \bzn\tps \shrinker\bzn -\meann p \rvert > \mythresh \mid \bXn, \fH_0 \right]
        \le 2\exp\left[-c\min\left(\frac{\mythresh^2}{C^4(\sign)^2 p}, \frac{\mythresh}{C^2\lv \sigtil \rv} \right)\right]
    \end{split}
\end{equation}
for a known absolute constant \(c > 0\) and some constant \(C\) depending only on the moments of \(\myW\). Scaling \(\mythresh\) by \(\sign \sqrt{p}\), this upper bound can naturally be used to set a nontrivial asymptotic significance level, since the second term in the minimum then goes to infinity like \(\sqrt{p}\), almost surely. By the definition of \(\Zorign\), then, we have:
\begin{equation*} \label{eq:significance}
\Pr\left[ \Zorign  > \mythresh \mid \bXn, \fH_0 \right] \le 2\exp(-c\mythresh^2/C^4) + \oas(1).
\end{equation*}
Taking the expectation with respect to the distribution of \(\bXn\) and using Lemma~\ref{thm:empirical-standardization} gives the following main result for this subsection.
\begin{thm} \label{thm:significance}
    Assume \emph{[\textsc{Train}]}, that \(f_n\) is regular, and that \(c\) and \(C\) are the constants defined above.  Then \(\Zanalyticn\) is asymptotically normal with zero mean, and
    \begin{equation*}
        \Pr\left[ \Zanalyticn  > \mythresh \mid \fH_0 \right] \le 2\exp(-c\mythresh^2/C^4) + o(1)
    \end{equation*}
    as \(n,p\to\infty\).
\end{thm}

\blue{
The asymptotic significance level in Theorem~\ref{thm:significance}  holds for any component-wise sub-Gaussian tail behavior, and can be used whenever one can find an upper bound on the sub-Gaussian constant $C^4$.  If, on the other hand, one cannot bound $C^4$, a natural question is whether it is still possible to control the asymptotic size.  To that end, observe that the exact variance of the unscaled quadratic form $\bz_n\tps\tilde{\bfr}_n\bz_n$ is given by
\begin{equation*} \label{eq:null-variance}
 \exkurt \sum_{i=1}^p (\tilde{\bfr}_n)_{ii}^2 + 2\left\Vert\tilde{\bfr}_n\right\Vert_{HS}^2,
\end{equation*}
where $\exkurt$ is the excess kurtosis, and the matrix elements $(\tilde{\bfr}_n)_{ij}$ are evaluated with respect to the basis in which noise components are independent.
Because the summation multiplying $\exkurt$ depends on the unknown basis of independence, it is generally unobservable. However, one can construct an asymptotic upper bound for the variance in the inset
above
by bounding the summation and estimating $\exkurt$. 
Fortunately, a procedure for consistently estimating $\exkurt$ in a general high-dimensional independent-components model can be found in \cite{lopes2019bootstrapping}. We provide a derivation of the resulting data-driven asymptotic significance level in Section~S.6
of the Supplementary Material.
}

\begin{figure}[htbp]
    \centering
    \begin{minipage}[b]{0.48\columnwidth}
        \centering
        \includegraphics[width=\linewidth]{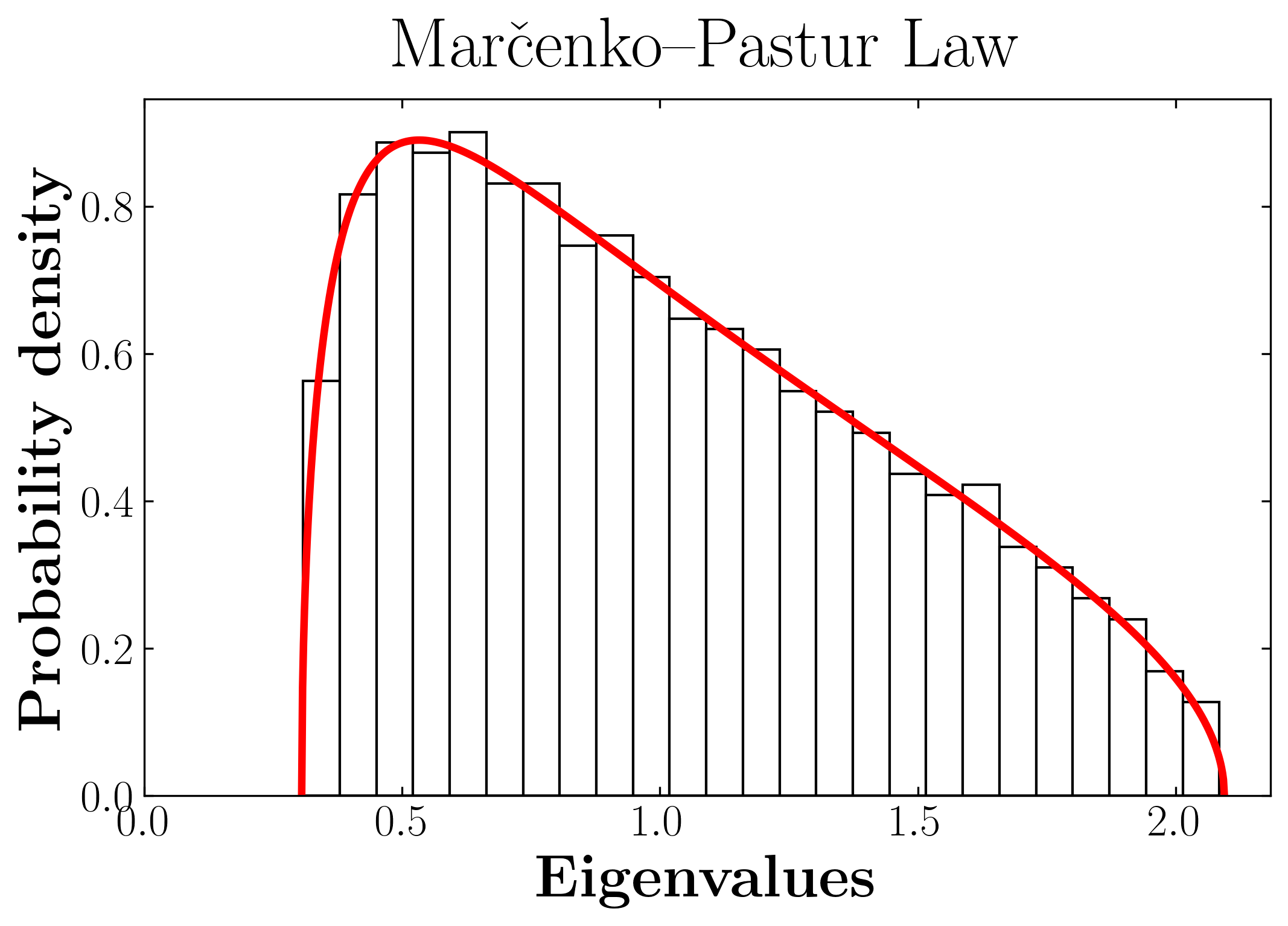}
        \caption{\MP{} density \(\mywx = d\umu(x)/dx\) (red) for \(p/n = 1/5\) and \(\upopsm = \delta_1\) versus histogram of eigenvalues of $\bSn$ for \(n=5000\), \(p=1000\), and \(\bfrn = \bI_p\), and Gaussian data. Due the error bounds of roughly \(1/n\) in the small-scale law of Theorem~\ref{thm:local-mp}, a fairly high-resolution histogram matches the well-known density of \(\mywx\) reasonably well.}
        \label{fig:mp}
    \end{minipage}\hfill
    \begin{minipage}[b]{0.48\columnwidth}
        \centering
        \includegraphics[width=\linewidth]{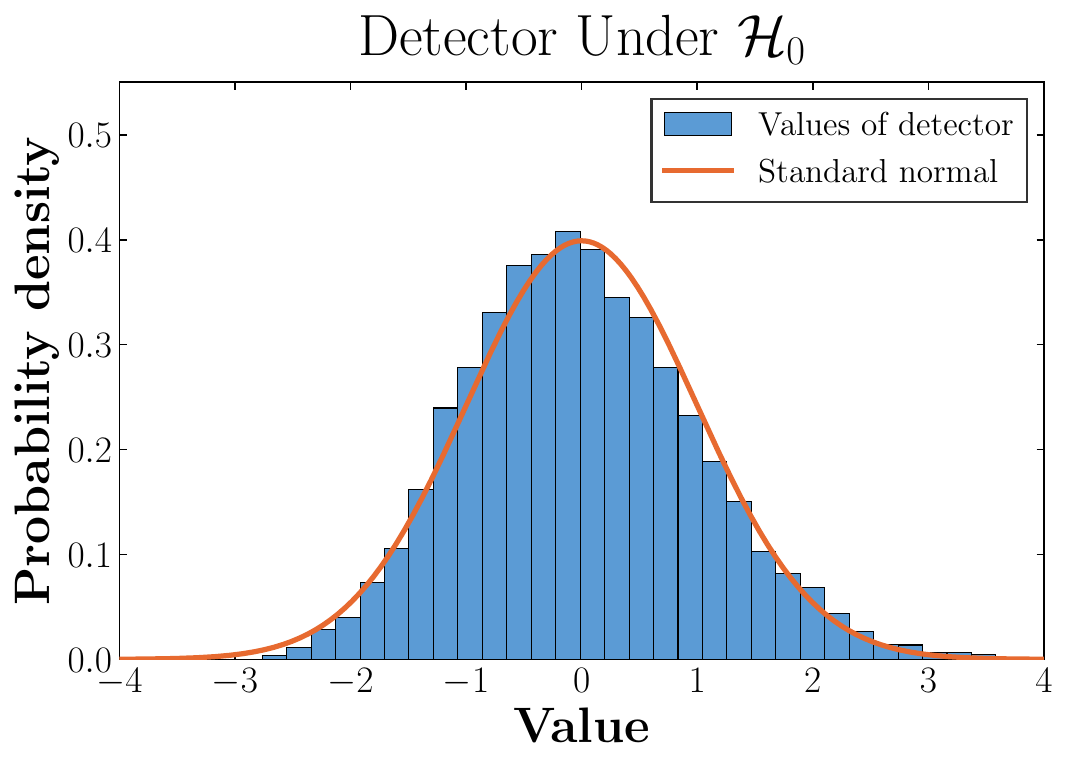}
        \caption{A plot showing \(\Zanalyticn\) of \eqref{eq:Zon} to be roughly standard normal when both the test datum \(\by\) and the reference sample \(\{\bx_i\}_{i=1}^{400}\) are a coloring matrix times a vector of standard Gaussian random variables. Here, the dimension is \(200\); $\bfr_n$ is the sum of a matrix with \(U(0,1)\) eigenvalues and rank-40 matrix with eigenvalues \(10^{(40-j)/10}\) for \(j=0, ... 39\); and $f_n(x)\equiv 1/(x+1)$.}
        \label{fig:standard}
    \end{minipage}
\end{figure}

    
    
    

\subsection{Detection power and optimal shrinkage} \label{sec:detection-power}

Assume again that \(f_n\) is a regular sequence.  Further, assume now that \(\fH_1\) is in force and consider the detection rate.  
By expanding quadratic forms, the asymptotic variance of \(\myTsquared\) conditioned on \(\bXn\) and \(\bmun\) is given by
\begin{align} \label{eq:alt-var}
& \exkurt\sum_i (\tilde{\bfr}_n)_{ii}^2 + 2\left\Vert \tilde{\bfr}_n\right\Vert_{HS}^2 + 4\bmun\tps \shrinker \bfr\shrinker \bmun.
\end{align}
Since large values of \(\lv \bmun\rv_2\) would trivialize the hypothesis test, we assume for now that \(\lv \bmun \rv_2 = o_{P}(\sqrt{p})\), which allows us to ignore the term containing \(\bmun\) above.  Thus, \(\myTsquared\) has the same conditional asymptotic variance under \(\fH_1\) and \(\fH_0\).  Arguing similarly, define
\begin{equation} \label{eq:fUn}
\fUn \equiv \fUn(f_n):=\frac{\bmun\tps \shrinker\bmun}{\tsign \sqrt{p}},
\end{equation}
where \(\tsign = \tsign(f_n)\), as defined in Lemma~\ref{thm:empirical-standardization}.
The difference between \(\Zanalyticn\) under \(\fH_1\) and \(\fH_0\) is given by
\begin{equation*} 
\frac{2\bzn\tps\bfrhalfn \shrinker\bmun}{\tsign \sqrt{p}}+\fUn,
\end{equation*}
and the linear term in \(\bz_n\) can be neglected due to the asymptotic assumption on \(\lv \bmun\rv_2\) and another application of Hanson--Wright.  
Thus, \(\Zanalyticn-\fUn\) under \(\fH_1\) and \(\Zanalyticn\) under \(\fH_0\) almost surely have the same asymptotic distribution, conditioned on \(\bXn\) and \(\bmun\).  

Based on our analysis so far, if we consider only \(f_n\) for which the conditional null variance of \(\Zanalyticn\mid\bXn\) is asymptotically constant, the quantity \(\fUn(f_n)\) would be the detection criterion of the test---the quantity to maximize in order to maximize power.  In actuality, however, this variance may depend on \(f_n\), so identifying a tractable detection criterion is a challenge.  Still, \(\fUn\) can be seen as a reasonable surrogate detection criterion since Theorem~\ref{thm:significance} implies the variance of the limiting null distribution of \(\Zanalyticn\) is bounded independently of \(f_n\), in terms of \(c\) and \(C\) alone.  In other words, \(\fUn(f_n)\) controls the worst-case detection power over all sub-Gaussian tail behaviors, and choosing \(f_n\) to maximize it will be the main task of the remainder of this paper.

In addition to the qualitative argument above, we may quantitatively estimate this worst-case detection power by using Hanson--Wright again:
\begin{equation} \label{eq:pd-unobservable}
\begin{split}
    &  \Pr\left[ \Zanalyticn > \mythresh \mid \bXn, \bmun, \fH_1 \right]  
   \\
   & \ge 1- \Pr\left[ \lvert\fUn - \Zanalyticn \rvert \ge (\fUn - \mythresh)_+ \mid \bXn, \bmun, \fH_1 \right]
   \\
   & \ge 1- 2\exp\left(-c(\fUn -\mythresh)_+^2/C^4\right) + \oas(1).
\end{split}
\end{equation} 
Moreover, taking the expectation over the distributions of \(\bmun\) and \(\bXn\) gives
\begin{equation}
    \label{eq:pd-finalfinal}
    \Pr\left[ \Zanalyticn > \mythresh \mid \fH_1 \right] \ge 1- 2\bbE\exp\left(-c(\fUn -\mythresh)_+^2/C^4\right) + o(1).
\end{equation}
This result shows the expected monotonicity in \(\fUn\) for any fixed \(\tau\).  We add that in the special case that \(W\) is Gaussian, power is directly monotonic in \(\fUn\)---a fact that follows, as before, from Berry--Esseen.

Since the direction in \(\bbR^p\) of the signal \(\bmun\) is unknown, it is not possible to analyze the detection criterion
\(\fUn\) without additional information. 
To address this fact, one often assumes some sort of prior on \(\bmun\).  For example, in \cite{li2020adaptable}, the authors make the \textit{maximum-entropy} assumption \(\bmun \sim \mathcal{N}(0,\myOmega_n)\) for some specified dispersion matrix \(\myOmega_n\).  In particular, for convenience they assume \(\myOmega_n\) is some polynomial in \(\bfrn\), scaled by \(n^{1/2}\) to ensure that 
\(\lv\bmun\rv^2 \asymp n^{1/2}\) with high probability, which means 
\(\fUn\asymp 1\) with high probability, which means the hypothesis test is tractable but not trivially so in the limit.  Examples include \(p^{1/2}\mathbf{I}_p\) (ignorant prior) or \(p^{1/2}\bfrn\) (covariance-matched prior). In this paper, we allow for more general dispersion matrices, as follows.  
\begin{assumptiontest*} 
Let \(\bmun\) be a mean-zero random vector independent of \(\bzn\) for which \(\bbE[\byn\mid\bmun]=\bmun\) (a.s.), and  
let \(\bmun\) have the (maximum-entropy) distribution \( \mathcal{N}(0,\myOmega_n)\) for some known covariance \(\myOmega_n\) with \(\lv \myOmega_n \rv \asymp n^{1/2}\).  Further, assume there is a measure \(d\omegainfty(x) = h(x)\, dx \equiv \overline{h}(x)w(x)\, dx\) such that \(\overline{h}\) is \(C^2\) on \(F\), and assume that
for any regular shrinker \(f_n\) we have
\begin{equation} \label{eq:int-omega}
    p^{-3/2}\tr\left(\myOmega_n f_n(\bSn) \right) = \int f_n(x) d\omegainfty(x) + o_P(1).
\end{equation}
\end{assumptiontest*}
\noindent 
We note that the reasons the examples of \(\myOmega_n\) equal to \(p^{1/2}\mathbf{I}_p\) and \(p^{1/2}\bfrn\) satisfy \eqref{eq:int-omega} are the last small-scale laws in Theorems~\ref{thm:local-mp} and \ref{thm:local-lp}. We also note that the scaling factor of \(p^{1/2}\) is not necessary for our proposed shrinker to be asymptotically optimal but simply guarantees convergence of the test's power to a number in \((0,1)\): if a divergent scale is chosen, all relevant tests' power will converge to zero or one. Thus, our test will remain asymptotically optimal even in these trivial power regimes.

In the following theorem (Theorem~\ref{thm:deterministic-limit}), we will show that under this new assumption [\textsc{Test}], the detection criterion \(\fUn( f_n)\), like the variance \(\sign(f_n)\) from the previous subsection, can be approximated using a deterministic limit.  
Further, this deterministic limit can be optimized by solving a variational problem explicitly in terms of Hilbert transforms. In a moment, this solution  will  inform the approximate solution presented in  Theorem~\ref{thm:approx-fopt}.  

Recalling the definitions of \(g\) and \(G\) from \eqref{eq:gG}, our claims about the deterministic limit of \(\fUn(f_n)\) are as follows.

\begin{thm} \label{thm:deterministic-limit}
    Assume \emph{[\textsc{Train}]} and \emph{[\textsc{Test}]}.   
    Then 
    \begin{itemize}
        \item[(a)]  (Deterministic Limit.) Suppose \(f_n\) is regular. Then
    \[
     \fUn(f_n) = \fU(f_n) + \oas(1), 
    \]
    as \(n, p\to\infty\), where
    \[
    \fU(f) = \frac{\int f\, d\omegainfty}{\mysigma(f)}.
    \]

\item[(b)] (Optimal Limiting Shrinkage.)
    If \(H(x)\) is the Hilbert transform \(\htrans\omegafcn(x)\) and \(a(x) =  x\delta(x) w(x)\), then \(f \in L^2(\suppumu) \mapsto \fU(f)\) is maximized 
    by the function 
    \begin{equation} \label{eq:myfstar}
        f^* := \frac{\myg^2\omegafcn + \myg \myG H}{a} - \htrans\left[ \frac{\myG^2 H + \myG\myg \omegafcn}{a}\right].
    \end{equation}
    
    \end{itemize}
\end{thm}
\begin{proof}
    For part (a) we argue as follows.  
    Since Lemma~\ref{lem:det-lim-sigma} proves \(\tsign(f_n) = \mysigma(f_n) + \oas(1)\), it suffices to show that 
\begin{equation}
 \label{eq:numomega}
 p^{-1/2}\bmun\tps f_n(\bSn) \bmun = \int f_n\, d\omega_\infty + \oas(1).
\end{equation}
By Hanson--Wright, the left side is almost-surely equivalent to
\[
p^{-3/2}\tr (f_n(\bSn) \myOmega_n).
\]
Thus, part (a) follows from Assumption [\textsc{Test}].
    
    See Section~S.7
    of the Supplemental Material for part (b).
\end{proof}
\newcommand{\myhn}{\overline{h}_n}
\newcommand{\mygn}{g_n}
\newcommand{\myGn}{\overline{G}_n}

\begin{rem} \label{rem:sample-starved}
While the above theorem, like the majority of this paper, restricts to the case where $n$ eventually exceeds $p$, a regime of significant practical interest is the \textit{singular} case where $p > n$. In this case, the spectrum of $\bS_n$ is degenerate, having $p-n$ eigenvalues equal to zero, making finding a replacement for $\bfr_n^{-1}$ especially difficult. At a high level, this case requires applying a shrinkage function to the nonzero eigenvalues of $\bS_n$ and assigning a common positive number to every zero eigenvalue.  In practice, these two tasks turn out to be a significant challenge in the singular regime, as they are tightly coupled.  Nevertheless, we provide an outline for their solutions in Section~S.10
of the Supplemental Material, 
along with \texttt{python} code and some supporting simulations. 
\end{rem}

The main thing that remains is constructing a suitable approximation \(f_n\) to the optimizer \(\fopt\) of \(\fU\), which is accomplished by the following theorem.  We note that suitable choices for \(\myhn\) in the upcoming equation \eqref{eq:h-cond} are \(\lwEig\) if \(\myOmega_n = \bfrn\) (by Theorem~\ref{thm:lw}) and \(\myhn\equiv 1\) if \(\myOmega_n = \mathbf{I}_p\) (by the \MP{} Theorem).  In addition, the work of \cite{ledoit2011eigenvectors, ledoit2020analytical} can be used to show similar formulas hold if \(\myOmega_n\) is a polynomial in \(\bfrn\), though we do not pursue this direction for ease of exposition.

\begin{thm}
    \label{thm:approx-fopt}
    Assume \emph{[\textsc{Train}]} and \emph{[\textsc{Test}]}.  Suppose there are real-valued regular functions \(\myhn\) such that the following holds:
    \begin{equation} \label{eq:h-cond}
    p^{-1}\sum_{j=1}^p \left| \myhn(\lambda_{nj})-\overline{h}(\lambda_{nj})\right| = o_P(1)
    \end{equation}
    as \(n,p\to\infty\).
    Let \(\nasp = p/n\), let
    \[
   H_n(x) = p^{-1}\sum_{i=1}^p \overline{h}_n(\lamnind) K_{ni}(x-\lamnind),
    \] 
    and let
    \[
    \mygn(x) = 1-\nasp
    -\nasp\pi x\Hwtn(x) \qquad \text{and}\qquad \myGn(x) = -\nasp\pi x.
    \]
    Further, let
    \[
    \xi_n(x) = \frac{\mygn(x)^2\myhn(x)+\mygn(x)\myGn(x) H_n(x)}{\lwEig(x) x}
    \]
    for \(x > 0\), and
    \[
    \eta_{nj} = \frac{\myGtnx^2 H_n(\lambda_{nj}) + \myGtnx \mygtnx  \myhnj}{\lwEig(\lambda_{nj}) \lambda_{nj}}
    \]
    for \(j=1,2,\dots, p\). Finally, let \(f_n(x)\) be given by
    \begin{equation*}
       \xi_n(x) - p^{-1}\sum_{j=1}^p \eta_{nj} K_{nj}(x-\lambda_{nj}).
    \end{equation*}
    Then \(f_n\) is regular and satisfies
    \begin{equation*}
	\mathcal{U}_n(f_n) = \fU(f^*)  + o_P(1).
    \end{equation*}
   As a result, the lower bound on detection power in \eqref{eq:pd-finalfinal} is,  asymptotically in probability, (a) maximized by \(f_n\) and (b) equivalent  to
    \begin{equation} \label{eq:finalpd}
       1- 2\exp\left(-c(\tilde{\mathcal{U}}_n(f_n) -\mythresh)_+^2/C^4\right) + o_P(1),
    \end{equation}
    where
    \[
    \tilde{\mathcal{U}}_n(f) := \frac{p^{-1}\sum_{i=1}^p \overline{h}_n(\lamnind) f(\lamnind)}{\tsign(f) }.
    \]
    \end{thm}
    
\begin{proof}[Proof Sketch]
The vectors \(\xi_{ni}\) and \(\eta_{ni}\) correspond to the terms \((g^2\omegafcn+gG H)/a\) and \((G^2 H+Gg\omegafcn)/a\) of \eqref{eq:myfstar}, evaluated at \(\lambda_{ni}\).   See Section~S.8
of the Supplemental Material for details.  The expression in \eqref{eq:finalpd} arises from \eqref{eq:pd-unobservable} after taking expectations and in-probability limits, then using concentration of quadratic forms to substitute \(\tilde{\mathcal{U}}_n(f_n)\) for \(\fUn(f_n)\). \end{proof}

Note that although the formula above for \(f_n(x)\) can be evaluated at any \(x>0\), we need only evaluate the shrinkage eigenvalues---\textit{i.e.},  \((f_n(\lamnind))_{i=1}^p\).   Since each summation in \(f_n(x)\) naively takes \(O(p)\) time to evaluate, computing these \(p\) values takes \(O(p^2)\) time, which is insignificant compared to the \(O(p^3)\) time generally required to compute the eigenvectors of \(\bSn\).  If necessary, the desired shrinkage eigenvalues can be calculated even more quickly, albeit using \(O(p^2)\) memory, by computing the matrix \(K_{ni}(\lambda_{nj}-\lamnind)\) once and using fast matrix-vector multiplication to compute the summations appearing, for example, in  \((H_n(\lambda_{nj}))_{j=1}^p\) and \((\lwEig(\lamnind))_{i=1}^p\).  
See Listing~S.1
of the Supplemental Material for a \texttt{python} implementation of this idea.

\begin{rem}[Selection of $\myOmega_n$ and $\overline{h}_n$]
\label{rem:prior}
In realistic settings, exact knowledge of the signal dispersion matrix $\myOmega_n$, or its finite-sample spectral parameters $\overline{h}_n(\lamnind)$ may be unavailable.
  In such cases, it would be desirable to estimate $\overline{h}_n(\lamnind)$ directly from test data.  However, a singleton observation $\by_n$ is reflective only of the average eigenvalue of $\myOmega_n$,  and provides no further information about the eigen-decomposition of $\myOmega_n$.
 This limitation persists even if one makes strong structural assumptions, such as assuming that $\myOmega_n$ is a polynomial of fixed degree in $\bfr_n$ as in \cite[Section~2.3]{li2020adaptable}. A data-driven workaround is only practical when multiple i.i.d$.$ samples $\by_{ni}\sim \by_n$ are available, as we investigate in Section~S.11
 of the Supplemental Material.

As a result, in the regime of this paper where the test sample is a singleton, one must impose additional constraints on the hypothesis test in order to accurately specify the matrix parameter $\myOmega_n$. We provide three examples of how this could be done.
\begin{enumerate}
    \item[1.] A natural constraint is that $\by_n$ should maximize Shannon entropy subject to its squared norm, which by Hanson--Wright is almost surely asymptotic to $\mathrm{tr}(\myOmega_n+\bfr_n)$.  Since this trace approximates $\mathrm{tr}(\myOmega_n)$ by the scaling regime of [\textsc{Test}], we may assume $\lv \by_n\rv^2$ is asymptotic to $\mathrm{tr}(\myOmega_n)$.  Subject to this fixed-power constraint, entropy maximization gives the \textit{isotropic prior} matrix $\myOmega_n \propto \mathbf{I}_p$.  
    \item[2.] A more pessimistic approach is to assume the anomaly is well-masked by the background noise.  One way to ensure this is to assume $\myOmega_n \propto \bfr_n$---a challenging setting in which there is no ``easy'' direction for anomaly detection and which we refer to as the \textit{covariance-matched prior} assumption.  
    \item[3.] A further potential approach would be to assign an uninformative prior to a signal's source location based on geometric constraints.  For example, in distributed sensing 
    one might model a source
    to be uniformly distributed along a line-segment path based on prior knowledge. 
    In this case
    the signal dispersion matrix can be estimated arbitrarily well using simulated samples generated from the positional prior.
\end{enumerate}
\noindent A possible concern is how sensitive priors such as the isotropic and covariance-matched priors are to misspecification.
To address this, we evaluate 
our method's size-adjusted empirical power when our algorithm is forced to assume that the signal dispersion matrix $\myOmega_n$ is $\propto \mathbf{I}_p$ even though the true data-generating matrix is $\propto \bfr_n$ in Figure~\ref{fig:isofalse_subg}.  (Further, some simulations displaying the reverse misspecification can be found in Sections~S.9
and 
S.10
of the Supplemental Appendix.)
As demonstrated in those simulations, our method empirically outperforms other methods when the prior is specified correctly, and remains highly robust even under prior mismatch.

\end{rem}


\section{Empirical results} \label{sec:empirics}

In this section, we evaluate the proposed shrinkage algorithm using both synthetic data and the real-world \textsc{Crawdad} Umich/RSS data set \cite{c7r30h-22}.   
In the synthetic experiments, we choose the population covariance matrix and generate simulated data satisfying the moment conditions in [\textsc{Train}] and [\textsc{Test}].   In the real-world set, a sensor network collects a time series of received signal strengths during periods of activity and inactivity in a lab setting.  In this case, the population covariance matrix and data model are, of course, not known in the real-world set, but the truth information of when people entered and exited the lab is.

In the simulated environment, we find that our method outperforms or nearly matches all tested competitors, which we will describe below---sometimes by significant margins.  For the \textsc{Crawdad} data, we imagine that there are significant dependencies and inhomogeneities between the samples, making the effective number of independent samples different from the number of reference samples and weakening some of the detectors.  Nevertheless, our detector performs as well as the best competitor in our tests.

\subsection{Competing algorithms}
Before we discuss our empirical results further,
we briefly describe each competing method and highlight its limitations relative to the proposed algorithm.

\subsubsection{\bsalg{} and \cqalg{}}

In the context of broader two-sample testing, where two samples have mean \(\ox_1\) and \(\ox_2\),
Bai and Saranadasa \cite{bai1996effect} suggest replacing Hotelling's \(T^2\) with \(\lv \ox_1 - \ox_2\rv^2\), effectively using the trivial shrinkage estimator of \(f(\bSn) = \mathbf{I}_p\), under the ``condition-number-type'' constraint that $\left\Vert \bfrn \right\Vert^2 \ll \tr\left(\bfrn^2\right)$ holds. 

Similar to \bsalg{}, but also suitable to the case of \(p \gg n\), Chen and Qin \cite{chen2010two} propose a test that is asymptotically optimal given a similar constraint on the condition number of \(\bfrn\).  To be precise, the covariance matrix is required to satisfy the condition-number-type constraint $\tr(\bfrn^4) \ll \tr^2(\bfrn^2)$.  Their test statistic for samples \(\{\bx_{1i}\}_{i=1}^{n_1}\) and \(\{\bx_{2i}\}_{i=1}^{n_2}\) is
\[
\frac{\sum_{i\ne j}^{n_1} \bx_{1i}'\bx_{1j}}{n_1(n_1-1)} + \frac{\sum_{i\ne j}^{n_2} \bx_{2i}'\bx_{2j}}{n_2(n_2-1)} - \frac{2\sum_{i=1}^{n_1}\sum_{j=1}^{n_2} \bx_{1i}'\bx_{2j} }{n_1 n_2},
\]
Although this test statistic, \bsalg{}, and several other algorithms we describe are applicable to values of \(n_1\) and \(n_2\) that both differ from 1, we will always assume that \(n_2=1\) for the sake of comparison with ours.
In our simulations, we do not test the ultra-high-dimensional regime, apparently making \cqalg{} perform equivalently to \bsalg{}.  As a result, in the plots that follow, we will show only \cqalg{}, to represent both algorithms.  

\textbf{Limitation}: These algorithms are designed for covariance matrices with condition numbers satisfying growth constraints, and tend to degrade in performance for ill-conditioned ones.

\textbf{Knowledge of $\bfr_n$ and $\myOmega_n$ used}: It is assumed, in our setting, that $\bfr_n$ and $\myOmega_n$ can be replaced in the limit by $c p^{1/2} \mathbf{I}_p$ without any loss of asymptotic power, for some constant $c$.  (For test samples of size order-$n$, the factor of $p^{1/2}$ can be dropped.)

\subsubsection{\lalg{}} 
Li et al. \cite{li2020adaptable} assume that $\bfrn$ is a finite-rank perturbation of the identity matrix---the well-known spiked covariance model of Johnstone \cite{johnstone2001distribution}.  
In  \lalg{}, the ``pooled'' sample covariance matrix $\bSn$ (asymptotically equivalent to our \(\bSn\) when \(n_2=1\)) is replaced by a linear shrinkage estimator of the form $\bSn+b\bI_p$, for some data-dependent \(b > 0\). The scalar \(b\) is chosen to locally maximize an asymptotic detection criterion similar to ours, where the \blue{signal} dispersion matrix \(\myOmega_n\) 
is \(\mathbf{I}_p, \bfrn, \bfrn^2\), with various prior probabilities.  (Exponents higher than two can be incorporated, as well, but are not addressed explicitly.)  The proposed method of maximization is to use a grid search from \(p^{-1}\tr(\bSn)\) to \(20\lv \bSn\rv\) to maximize a spiked analogue of \(\fU(1/f)\) for $f(x)=x+b$. 

In our implementation, we only consider \blue{signal} dispersion matrices of \(\mathbf{I}_p\) and \(\bfrn\), and we perform a log-scale grid search consisting of \(100\) points over the suggested interval.

\textbf{Limitation}: This algorithm tends to degrade in performance when the spectrum of $\bfr$ is highly complex, consisting of many significant eigenvalues, rather than conforming to the low-rank-plus-noise model.  Further, only partial guidance is provided on how to implement the grid search for $b$, and overly fine grid searches appear to yield spurious optima in moderate dimensions such as 200. 

\textbf{Knowledge of $\bfr_n$ and $\myOmega_n$ used}: The matrix $\bfr_n$ is assumed to be a finite-rank matrix plus the identity (spiked).  The signal dispersion matrix is assumed to belong to a set of matrices of the form $w_0 \mathbf{I}_p + w_1\bfr_n + w_2 \bfr_n^2$ for some constrained set of weights $w_j$.  (Canonically, the weights belong to $\{(1,0,0), (0,1,0), (0,0,1)\}$.)  In our simulations, we will make the generous assumption that the correct weights are  known exactly, and we will only consider the cases of $\myOmega_n = \mathbf{I}_p$ or $\bfr_n$. 

\subsubsection{\lwalg{}}
Based on Ledoit and Wolf's work in \cite{ledoit2022quadratic}, we choose to replace the sample covariance matrix with the \textit{nonlinear} ``quadratic inverse shrinkage'' estimator using code provided by the authors in their supplemental material.   
\blue{
Motivated by unifying their earlier shrinkage work with 
the pioneering covariance-estimation work by Stein 
\cite{stein1975estimation},
\lwalg{} refines the earlier works by grounding both in a rigorous convolution-kernel based approach. 
The derived estimator was later proven to order-minimize the Frobenius distance from the equivariant oracle estimator \cite{lin2026eigenvector}, making it not only first-order but second-order asymptotically optimal.
}
The algebraic form of their covariance estimator differs slightly from our Theorem~\ref{thm:local-lw}, but both satisfy the conclusions of Theorem~\ref{thm:lw}, differing mainly in kernel choice.  

\textbf{Limitation}: This estimator is designed to optimize several objective functions, such as Frobenius loss, ``MV loss,'' and inverse Stein's loss.  MV loss is essentially the reciprocal of the following \cite{ledoit2017nonlinear}:
\begin{equation} \label{eq:mv-loss}
\mathcal{U}_{n,MV}(f_n):=\frac{\tr(f_n(\bSn))}{\sqrt{\tr( f_n(\bSn) \bfr_n f_n(\bSn))}}.
\end{equation}
This differs by a factor of \(\bfr\) in the denominator from our \(\mathcal{U}_n(\mydot)\), meaning that that our estimator may differ from \lwalg{} when \(\bfr\) is far from the identity matrix (or a scaling thereof).

\textbf{Knowledge of $\bfr_n$ and $\myOmega_n$ used}: Essentially none.  The matrix $\bfr_n$ need only have limiting spectrum supported on a finite union of intervals (in particular, true for spiked models), and $\myOmega_n$ is not assumed to be known.  


\subsubsection{\cwhalg{}}
Chen et al. \cite{chen2011robust} apply their detector in a similar context to ours, where two samples are tested but one sample is just a singleton.  However, their data model generalizes the Gaussian distribution differently, assuming a spherically-invariant (elliptical) model rather than the separable-covariates model we adopt.  The \MP{} Theorem happens to be valid in both settings, so it would not be a surprise if their work extends to our model and vice versa.

Fundamentally, \cwhalg{} first extends the well-known Tyler M-estimator \cite{tyler1987distribution} into Hotelling's \(T^2\) to the sample-starved case of \(n < p\), by formulating and iteratively solving a fixed-point equation.  Thus, in our plots, \cwhalg{} effectively reduces to Tyler's estimator until we consider the case of $p > n$ in Section~S.10
of the Supplemental Material.  The resulting estimator is then substituted for \(\bSn\) in Hotelling's \(T^2\).
This linear shrinkage is shown to outperform Ledoit and Wolf's \textit{linear} shrinkage algorithm \cite{ledoit2004well}, for example, on a supervised detection task using the \textsc{Crawdad} data set.  

\textbf{Limitation}: This is another form of linear shrinkage and is thus relatively inexpressive, having only one degree of freedom.  By contrast,  nonlinear shrinkage has $p$ degrees of freedom.

\textbf{Knowledge of $\bfr_n$ and $\myOmega_n$ used}: None. This algorithm is based solely on maximum-likelihood and least-squares computations, and it does not require or use knowledge of $\myOmega_n$.

\subsubsection{Proposed algorithm}

The proposed algorithm is essentially just to use the approximation to \(\fopt\) described in Theorem~\ref{thm:approx-fopt}. The only modification we make is to replace every shrunken eigenvalue by its positive part (maximum with zero).  Since  \(\fopt\) is asymptotically non-negative, this modification does not affect asymptotic consistency and can only accelerate convergence.

\textbf{Knowledge of $\bfr_n$ and $\myOmega_n$ used}: The matrix $\bfr_n$ is assumed to follow [\textsc{Train}].  The signal dispersion matrix is assumed to be known.  We denote this method by SRHT($\myOmega_n$) when the signal-dispersion matrix is assumed to be proportional to $\myOmega_n$.  For example, we will write SRHT($\mathbf{I}_p$) or SRHT($\bfr_n$) if $\myOmega_n$ is proportional to $\mathbf{I}_p$ or $\bfr_n$, often omitting the subscripts of $n$ and $p$. 

\subsection{Results for synthetic data} 
\label{sec:synthetic}

For our first series of experiments, we tested our detector on artificially generated reference and test data sets in \texttt{python}.  The data generation process was similar to that in \cite{robinson2022improvement}. 
 First, let \(p=200\), choose a \blue{signal} dispersion matrix \(\myOmega\) and non-centrality scalar \(\magnitude\), choose the number of reference samples \(n,\) and generate a \(p\times p\) covariance matrix \(\bfr\) with piece-wise log-linear eigenvalues \(\{\kappa^{ i/40}\}_{i=1}^{40}\cup\{10^{ (i-1)/(40(p-41))}\}_{i=1}^{p-40}\) for some rough condition number \(\kappa = \lambda_{\max}(\bfr)/\lambda_{\min}(\bfr)\), not to be confused with a distance to a spectral edge.  (See Figure~\ref{fig:scree1e2} 
 for a scree plot with \(\kappa = 10^2\) and Remark~\ref{rem:spiked} for a comment on the smaller \(p-40\) eigenvalues.)
If $\kappa =1$, we take $\bfr$ simply to be $\mathbf{I}_p$. Next, perform 3,000 Monte Carlo trials  as follows:
\begin{itemize}
    \item Generate an i.i.d. \(p\times n\) matrix \(\bZ\) with 
    \blue{components $\sim W$ and two possible choices for the tail behavior of $W$: $W\sim U(-\sqrt{3},\sqrt{3})$ (sub-Gaussian case), or $W\sim t_\nu$ (Student-$t$ case) for $\nu>2$, normalized to have variance one. Left-multiply $\bZ$ by a random orthogonal matrix with rotation-invariant (Haar) distribution, and} 
    let \(\bX = \bfrhalf \bZ\), the reference-sample matrix.
    \item Generate a large number test observations, some signal-free ones distributed as the columns of \(\bX\), and some signal-containing ones, which are signal-free observations added to \(\magnitude\bz/\lv \bz\rv\), where \(\bz \sim \mathcal{N}(0,\myOmega)\).
    \item Compute \(\bSn\) from \(\bXn\) as in \eqref{eq:scm} and generate detection scores for each algorithm in the last section and each test observation.
\end{itemize}
Once detection scores have been generated for all 3,000 trials, they are assembled into approximate receiver-operator-characteristic (ROC) curves of size (false-alarm rate) versus power (detection rate) by counting the number of \(\fH_0\) and \(\fH_1\) observations that exceed each of a large number of thresholds.

\begin{figure}[htbp]
\centering
\includegraphics[width=0.85\columnwidth]{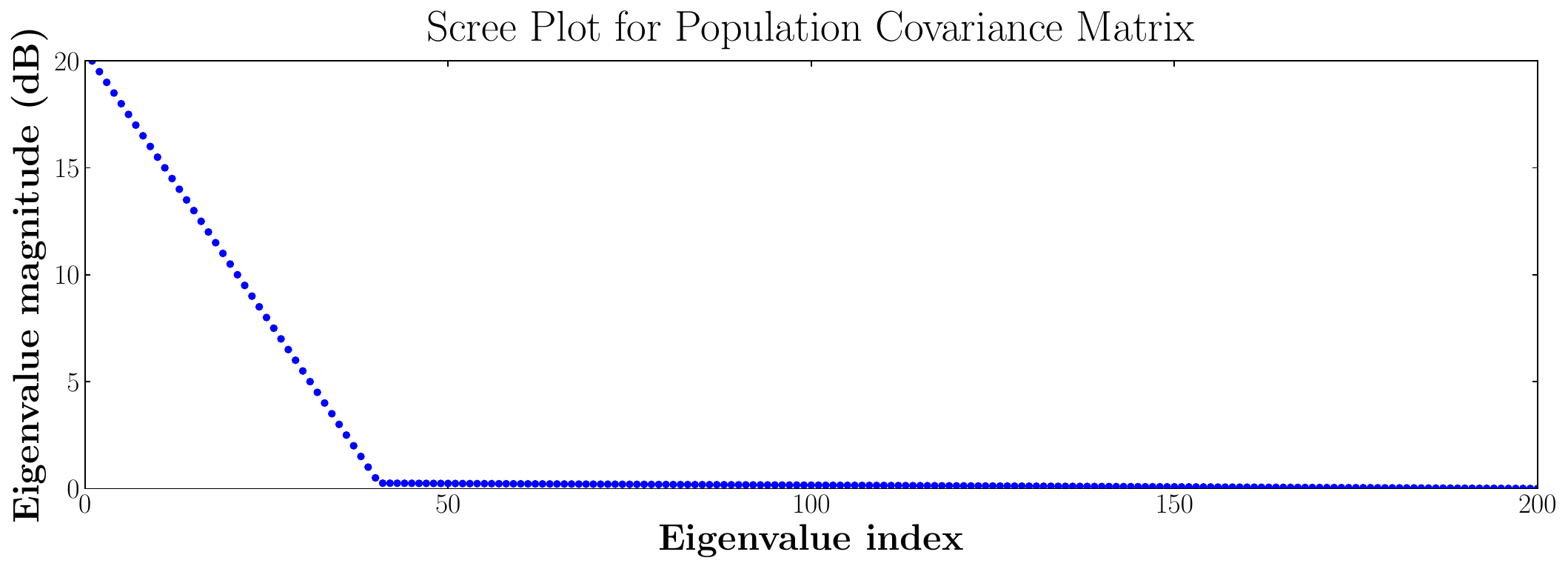}
\caption{Scree plot of population covariance matrix chosen to generate artificial data, \(\kappa=10^2\).
}
\label{fig:scree1e2}
\end{figure}


Some of the results for the case of an isotropic signal-dispersion matrix are shown in a series of plots in Figure~\ref{fig:isotrue_subg}.  Specifically, the case where the data-generating matrix $\myOmega$ is known to be proportional to the identity matrix is investigated.  Figure~\ref{fig:isotrue_subg} demonstrates our method's robustness to spectral behavior for sub-Gaussian data, showing performance as \(\kappa=\kappa(\bfr)\) ranges over \((10^0, 10^2, 10^4)\).
Further, Figures~S.2, S.3, S.6, and S.7
of the Supplemental Material demonstrates our method's robustness to heavy-tailed behavior, showing similar qualitative performance when $W$ is proportional to $t_8$-, $t_6$-, and $t_4$-distributed random variables. 

Figure~\ref{fig:isofalse_subg} demonstrates our method's additional robustness to misspecification of $\myOmega$.  Specifically, the effect of incorrectly assuming $\myOmega \propto \mathbf{I}$ when in fact $\myOmega\propto \bfr$ is investigated.    Results for both the incorrect specification and the correct specification are plotted for the case of sub-Gaussian data, as $\kappa$ ranges over ($10^1, 10^2, 10^4$).  (The case of $\kappa = 10^0$ matches the one in Figure~\ref{fig:isotrue_subg} and is suppressed to conserve space.) 

\begin{figure}[htbp]
    \centering
    {\large\textbf{Robustness to the Spectral Complexity of $\bfr$}}\\[1ex]
        \includegraphics[width=\linewidth]{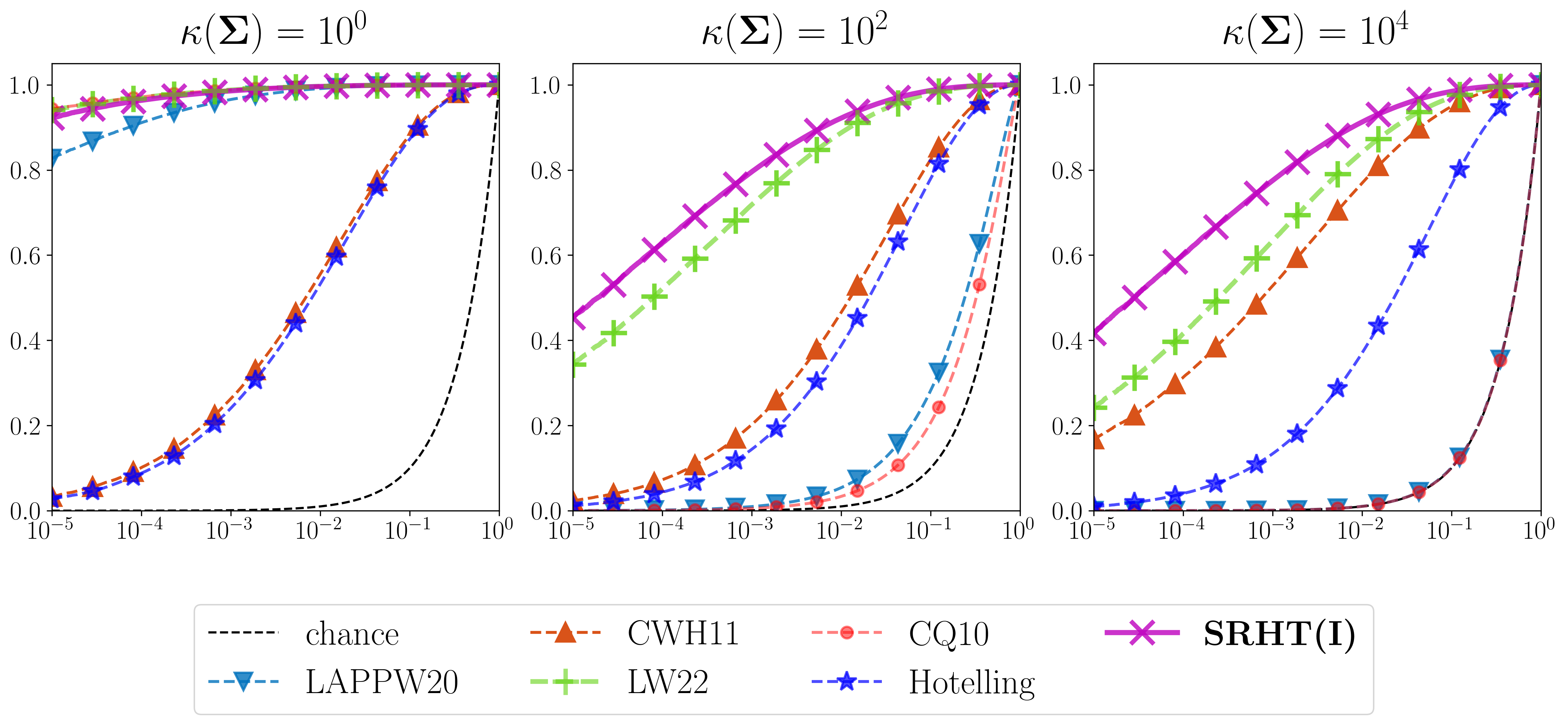}
        \caption{
        Size-adjusted empirical power of several methods, where $p=200$, $n=300$, the data are sub-Gaussian, and the true signal-dispersion matrix $\myOmega\propto\mathbf{I}$.
        The proposed method SRHT($\mathbf{I}$) increasingly outperforms the competition as the spectral complexity of $\bfr$ increases.  
        }
    \label{fig:isotrue_subg}
\end{figure}

\begin{figure}[htbp]
    \centering
    {\large\textbf{Performance Under Covariance-Matched Prior}}\\[1ex]
        \includegraphics[width=\linewidth]{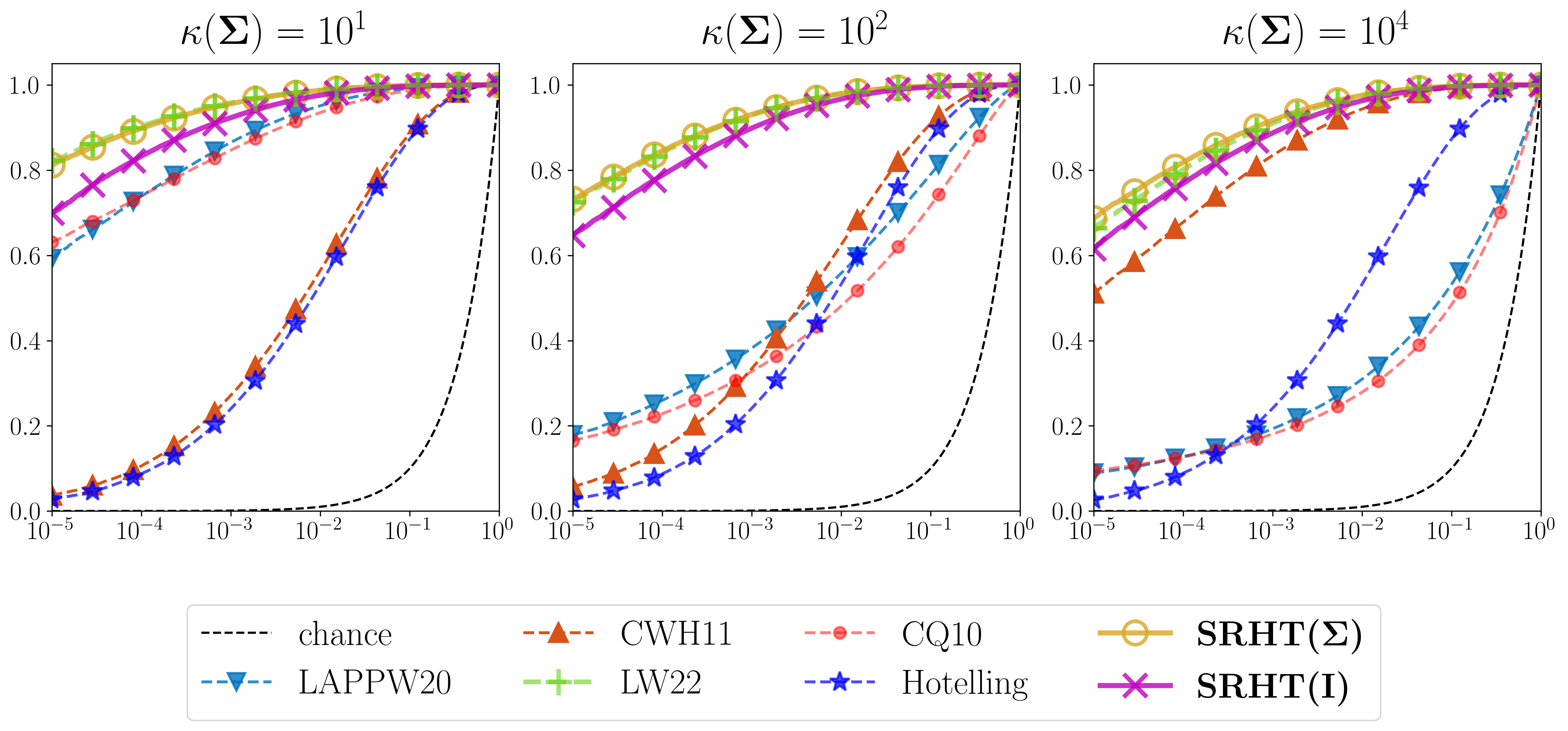}
        \caption{
        Size-adjusted empirical power of several methods, where $p=200$, $n=300$, data are sub-Gaussian, the true signal-dispersion matrix $\myOmega \propto \bfr$, and $\kappa(\bfr)$ ranges from $10^1$ to $10^4$.
        Our (finite-sample) method SRHT($\bfr$) demonstrates the highest performance, as expected, roughly tying \lwalg{}.  By comparison, the method  SRHT($\mathbf{I}$) no longer outperforms all competitors but displays robustness to signal-prior misspecification, especially as spectral complexity increases.
        }
    \label{fig:isofalse_subg}
\end{figure}

\begin{rem} \label{rem:spiked}
We note that, in accordance with the remarks following {[\textsc{Train5}]}, essentially equivalent results can be observed in the more nearly spiked model where the trailing 160 eigenvalues of \(\bfr\) are chosen to be 1, reaffirming that our results extend beyond the model of {[\textsc{Train5}]} to the spiked model.  Nevertheless, to demonstrate the validity of our results under the conditions we have given, we choose instead to conduct tests in which {[\textsc{Train5}]} is more closely approximated.  
\end{rem}

\subsection{Results for measured data} \label{sec:crawdad}

\begin{figure}[htbp]
    \centering
    {\large\textbf{Performance on Crawdad Data Set}}\\[1ex]
        \includegraphics[width=\linewidth]{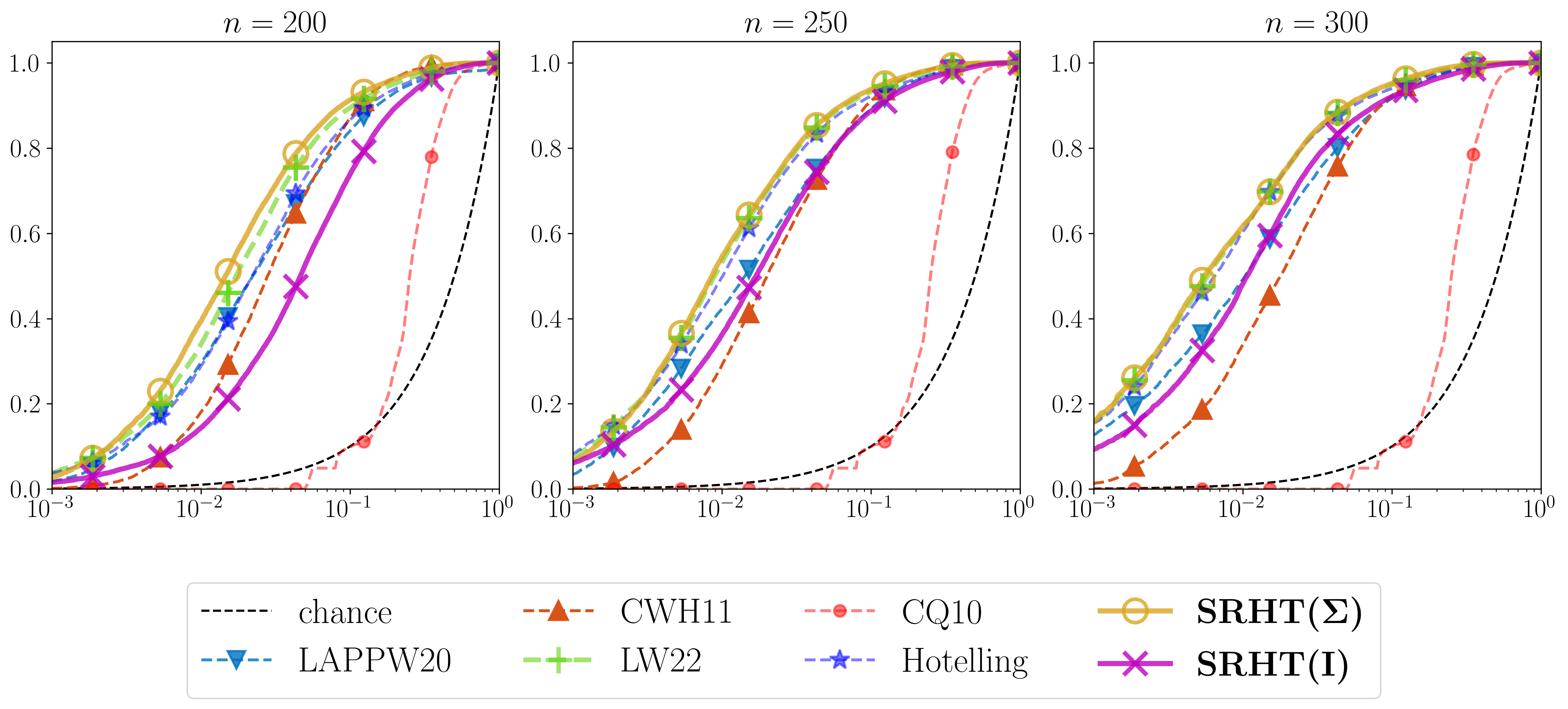}
        \caption{
        Size-adjusted empirical power of several methods and training set sizes $n$ for the Crawdad data set, where $p=182$.
        The method SRHT(\textbf{I}) appears to be outperformed by several methods for $n=200$---although the effect appears to lessen with increasing $n$---while another of our methods exists, SRHT($\bfr$), that roughly outperforms the competition.
        }
    \label{fig:crawdad}
\end{figure}

In addition to testing our algorithm on synthetically generated sub-Gaussian data, we also test our algorithm on real, measured sensor-network data.  The data set to which we apply our method is the \textsc{Crawdad} Umich/RSS one \cite{c7r30h-22}, 
in which 14 Mica2 sensors were distributed throughout a lab space to detect whether a person was present and/or moving there.  In order to do so, the sensor network collected and recorded received signal strength (RSS) measurements for each sender-receiver pair of sensors, totaling \(p=14\cdot 13 = 182\) measurements at each time instant.  Whether or not there was activity in the room was known during the experiment and can be used to determine the performance of a detection algorithm applied to the data.  RSS measurements were collected over 3127 time instants, spaced 0.5 seconds apart, with 327 time instants corresponding to activity.  In addition to human activity, RSS measurements were weakened/disturbed by background interference and noise due to cellular, wifi, and radio signals, and other nuisance sources of electromagnetic radiation.  
We have de-trended the data in the same way as in the simulations of \cite{chen2011robust}.

In contrast to our tests in the previous section, one does not know a priori the \blue{signal} dispersion matrix for the \textsc{Crawdad} data set.  
A natural guess might be the ignorant prior of \(\myOmega \propto \mathbf{I}\).  However, this choice appears to result in \blue{degraded} performance, 
\blue{
at least for $n=200$, 
in comparison to its seemingly more robust performance in Figure~\ref{fig:isofalse_subg}'s plots for synthetic data.  This effect could be due to model bias in [\textsc{Train}]: in particular, the apparent correlation and heavy-tailed nature of training samples, noted in \cite[Section~V]{chen2011robust}.

As we have observed in Remark~\ref{rem:prior}, we cannot learn $\myOmega$ from only one test observation.  Still, it is illustrative to 
examine the performance of our method for more than one signal prior: not only $\myOmega \propto \mathbf{I}$ but
\(\myOmega \propto \bfr\).  We note that our method essentially outperforms the competition once this second signal-prior assumption is made, suggesting that an optimal method \textit{exists} within the family SRHT($\myOmega$), at least for this data set.}

 This assumption that $\myOmega \propto \bfr$ would imply that signals indicating human activity are well-masked by interfering signals, perhaps because the sensor network mainly detects the component of lab users' cell phone signals that are well-aligned with the interference environment. Assuming that \(\myOmega \propto\bfr\) means that, according to \eqref{eq:int-omega}, in the limit \(\overline{h}(x)\) should be equal to \(\delta(x)\), and thus \(\overline{h}_n(x)\) can be chosen to be \(\tilde{d}_n(x)\) in Theorem~\ref{thm:approx-fopt}.  

Our experimental procedure is to pick 1000 reference samples of cardinality \(n\) from the \(3127-327=2800\) inactive time indices, and for each reference sample test the remaining \(3127-n\) time indices, generating a detection statistic for each using our algorithm and the comparators from the last section.
We then generate an empirical ROC curve for each algorithm by plotting how many detection statistics for inactive and active time indices exceed a sliding threshold relative to the totals of 2800 and 327, respectively.  Our results for \(n=200\), $n=250$, and \(n=300\) are shown in Figure~\ref{fig:crawdad}.  

The most noticeable feature of this figure is that our  SRHT($\bfr$) algorithm roughly ties with \lwalg{}.  Similar experiments with synthetic data show the same pattern, which is intriguing since the eigenvalues of \lwalg{} differ significantly from ours.  This trend suggests that the class of nearly optimal estimators is fairly flexible in some settings.  An interesting avenue for future research would be to explore whether this flexibility can be exploited to significantly improve variance of our algorithm's finite-sample performance without much effect on its mean.

\section*{Acknowledgments}
This work was supported by the US Air Force Sensors Directorate and AFOSR Lab Task 22RYCOR006.
However, the views and opinions expressed in this article are those of the authors and do not necessarily reflect the official policy or position of any agency of the U.S. government.
Examples of analysis performed within this article are only examples. 
Assumptions made within the analysis are also not reflective of the position of any U.S. Government entity. The Public Affairs approval number of this document is AFRL-2026-2955.


\begin{supplement}
\stitle{Proofs, Code, and Discussion}
\sdescription{This supplement contains detailed proofs of the paper's results, \texttt{python} code for our proposed shrinkage methods, additional plots, and discussions of several possible extensions to the developed theory.}
\nocite{suppRobinson}
\end{supplement}

\bibliographystyle{imsart-number} 
\bibliography{information-geometry-bib2.bib}






\renewcommand{\blue}[1]{{\color{blue!65!black}{#1}}}

\clearpage
\setcounter{page}{1}


\renewcommand{\blue}[1]{#1} 
\renewcommand{\theequation}{S.\arabic{equation}}
\renewcommand{\thefigure}{S.\arabic{figure}}
\renewcommand{\thesection}{S.\arabic{section}}
\numberwithin{equation}{section}

\begin{center}
{\Large \textbf{Supplementary Material for ``Spectrally robust covariance shrinkage for Hotelling's $T^2$ in high dimensions''}}

\vspace{0.25cm}
Benjamin D. Robinson$^{1}$ and Van Latimer$^{2}$

\vspace{0.25cm}
${}^{1}$Air Force Office of Scientific Research \\
${}^{2}$Radial Research and Development, Inc.

\end{center}
\vspace{1cm}

\section{Preliminary boundedness results for $\myw, \Hw,$ and $\delta$} \label{app:bounded-delta}

In what follows we will need to use the fact that \(\delta(x)\), defined in 
(2.4),
is bounded below and smooth on a neighborhood of the limiting sample spectral support \(F\), except at its boundary points.

 First consider boundedness.  There is a bound \(C\) such that  of \(1/\delta(x) \le C\) for almost every \(x\in F\), due to the mean value theorem and the identity 
 (4.2).
 Continuity of \(1/\delta(x)\) everywhere follows from continuity of \(\myw(x)\) and \(\Hw(x)\) on \(\bbR\)---the latter of which follows from continuity of \(\myw(x)\) on \(\bbR\) \cite{pan2026local}.  In other words,  \(1/\delta(x)\) is continuous on any neighborhood of \(F\) and bounded on \(F\), which implies a bound of \(2C\) on some neighborhood of \(F\), as desired.
   
For behavior near $\partial F$, we use the following lemma.
 \begin{lem} \label{lem:van-wHw}
    For all \(\mathbb{Z} \ni k \geq 1\), the function \(w\) satisfies 
\begin{equation}
    \lvert w^{(k)}(x) \rvert = O(\kappa(x)^{1/2-k}),
\end{equation}
as does \(\mathcal{H}w\), where \(\kappa(x)\) is the distance from \(x\) to the spectral edge. 
\end{lem}

We use an alternative but equivalent self-consistent equation to (3.2) from \cite{knowles2017anisotropic}:
\begin{equation}
    \frac{1}{\altm} = -z + \phi \int \frac{x}{1 + \altm x} \mathrm{d}\upopsm(x),
\end{equation}
where
\begin{equation} \label{eq:sce_alt}
    \altm(z) = (1-\phi)\left( \frac{-1}{z} \right) + \phi \um(z).
\end{equation}
Whereas \(\um\) describes the limiting density of the eigenvalues of \(\bX^\top \bX\) \vll{correct notation}, \(\altm\) describes that of \(\bX \bX^\top\), so that \(\altm\) differs from \(\um\) precisely in the inclusion of a dirac mass at 0 (in the measures of which they are Stieltjes transforms). The reason we prefer \(\altm\) and its self-consistent equation \eqref{eq:sce_alt} to \(\um\) and equation 
(3.2)
is that \(\altm\) is a explicitly a solution to 
\begin{equation}
    f(\altm) = z
\end{equation}
where \(f\) is an analytic function. At the conclusion of the lemma we will show how these self-consistent equations are equivalent. From from \cite{knowles2017anisotropic}, we know that the spectral edges are those values \(a \in \mathbb{R}\) for which the preimage \(m\) satisfying \(f(m) = a\) also satisfies \(f'(m) = 0\). 

From this alone we can deduce the square-root behavior at the edges that we desire. Firstly, 
by the inverse function theorem, $\myw(x)$ is real-analytic when $x\not\in\partial F$, and by \cite{pan2026local}, the same is true of $\Hw(x)$.  Second, since 
\begin{equation}
    i \pi w^{(k)}(x) + \pi (\mathcal{H}w)^{(k)}(x) = \um^{(k)}(x), 
\end{equation}
for real \(x\) (where the derivatives on the left-hand side are real derivaties of real function of one variable and the derivative on the right is the usual complex derivative of a meromorphic function of one complex variable), it suffices to bound
\begin{equation}
    \lvert \um^{(k)} \rvert \lessapprox \kappa^{\frac{1}{2} - k}
\end{equation}
for \(k \geq 1\). Moreover, it is evident that it suffices to prove the above for \(\um\) replaced with \(\altm\). 
We proceed by induction. See that it suffices to show 
\begin{equation}
    \lvert (\altm - \altmNaught)^{(k)} \rvert = O(\kappa^{\frac{1}{2} - k}).
\end{equation}
for \(k \geq 0\) (\(\altmNaught\) is a constant). 

The statement for \(k=0\) is given explicitly in \cite{knowles2017anisotropic}. 

Now Faa di Bruno's rule gives that 
\begin{equation}
    \begin{split}
        \frac{d^{(k+1)}}{d z^{(k)}} f(\altm)&=\sum_{\pi \in \Pi} f^{(|\pi|)}(\altm) \cdot \prod_{B \in \pi} \altm^{(|B|)} \\
        &= f'(\altm) \altm^{(k+1)} + \sum_{\pi \in \Pi\setminus \{\{1, \dots, k+1\}\}} f^{(|\pi|)}(\altm) \cdot \prod_{B \in \pi} \altm^{(|B|)} \\
    \end{split}
\end{equation}
where \(\Pi\) is the set of partitions of the integers \(\{1, \dots, k+1\}\), \(\lvert \pi \rvert\) is the number of blocks in the partition \(\pi\), and where \(\lvert{B}\rvert\) is the number of elements in a block \(B\) in a partition \(\pi\); in the second line we have split out the partition with only one block, where the \(k+1\)th derivative that we seek is apparent. This gives 
\begin{equation}
    \begin{split}
        \lvert \altm^{(k+1)}\rvert &= \left \lvert\frac{\sum_{\pi \in \Pi\setminus \{\{1, \dots, k+1\}\}} f^{(|\pi|)}(\altm) \cdot \prod_{B \in \pi} \altm^{(|B|)}}{f'(\altm)} \right \rvert \\
        &\lessapprox \left \lvert\frac{\sum_{\pi \in \Pi\setminus \{\{1, \dots, k+1\}\}} \prod_{B \in \pi} \kappa^{(\frac{1}{2} - |B|)}}{f'(\altm)} \right \rvert \\
        &\lessapprox \left \lvert\frac{\sum_{\pi \in \Pi\setminus \{\{1, \dots, k+1\}\}}\kappa^{(1 - (k+1))}}{\altm-\altmNaught} \right \rvert\\
        &\lessapprox \left \lvert\frac{\kappa^{(1 - (k+1))}}{\altm-\altmNaught} \right \rvert\\
        &\lessapprox \left \lvert\kappa^{(\frac{1}{2} - (k+1))}\right \rvert
    \end{split}
\end{equation}
where in the second line we used the analyticity of \(f\) and the inductive hypothesis \(\altm^{(l)} = O(\kappa^{\frac{1}{2} - l})\) for \(l \leq k\) and in the third line that \(B\) has at least 2 blocks.

Now, regarding the equivalence between the self-consistent equations, we will show that if \(m\) satisfies (3.2), then \(\altm\) satisfies \eqref{eq:sce_alt} (the reverse direction is similar). 

Equation (3.2) directly becomes, upon plugging in the definition of \(\altm\) to the denominator of the integrand,
\begin{equation}
    \um = -\frac{1}{z} \int \frac{1}{1 + x \altm}\mathrm{d} \pi_\infty(x) ,
\end{equation}
and a second application of the definition of \(\altm\) shows that then
\begin{equation}
    \begin{split}
        \altm &= (1-\phi)\left( \frac{-1}{z} \right) -\frac{1}{z} \phi \int \frac{1}{1 + x \altm} \mathrm{d} \pi_\infty(x) \\
        &= (1-\phi)\left( \frac{-1}{z} \right) -\frac{1}{z} \phi \int \left(1 - \frac{x\altm }{1 + x \altm} \right) \mathrm{d} \pi_\infty(x) \\
        &= (1-\phi)\left( \frac{-1}{z} \right) -\frac{1}{z} \phi \left( 1 - \altm \int \left(\frac{x }{1 + x \altm} \right) \mathrm{d} \pi_\infty(x)\right) \\
        &= \frac{-1}{z} +\frac{1}{z} \phi \altm \int \left(\frac{x }{1 + x \altm} \right) \mathrm{d} \pi_\infty(x).
    \end{split}
\end{equation}
Multiplying by \(\frac{z}{\altm}\) and rearranging terms yields 
\begin{equation}
    \frac{1}{\altm} = -z + \phi \int \left(\frac{x}{1 + x \altm} \right) \mathrm{d} \pi_\infty(x)
\end{equation}
as desired.


   

\section{Helffer-Sj\"ostrand results}\label{sec:app-HS}

This section closely follows \cite[Section 8]{benaych2016lectures}.  First we make an important remark about notation that will apply throughout the rest of this paper.

\begin{rem} 
In this Section and the following ones, some of the formulas will involve many key quantities that depend on \(n\), which could result in an excessive number of subscripts.  As a result, we will follow the convention of \cite{knowles2017anisotropic} and typically omit indices of \(n\), so that all quantities not expressly deemed constant (as the measure \(\upopsm\) is) may have a hidden dependence on \(n\).  Thus, we abbreviate \(\msamplen(z)\) by \(\msample(z)\), \(\bfrn\) by 
\(\bfr\), 
\(\bSn\) by \(\bS\), the sample resolvent 
\(\bRn(z)\) by 
\(\bR(z)\), etc; and statements such as \(\msample(z)\to \um(z)\) (a.s.) should cause no confusion.
Having made this note, however, we will still occasionally use indices of \(n\) for emphasis or clarity.
\end{rem}

The main result of this section is as follows. 
\begin{lem} \label{lem:hs_main_result}
    Assume \emph{[\textsc{Train}]}. Then, uniformly for \(f \in C^2(\mathbb{R})\), we have
    \begin{equation} \label{S-eq:hs_main_result_eig}
        \int f\myddiff \mu -\int f\myddiff \umu =O_\prec \left(
        \frac{1}{n} \left\lVert f \right\rVert _1 
        + \frac{1}{n} \left\lVert f' \right\rVert _1 
        + \frac{1}{n^2} \left\lVert f'' \right\rVert _1 \right)
    \end{equation}

    Moreover, letting \(\nu\) be the sample Ledoit-P\'ech\'e measure and \(\ulpmeas\) its limit, we have the same result when \(\mu\) and \( \umu\) are replaced with \(\nu\) and \( \ulpmeas\) 
\end{lem}

We have stated Lemma \ref{lem:hs_main_result} in this way to accomodate general \(n\)-dependent functions \(f\). One consequence is the more familiar ``law on small scales'':
\begin{corollary} \label{cor:law_on_small_scales}
    Let \(I = [a, b] \subseteq \mathbb{R}\) be an interval. Then uniformly in \(I\), 
    \begin{equation} \label{S-eq:law_small_scales_eig}
        \mu(I) - \umu(I) =O_\prec \left(
        \frac{1}{n} \right)
    \end{equation}
    and
    \begin{equation*} \label{S-eq:law_small_scales_lp}
        \nu(I) - \ulpmeas(I) =O_\prec \left(
        \frac{1}{n} \right)
    \end{equation*}

\end{corollary}
\begin{proof}[Proof of Corollary \ref{cor:law_on_small_scales}]
    Fix \(\epsilon > 0\) and let \(\eta = n^{\epsilon -1}\). 
    Let \(f: \mathbb{R}\to[0,1]\) be \(C^\infty\), equal \(1\) on \(I\), equal \(0\) on \(\left(a-\eta, b+\eta\right)^c\), and satisfy \(\left\lVert f^{\prime}\right\rVert _\infty \leq C\eta^ {-1}\) and \(\left\lVert f^{\prime\prime}\right\rVert _\infty \leq C\eta^{-2}\). Thus \(f\) ``just barely'' dominates the indicator function of \(I\). Lemma \ref{lem:hs_main_result} implies that 
    \begin{equation}
        \int f \myddiff \mu - \int f \myddiff \umu = O_\prec \left(n^{-1}\right)
    \end{equation}
    so that, using the facts that the densities of \(\umu\) and \(\ulpmeas\) are bounded above and below on \(F\) and
    \begin{equation*}
        \mu (I) \geq \int f \myddiff  \mu - Cn^{-1+\epsilon} = \int f \myddiff  \umu + O_\prec \left( n^{-1} \right) \geq \umu(I) + O_\prec \left( n^{-1} \right).
    \end{equation*}
    Repeating this logic when \(f\) is instead ``just barely'' dominated by the indicator function of \(I\) provides the reverse inequality. 
\end{proof}

\begin{proof} [Proof of Lemma \ref{lem:hs_main_result}]
    We proceed under the assumption that \(f\) is compactly supported; this only strengthens the bound since the supports of \(\mu\) and \(\nu\) are bounded by some fixed constant with high probability.  

    Let \(\hat{\mu} = \mu - \umu\), and let \(\hatm = m - \um\) be the Stieltjies transform of \(\hat{\mu}\). 

    Let \(\epsilon > 0\) and \(\eta = n^{-1 + \epsilon}\).

    Let \(\chi \in \mathcal{C}^\infty(\mathbb{R}, [0,1])\) be a  smooth cutoff function with \(\chi(0) > 0\). Lastly, let \(h \in \mathcal{C}^\infty(\mathbb{R}, [0,1])\) be supported on \([a-\eta, b+\eta]\), be identically 1 on \([a, b]\), and satisfy \(\lVert h' \rVert _\infty \leq C \eta^{-1}, \lVert h'' \rVert _\infty \leq C \eta^{-2}\).

    By Helffer-Sj\"{o}strand formula, we may write the left-hand side of equation \eqref{S-eq:hs_main_result_eig} as
    \begin{equation*}
        \int f(\lambda) \hat{\mu}(\myddiff  \lambda)=\frac{1}{2 \pi} \int \myddiff  x \int \myddiff  y\left(\partial_x+\imath \partial_y\right)\left[\left(f(x)+\imath y f^{\prime}(x)\right) \chi(y)\right] \hatm(x+\imath y),
    \end{equation*}
    which upon expanding and using that the left-hand side is real, reads
    \begin{equation*} \label{S-eq:split_hs}
        \begin{split}
            \int f(\lambda) \hat{\mu}(\myddiff  \lambda)=&-\frac{1}{2 \pi} \int \myddiff  x \int \myddiff  y f^{\prime \prime}(x) \chi(y) y \operatorname{Im} \hatm(x+\imath y) \\
            &-\frac{1}{2 \pi} \int \myddiff  x \int \myddiff  y f(x) \chi^{\prime}(y)\operatorname{Im} \hatm(x+\imath y) \\
            &- \frac{1}{2 \pi} \int \myddiff  x \int \myddiff  y f'(x) \chi^{\prime}(y) y \operatorname{Re} \hatm(x+\imath y)
        \end{split}
    \end{equation*}

    The second and third terms are easiest to bound. Since 
    \(\left| \chi'(y)\hatm(x+\imath y) \right| \leq \eta\) uniformly on the purely complex set \(\mathrm{supp}\, \chi^{\prime}(y) \subseteq \mathbb{C}\), we bound the second term by 
    \begin{equation*} \label{S-eq:hs_2_bound}
        \begin{split}
            &\lessapprox \left| \int \myddiff x f(x) \int \myddiff y \chi'(y) \operatorname{Im} \hatm(x+\imath y)\right|\\
            &\lessapprox \eta \int \myddiff x \left| f(x) \right|.
        \end{split}
    \end{equation*}
    Similarly the third term is bounded by 
    \begin{equation*} \label{S-eq:hs_3_bound}
        \begin{split}
            &\lessapprox \left| \int \myddiff x f^{\prime}(x) \int \myddiff y \chi'(y) y \operatorname{Re} \hatm(x+\imath y)\right|\\
            &\lessapprox \eta \int \myddiff x \left| f^{\prime}(x) \right|.
        \end{split}
    \end{equation*}
    
    Now, we must bound the first term
    \begin{equation} \label{S-eq:hs_first_term}
        -\frac{1}{2 \pi} \int \myddiff  x \int \myddiff  y f^{\prime \prime}(x) \chi(y) y \operatorname{Im} \hatm(x+\mathrm{i} y). 
    \end{equation}
    First we bound
    \begin{equation} \label{S-eq:hs_first_term_small_y}
        -\frac{1}{2 \pi} \int \myddiff  x \int_{\left|y\right| \leq \eta} \myddiff  y f^{\prime \prime}(x) \chi(y) y \operatorname{Im} \hatm(x+\mathrm{i} y). 
    \end{equation}
    It is a general fact of Stieltjies transforms \(t\) of positive measures that 
    \begin{equation*}
        y\mapsto y |t(x + \mathrm{i} y)|
    \end{equation*}
    is nondecreasing for \(y > 0\) and for any \(x\). Since \(\eta |\hatm(x+\mathrm{i} \eta)| \leq \eta\) with high probability, and moreover \(\eta |\um(x+\mathrm{i} \eta) | \leq \eta\) since the density of \(\umu\) is bounded, we find that \(\eta |m(x+\mathrm{i} \eta)| \leq \eta\).
    This general fact may then be applied to both \(y |m(x + \mathrm{i} y)|\) and \(y |\um(x + \mathrm{i} y)|\), together with \(t(\overline{z}) = \overline{t(z)}\), to establish 
    \begin{equation} \label{S-eq:maxineq}
        \max_{\left|y\right| \leq \eta} \left|y\hatm(x + \mathrm{i}y)\right| \leq \eta
    \end{equation}
    with high probability. Now we may bound \eqref{S-eq:hs_first_term_small_y} by 
    \begin{equation*}
        \lessapprox \int \myddiff  x \left|f^{\prime \prime}(x)\right|  \int_{\left|y\right| \leq \eta} \myddiff  y \cdot \eta \lessapprox \eta^{2}\int \myddiff  x \left|f^{\prime \prime}(x)\right|
    \end{equation*} 
    with high probability. 

    Lastly we bound 
    \begin{equation} \label{S-eq:hs_first_term_large_y}
        -\frac{1}{2 \pi} \int \myddiff  x \int_{\left|y\right| > \eta} \myddiff  y f^{\prime \prime}(x) \chi(y) y \operatorname{Im} \hatm(x+\mathrm{i} y). 
    \end{equation}

    This is more difficult that the previous bound because the triangle inequality for integrals is no longer affordable. The way around this is to use that \(\int \myddiff x f''(x)\) has a much tighter bound than \(\int \myddiff x \left|f''(x)\right|\), through the fundamental theorem of calculus. Some care is still required, however, since we must not bound \(\int \myddiff x f''(x)\) only but the larger expression \eqref{S-eq:hs_first_term_large_y} which has another \(x\)-dependent factor in the integrand. 

    We integrate by parts: first in \(x\), noting that \(f'\) vanishes at \(\pm \infty\), 
    \begin{equation*}
        \begin{split}
            & -\frac{1}{2 \pi} \int \myddiff  x \int_{\left|y\right| > \eta} \myddiff  y f^{\prime \prime}(x) \chi(y) y \operatorname{Im} \hatm(x+\mathrm{i} y)
            \\
            &= \frac{1}{2 \pi} \int \myddiff  x \int_{\left|y\right| > \eta} \myddiff  y f^{ \prime}(x) \chi(y) y \partial_x\operatorname{Im} \hatm(x+\mathrm{i} y) \\
            &= -\frac{1}{2 \pi} \int \myddiff  x \int_{\left|y\right| > \eta} \myddiff  y f^{ \prime}(x) \chi(y) y \partial_y\operatorname{Re} \hatm(x+\mathrm{i} y)
        \end{split}
    \end{equation*}
    where in the last line we used the holomorphy of \(\hatm\), and then integrating by parts over \(y\), 
    \begin{equation*}
        \begin{split}
            &= -\frac{1}{2 \pi} \int \myddiff  x \int_{\left|y\right| > \eta} \myddiff  y f^{ \prime}(x) \chi(y) y \partial_y\operatorname{Re} \hatm(x+\mathrm{i} y) \\
            &= -\frac{1}{2 \pi} \int \myddiff  x f^{ \prime}(x)\left(\left[y\chi(y) \operatorname{Re} \hatm(x+\mathrm{i} y)\right]_{\eta}^{-\eta}- \int_{\left|y\right| > \eta} \myddiff  y(\chi(y) +y\chi'(y)) \operatorname{Re} \hatm(x+\mathrm{i} y)\right) 
        \end{split}
    \end{equation*}
    Continuing, 
    \begin{equation*}
        \begin{split}
            \left|\cdot\right|&\lessapprox  \int \myddiff  x \left|f^{ \prime}(x)\right|\left|O(\eta)- \int_{\left|y\right| > \eta} \myddiff  y(\chi(y)\operatorname{Re} \hatm(x+\mathrm{i} y) +y\chi'(y)\operatorname{Re} \hatm(x+\mathrm{i} y)) \right| \\
            &\lessapprox  \int \myddiff  x \left|f^{ \prime}(x)\right|\left(O(\eta)+ \int_{\eta<\left|y\right|<2} \myddiff  y \chi(y)\left|\operatorname{Re} \hatm(x+\mathrm{i} y)\right|  + \int_{1<\left|y\right|<2}y\left|\operatorname{Re} \hatm(x+\mathrm{i} y)\right| \right) \\
            & \lessapprox  \int \myddiff  x \left|f^{ \prime}(x)\right|\left(O(\eta) + \int_\eta^2 \myddiff y \frac{\eta}{y} + O\left(\eta\right) \right)\\
            & \lessapprox  \int \myddiff  x \left|f^{ \prime}(x)\right|\left(O(\eta) - \eta\log \eta \right) \\
            & \lessapprox \eta \log\eta \int \myddiff  x \left|f^{ \prime}(x)\right|
        \end{split}
    \end{equation*}
    with high probability, where we have used \eqref{S-eq:maxineq} in the middle step.  Thus, we have established equations \eqref{S-eq:hs_first_term_large_y}, \eqref{S-eq:hs_first_term}, and ultimately, \eqref{S-eq:hs_main_result_eig}.

    The proof is exactly the same when \(\mu, \umu\) are replaced with \(\nu, \ulpmeas\).

\end{proof}

\section{$L^1$ Error in the \LW{} Eigenvalues} \label{app:lw-rate}

In what follows, continuing to use our convention of suppressing subscripts of \(n\), \(L^1(\mu)\) will denote the \(n\)-dependent norm \(L^1(\mu_n)\), and \(\mywt(x)\) will denote \(\mywtnx\), for example.

The goal is to approximate \(\delta(x) =x/\{[1-\asp-\asp x\Hw(x)]^2+\asp^2 x^2 \pi^2 \myw(x)^2\}\) in the  \(L^1(\mu)\)-norm to order \(O_\prec(n^{-2/3})\).  Since \(\delta(x)\) is bounded below on a neighborhood of \(F\) by Section~\ref{app:bounded-delta}, it will suffice to show \(\lv \Hwt - \Hw\rv_{L^1(\mu)}\) and \(\lv \mywt - \myw\rv_{L^1(\mu)}\) are both \(O_\prec(n^{-2/3})\).  Throughout this and later appendices, \(\Delta\) will stand for \(n^{-1/3}\). 

For the sake of brevity, we will alternatively denote \(\mywt(x)\) to be the additive convolution \(\varphi_\Delta * \mu\), where \(\varphi_\Delta(x)\) is the even function \(\Delta^{-1} k(x/\Delta)\),  and \(\mu\) is again the empirical spectral measure.  
This marks a slight departure from the multiplicative convolution with the approximation to the identity \(x\mapsto \Delta^{-1}k((x-1)/\Delta)\) appearing in the theorem, but the proofs are identical to order \(\Delta^2\)  after appropriate logarithmic transformations and Hilbert-transform invariances are applied.  The reason the multiplicative convolution was originally proposed is mainly a practical one: the smallest estimated eigenvalue converges more quickly as \(n\to\infty\) in the multiplicative formulation.

Our first main result of this section is the following. 

\begin{prop} \label{prop:prop1}
Assume the conditions of Theorem~4.5.
Then $\wmu$ satisfies
\begin{equation*}
\frac{1}{p}\sum_{i=1}^p|\wmu(\lambda_i)-w(\lambda_i)| \prec n^{-2/3}.
\end{equation*}
(In other words, \(\lv \mywt - \myw\rv_{L^1(\mu)} \prec n^{-2/3}\) .)
\end{prop}

We will use the decomposition, with $\wdel:=\varphi_{\mysmall}*d\umu$,
\begin{align} \label{S-eq:main-decomp}
    |\myw-\wmu| & \le |\myw-\wdel| + |\varphi_{\mysmall}*(d\umu-d\mu)|,
\end{align}
\blue{
first considering the $L^1(\mu)$ norm of second term in the decomposition.


\begin{lem} \label{lem:lem1} Assume the conditions of Theorem~4.5.
Then
\[
\frac{1}{p}\sum_{i=1}^p |\varphi_{\mysmall}*(d\umu-d\mu)|(\lambda_i) \prec \mysmall^{-1}p^{-1}
\]
\end{lem}
}
\begin{proof}
By Lemma~\ref{lem:hs_main_result}, the second term of the right-hand side \bdradd{of \eqref{S-eq:main-decomp}} is 
\begin{equation*}
    \prec n^{-1} \left\lVert \varphi_\Delta \right\rVert _1
    + n^{-1} \left\lVert \partial_x \varphi_\Delta \right\rVert _1
    + n^{-2} \left\lVert \partial_x^2 \varphi_\Delta \right\rVert _1.
\end{equation*}
We only consider the case where  \(\varphi\) and \(\varphi'\) have no more than 3 monotonic intervals. The fundamental theorem of calculus then yields 
\begin{equation*} 
    \begin{split}
        \prec n^{-1} + 6n^{-1} \left\lVert \varphi_\Delta \right\rVert _\infty + 6n^{-2} \left\lVert \partial_x \varphi_\Delta \right\rVert _\infty\\
        \prec n^{-1} + 6 n^{-1} \Delta^{-1} + 6 n^{-2} \Delta^{-2} \prec \Delta^{-1} n^{-1}
    \end{split}
\end{equation*}
so that 
\begin{equation} \label{S-eq:w_bound_law_part}
    \frac{1}{p}\sum_{i=1}^p|w_\Delta(\lambda_i)-\wmu(\lambda_i)| \prec \Delta^{-1} n^{-1}.
\end{equation}
\end{proof}

We will now focus on the first term \blue{in the decomposition of \eqref{S-eq:main-decomp}, along with a similar result for $\Hw$}.  

\begin{lem} \label{lem:lem2}
Assume the conditions of Theorem~4.5.
 Then
\begin{align*}
\text{\emph{(a)}} \qquad & \frac{1}{p} \sum_{i=1}^p |\myw(\lambda_i) - \myw_{\mysmall}(\lambda_i)|  \prec \Delta^2, \qquad \text{\emph{and}}
\\
\text{\emph{(b)}} \qquad & 
\frac{1}{p} \sum_{i=1}^p |\Hw(\lambda_i) - \Hw_{\mysmall}(\lambda_i)|  \prec \Delta^2.
\end{align*}

\end{lem}

\begin{proof}
\blue{
By Lemma~\ref{lem:van-wHw}, we have $\myw^{(k)}(x)$ exists and $\myw^{(k)}(x) \lessapprox \kappa(x)^{1/2-k}$ if $x\not\in \partial F$, and we get the same properties for $\Hw(x)$.  
 Thus, letting $a(x)$ be a smooth function on $\bbR\backslash(\partial F)$ that satisfies $\myfunc^{(k)}(x) \lessapprox \kappa(x)^{1/2-k}$, and defining $a_{\mysmall} = \varphi_{\mysmall}*a$, 
  the following is sufficient for both (a) and (b):
  \begin{equation*} 
\frac{1}{p} \sum_{i=1}^p |\myfunc(\lambda_i) - \myfunc_{\mysmall}(\lambda_i)| \prec \Delta^2.
\end{equation*}
  

To this end, for $\kappa(x) \le 2\mysmall$} we have
\begin{equation*} \label{S-eq:kappasmall}
    \begin{split}
        |\myfunc(x) - \myfunc_{\mysmall}(x)| & \le \int_{x-\mysmall}^{x+\mysmall}\varphi_{\mysmall}(x-t)|\myfunc(x)-\myfunc(t)|\, dt \\ 
        & \lessapprox
        \int_{x-\mysmall}^{x+\mysmall}\left|\varphi_{\mysmall}(x-t)\right|
        \mysmall^{-1/2}\mysmall\, dt \\
        & \leq \mysmall^{1/2},
    \end{split}
\end{equation*}
where we have applied the mean value theorem to $|\myfunc(x)-\myfunc(t)|$ and used $\myfunc'(x)\lessapprox \kappa(x)^{-1/2} \lessapprox \Delta^{-1/2}$.
Therefore, we have
\begin{equation} \label{S-eq:w_bound_conv_part_small}
    \frac{1}{p} \sum_{\kappa(\lambda_i) \leq 2\mysmall} |\myfunc(x) - \myfunc_{\mysmall}(x)| \leq \left( C (2\mysmall)^{3/2} + p^{\epsilon -1} \right) \Delta^{1/2} \lessapprox \Delta^2
\end{equation}
with high probability, where we estimated the number of terms in the sum using the small-scale Mar\v{c}enko--Pastur law (second clause of Theorem~3.3).

For \(\kappa(x) > 2\mysmall\), Taylor-expanding
$\myfunc(t)$ about \(t=x\) and using evenness of \(\varphi\) gives that 
\begin{equation} \label{S-eq:kappa-lg}
    \begin{split}
        \myfunc(x) - \myfuncsmall(x) &= \int_{x - \Delta}^{x + \Delta} \varphi_\Delta(x - t) (\myfunc(x) - \myfunc(t)) \, \myddiff t \\
        &= \int_{x - \Delta}^{x + \Delta} \varphi_\Delta(x - t) \\
        & \quad \cdot \left(\myfunc'(x) (x-t) - \frac{\myfunc''(x)}{2} (x - t)^2 + O(\Delta^3\max_{[x-\Delta, x+\Delta]}\myfunc''')\right) \, \myddiff t \\
        &\lessapprox \Delta^2 (\kappa(x) - \Delta)^{-3/2} \\
        &\lessapprox \Delta^2 \kappa(x)^{-3/2} 
    \end{split}
\end{equation}
where we have used 
the bound \(\left|\myfunc^{(k)}\right| \lessapprox \kappa^{1/2 - k}\) for any fixed \(k\ge 1\). 

Therefore, we have 
\begin{equation*}
    \frac{1}{p} \sum_{\kappa(\lambda_i) > 2\mysmall} |\myfunc(\lambda_i) - \myfunc_{\mysmall}(\lambda_i)| \lessapprox \frac{1}{p} \sum_{\kappa(\lambda_i) > 2\mysmall} \Delta^2 \kappa(\lambda_i)^{-3/2}
\end{equation*}
The right hand side is precisely 
\begin{equation*}
    \int \widetilde{f} \myddiff  \mu,
\end{equation*}
where \(\widetilde{f}(x) := \Delta^2 \kappa(x)^{-3/2}\boldsymbol{1}_{\kappa(x) > 2\mysmall}\). The only difficulty in applying Lemma \ref{lem:hs_main_result} is that \(\widetilde{f}\) is discontinuous, so we must adjust it. Let \(g:\mathbb{R} \to [0,1]\) be \(C^\infty\), equal \(0\) for \(x \geq 2\), equal \(1\) for \(x \leq 1\) and satisfy \(\left\lVert g' \right\rVert _\infty + \left\lVert g'' \right\rVert _\infty \leq C\). Then define
\begin{equation*}
    f(x) = \widetilde{f}(x) g((2\mysmall)^{-1} \kappa(x)),
\end{equation*}
which is now \(C^2\) and satisfies 
\begin{equation*}
    \begin{split}
        \left\lVert f \right\rVert _1 +
        \left\lVert f' \right\rVert _1 +
        \left\lVert f'' \right\rVert _1 &\leq 1.
    \end{split}
\end{equation*}
Therefore Lemma \ref{lem:hs_main_result} gives
\begin{equation} \label{S-eq:w_bound_conv_part_big}
    \begin{split}
        \frac{1}{p} \sum_{\kappa(\lambda_i) > 2\mysmall} |\myfunc(x) - \myfunc_{\mysmall}(x)| &\lessapprox \frac{1}{p} \sum_{\kappa(\lambda_i) > 2\mysmall} \Delta^2 \kappa(\lambda_i)^{3/2} \\
        &\lessapprox\int \widetilde{f}(x) \myddiff  \mu(x) \\
        &= \int \widetilde{f}(x) \myddiff  \umu(x) + O_\prec(n^{-1})\\
        &\lessapprox \int \widetilde{f}(x) \kappa(x)^{1/2} \myddiff  x + O_\prec(n^{-1}) \\
        &  \lessapprox \mysmall^2 \int_{\kappa(x)>2\mysmall} \kappa(x)^{-3/2}\kappa(x)^{1/2}\, dx + O_\prec(n^{-1}) \\
        &\asymp \Delta^2 \log|\Delta| \prec \Delta^2
    \end{split}
\end{equation}
with high probability. 


\end{proof}

We may now quickly prove the Proposition above.

\begin{proof}[Proof of Proposition~\ref{prop:prop1}]

Equations \eqref{S-eq:w_bound_conv_part_small} and \eqref{S-eq:w_bound_conv_part_big} together establish that 
\begin{equation*}
    \frac{1}{p}\sum_{i=1}^p|\mywt(\lambda_i)-w_\Delta(\lambda_i)| \prec \Delta^2,
\end{equation*}
which, with equation \eqref{S-eq:w_bound_law_part} gives
\begin{equation*}
    \frac{1}{p}\sum_{i=1}^p|w(\lambda_i)-\mywt(\lambda_i)| \prec \Delta^2 + \Delta^{-1} n^{-1}.
\end{equation*}
Thus the choice of \(\Delta = n^{-1/3}\) becomes clear.. 

\end{proof}

The Hilbert transform \(\mathcal{H}w\) may be treated in much the same way as \(w\). 
\begin{prop} \label{prop:prop2}
Assume the conditions of Theorem~4.5.
Then $\Hwmu$ satisfies
\begin{equation*}
    \frac{1}{p}\sum_{i=1}^p|\Hwmu(\lambda_i)-\Hw(\lambda_i)| \prec n^{-2/3}.
\end{equation*}
(In other words, \(\lv\htrans\mywt - \Hw \rv_{L^1(\mu)} \prec n^{-2/3}\) .)
\end{prop}

\begin{proof}
We decompose:
\begin{equation*}
    \begin{split}
        \lvert \mathcal{H}w - \htrans \mywt \rvert &\leq \lvert \mathcal{H}w - \mathcal{H}w_\Delta \rvert + \lvert \mathcal{H}w_\Delta - \Hwt \rvert
    \end{split}        
\end{equation*}
Using the relation between \(\mathcal{H}\) and convolution, the above may be written as 
\begin{equation} \label{S-eq:Hw-triangle}
    \begin{split}
        \lvert \mathcal{H}w - \varphi_\Delta * (\mathcal{H}w) \rvert + \lvert (\mathcal{H}\varphi_\Delta) * \umu - (\mathcal{H} \varphi_\Delta) * \mu \rvert
    \end{split}
\end{equation}
The latter term may be written, letting \(\hat{\varphi} := \mathcal{H} \varphi\), as
\begin{equation*}
    \lvert (\mathcal{H}\varphi_\Delta) * \umu - (\mathcal{H} \varphi_\Delta) * \mu \rvert = \lvert \hat{\varphi}_\Delta * \umu - \hat{\varphi}_\Delta * \mu \rvert
\end{equation*}
Let \(f = \hat{\varphi}_\Delta \cdot g\), where \(g\) is \(C^{\infty}\), equals 1 on \([-C, C]\), equals 0 on \(\mathbb{R} \setminus [-2C, 2C]\), and has bounded first and second derivatives, for some large constant \(C\). 
As in the proof of Lemma~\ref{lem:lem1}, we use Lemma~\ref{lem:hs_main_result}: uniformly in \(x \leq C/2\), 
\begin{equation} \label{S-eq:big-LW-split}
    \begin{split}
        \left|\hat{\varphi}_\Delta * \umu(x) - \hat{\varphi}_\Delta * \mu(x)\right| &= \left|f * \umu(x) - f * \mu(x)\right|\\
        &\prec n^{-1} \left\lVert f\right\rVert _1
        + n^{-1} \left\lVert f'\right\rVert _1
        + n^{-2} \left\lVert f''\right\rVert _1
    \end{split}
\end{equation}
(the appearance of \(\left\lVert f\right\rVert _1\) is why we needed to attenuate \(\hat{\varphi}_\Delta\) by a cutoff function: \(\mathcal{H}\varphi\) is not \(L^1\); we may do this because \(\umu\) and \(\mu\) have uniformly bounded support with high probability). Now we may conclude similarly to in the argument leading to \eqref{S-eq:w_bound_law_part}. We have that \( \left\lVert f\right\rVert _1 \lessapprox |\log \Delta|\), and that \(f\) and \(f'\) may be ensured to have a bounded number of monotonic intervals, so that \(\left\lVert f'\right\rVert _1 \lessapprox \left\lVert f\right\rVert _\infty \lessapprox \Delta^{-1}\) and \(\left\lVert f''\right\rVert _1 \lessapprox \left\lVert f'\right\rVert _\infty \lessapprox \Delta^{-2}\).  

\blue{As a result of the above, $\Vert\Hw - \Hwt\Vert_{L^1(\mu)}$
can be broken into $O_\prec (n^{-1}\mysmall^{-1})$
and the following 
\begin{equation*}
     \frac{1}{p} \sum_{i=1}^p \left| \mathcal{H}w(\lambda_i) - \varphi_{\mysmall}*(\mathcal{H}w)(\lambda_i)\right|,
\end{equation*}
which is $O_\prec(\Delta^2)$ by Lemma~\ref{lem:lem2}(b). 
%
Thus, choosing $\mysmall = n^{-1/3}$ as before completes the proof.}

\end{proof}

Combining the Propositions yields precisely the goal named at the beginning of this section, that \(\lv\mywt - \myw \rv_{L^1(\mu)} \prec n^{-2/3}\) and \(\lv\htrans\mywt - \Hw \rv_{L^1(\mu)} \prec n^{-2/3}\).

\section{Deterministic Limiting Variance} \label{app:alternating-trace}

We choose the somewhat more general task of estimating \(p^{-1}\tr_I(\myXi_n f_n(\bSn)\bfrn g_n(\bSn))\) for some matrix \(\myXi_n\) having 
'\begin{equation*} \label{S-eq:xilocal}
p^{-1}\sum_{i=1}^p \frac{\buni\tps \myXi_n\buni}{\lamnind - z} = \int \frac{\xi_\infty(x)}{x-z}\, dx + O_\prec\left(\frac{1}{n\eta}\right)
\end{equation*}
for some bounded measurable \(\xi_\infty:\bbR\to\bbR\) and all \(z\in \Cplus\). By the Ledoit--P\'ech\'e theorem (Theorem~4.1)
and Section~\ref{sec:app-HS}, this is certainly true of \(\bfrn\)---the value of \(\myXi_n\) that is of most interest---in which case \(\xi_\infty(x)=\delta(x)\).

For this section, we will need the notation \(\OPtilde(\mydot)\), which is the in-probability analogue of \(O_\prec(\mydot)\).  In other words, writing \(a_n =  \OPtilde(b_n)\) means for every \(\epsilon > 0\), 
\[
\Pr[|a_n | \ge n^\epsilon b_n ] \to 0,
\]
as \(n\to\infty\).

The strategy is to first estimate the alternating trace above for \(f\) and \(g\) equal to resolvent functions, then use the usual limiting argument as complex arguments go the real axis to go from resolvent functions to continuous functions.  For the remainder of this section, we will resume following the convention of suppressing most subscripts of \(n\).

Assume as a first case that \(f(\lambda) \equiv f_n(\lambda)=(\lambda-x_z-i\eta_z)^{-1}\) and \(g(\lambda)\equiv g_n(\lambda) = (\lambda-x_{\myclx}-i\eta_{\myclx})^{-1}\) for positive and distinct \(\eta_x\) and \(\eta_z\), and let
\(z=x_z+i\eta_z\) and \(\myclx=x_{\myclx}+i\eta_{\myclx}\). Our first main goal will be to prove the following lemma:
\begin{lem} \label{lem:multi-resolv} Assume \emph{[\textsc{Train1}]-[\textsc{Train4}]}.  Then
\begin{equation} \label{S-eq:tracio-denom-target-1}
\begin{split}
& p^{-1}\tr\left(\myXi f(\bS) \bfr g(\bS) \right) 
\\
& =  p^{-1}\sum_{i=1}^{p} \frac{\lambda_i\bu_i\tps \myXi\bu_i}{(\lambda_i-z)(\lambda_i-\myclx)}\left(1+\frac{1}{n}\sum_{i=1}^{p} \frac{\bu_i\tps\bfr\bu_i}{\lambda_i-z}\right)\left(1+\frac{1}{n}\sum_{i=1}^{p} \frac{\bu_i\tps\bfr\bu_i}{\lambda_i-\myclx}\right)
\\
&
+ \OPtilde\left( \frac{1}{n\eta_z \eta_{\myclx}} \right).
\end{split}
\end{equation}
\end{lem}
\noindent This lemma can be improved to high-probability convergence using the recent multi-resolvent laws for sample covariance matrices \cite{lin2026eigenvector}, but we reproduce the proof here with the weaker \(\OPtilde\) mode of convergence for completeness.

Let   \(\bc_{k} = n^{-1/2}\bfrhalf\bz_k\) and \(\bRk = (\bR(z)^{-1}-\bc_k\bc_k\tps)^{-1}\).
First, we prove a lemma involving the average of the recurring estimation error \(\epsilon_k\), where
\[
\epsilon_k := \mathbf{c}_{k}\tps\resolv^{(k)}(z)\mathbf{c}_{k}-\frac{\mathrm{tr}(\bfr \resolv(z))}{n}.
\]
\begin{lem} \label{lem:ip-rate}
Assume \emph{[\textsc{Train1}]-[\textsc{Train4}]}, and define \(\eta := \eta_z\) and \(\epsilon_k\) as above.  Then we have
\begin{align*}
&
\mathcal{E}_p := \frac{1}{p}\sum_{k=1}^p \epsilon_k = O_{\tilde{P}}\left( \frac{1}{n\eta } \right). \label{S-eq:Ep}
\end{align*}
\end{lem}
 \begin{proof}
 Let \(\tilde{\epsilon}_k = \mathbf{c}_{k}\tps\resolv^{(k)}(z)\mathbf{c}_{k}-n^{-1}\mathrm{tr}(\bfr \resolv^{(k)}(z))\).
 Since we have from \cite[Lemma~2.6]{silverstein1995empirical} that
 \[
 \left|\frac{\tr(\bfr \bR(z))}{n}-\frac{\tr(\bfr\bRk)}{n}\right| \le \frac{\lv \bfr\rv}{n\eta},
 \]it suffices to prove that \(\tilde{\mathcal{E}}_p := p^{-1}\sum_{k=1}^p \tilde{\epsilon}_k =  \OPtilde(n^{-1}\eta^{-1})\).
 
First we prove, uniformly in \(k\), 
\begin{align}
|\tilde{\epsilon}_k|^2 \prec \frac{1}{n\eta}. \label{S-eq:epsksqone}
\end{align}
 By Hanson--Wright \cite{rudelson2013hanson},
 \begin{align*}
 \Pr\left[ |\tilde{\epsilon}_k| \ge t \mid \bRk \right] \le 2\exp\left[-c\min\left(\frac{n^2 t^2}{C^4\lv \bfrhalf \bRk \bfrhalf \rv^2_\hs}, \frac{n t}{C^2 \lv \bfr \bRk \rv} \right)\right],
 \\
 \end{align*}
 for some absolute \(c>0\) and some \(C>0\) depending only on the moments of \(W\).
 Thus, we would like to estimate the Hilbert-Schmidt norm above: namely,
 \begin{align}
     \tr\left(\bfr \bRk \bfr \bRk^* \right)  & 
      \le \lv \bfr \rv \tr\left(\bRk\bfr\bR^k(\overline{z}))\right) 
      \label{S-eq:hs-resolv}
      \\
      & \lessapprox \sum_{k=1}^p \frac{\bu_k\tps \bfr \bu_k}{|\tilde{\lambda}_k-z|^2} 
      \nonumber
      \\
      & = \frac{p}{i\eta}\left( \myTheta^{(k)}(z) - \myTheta^{(k)}(\overline{z})\right), \nonumber
 \end{align}
 where \((\mydot)^*\) denotes the Hermitian transpose, \(\tilde{\lambda}_k\) and \(\tilde{\bu}_k\) are the eigenvalues/vectors, and \(\Theta^{(k)}(z)\) is the \lp{} analytic function, corresponding to \(\bS-\bc_k\bc_k\tps\).    Since, as discussed in Section~4.2,
 \(\myTheta^{(k)}(z)\) has a limiting value as \(z\to x \in \mathbb{R}\backslash \{0\}\), the left-hand side of \eqref{S-eq:hs-resolv} has order at most \(n/\eta\) as \(n\to\infty\), almost surely.  It can also be seen that the spectral norm has order at most \(1/\eta\).  Thus, taking \(t\) to be \(n^\epsilon/\sqrt{ \eta n}\) and smoothing gives \eqref{S-eq:epsksq}.  For the sake of the analysis that follows, we also note that by the (conditional) layercake theorem we get
 \begin{align}
& \bbE[|\tilde{\epsilon}_k|^2 \mid \bRk] \prec \frac{1}{n\eta }. \label{S-eq:epsksq}
\end{align}
 
 Consider terms of the form \(\tilde{\epsilon}_j\overline{\tilde{\epsilon}}_k\), where \(\overline{(\, \mydot\, )}\) denotes the complex conjugate operation and \(j\ne k\).
  Using the Woodbury formula, we obtain
  \[
  \bRj = \bRjk - \frac{\bRjk\bc_k\bc_k\tps\bRjk}{1+\bc_k\tps\bRjk\bc_k},
  \]
   where \(\bRjk = (\bRj^{-1} - \bc_k\bc_k\tps)^{-1}\).
   Using steps similar to \eqref{S-eq:epsksq}, we obtain
   \begin{align} \label{S-eq:Rjk}
   \bbE[\tilde{\epsilon}_j\overline{\tilde{\epsilon}}_k \mid \bRjk ] = O_\prec\left( \frac{1}{n^2 \eta^2}\right).
   \end{align}
   As a result, by applying Markov and smoothing \eqref{S-eq:epsksq} and \eqref{S-eq:Rjk}, we get
   \begin{align*}
   \Pr\left[ \lvert \tilde{\mathcal{E}}_p \rvert \ge t \right] & \le \frac{1}{t^2 p^2}  \sum_{j,k=1}^p   \bbE[\tilde{\epsilon}_j\overline{\tilde{\epsilon}}_k] 
   \\
   & = O\left(\frac{n^{\epsilon}}{t^2 n^2\eta} + \frac{n^\epsilon}{t^2 n^2 \eta^2} \right). 
   \end{align*}
   for any \(\epsilon > 0\).
   Taking \(t = n^{\epsilon}/(n\eta)\) gives the desired result for \(\tilde{\mathcal{E}}_p\), and thus, as remarked earlier in the proof, for \(\mathcal{E}_p\).
 \end{proof}

We now turn to the multi-resolvent law stated earlier.

\begin{proof}[Proof of Lemma~\ref{lem:multi-resolv}] We first observe the following:
\begin{align}
  p^{-1}\mathrm{tr}(\resolv(z)\myXi \resolv(\myclx)) & =p^{-1}\mathrm{tr}(\resolv(\myclx)\resolv(z)\myXi)\nonumber \\
 & =p^{-1}\sum_{i=1}^{p}\frac{{\bu_i\tps \myXi \bu_i}}{(\lambda_{i}-z)(\lambda_{i}-\myclx)} ,\label{S-eq:alt-trace-init}
\end{align}
Consider 
\begin{equation} \label{S-eq:S-trick}
p^{-1}\mathrm{tr}(\myXi \resolv(\myclx))+zp^{-1}\mathrm{ tr}(\resolv(z)\myXi \resolv(\myclx)).
\end{equation}
Since $\bI+z\resolv(z)=\bS\resolv(z)$, we can write the above as
\begin{align*}
 & p^{-1}\mathrm{tr}((\bI+z\resolv(z))\myXi \resolv(\myclx))\\
 & =p^{-1}\mathrm{tr}(\bS\resolv(z)\myXi \resolv(\myclx)).
\end{align*}
Writing $\bS=\sum_{k=1}^{n}\bc_{k}\bc_{k}\tps$, we get the above is
\begin{align}
 & p^{-1}\sum_{k=1}^{n}\bc_{k}\tps\resolv(z)\myXi \resolv(\myclx)\bc_{k}.\label{S-eq:before-woodbury}
\end{align}
Appling the Woodbury formula again, we obtain
 \begin{align*} 
     \bR(z) = \bRk - \frac{ \bRk\bc_k\bc_k\tps \bRk}{1+\bc_k\tps \bRk \bc_k},
 \end{align*}
 which gives the following expansion of \eqref{S-eq:before-woodbury}:
\begin{align}
 & p^{-1}\sum_{k=1}^{n}\bc_{k}\tps\resolv^{(k)}(z)\myXi \resolv^{(k)}(\myclx)\bc_{k}+\nonumber \\
 & -p^{-1}\sum_{k=1}^{n}\bc_{k}\tps\resolv^{(k)}(z)\myXi \resolv^{(k)}(\myclx)\bc_{k}\left(1-\frac{1}{1+\bc_{k}\tps\resolv^{(k)}(\myclx)\bc_{k}}\right)+\nonumber \\
 & -p^{-1}\sum_{k=1}^{n}\bc_{k}\tps\resolv^{(k)}(z)\myXi \resolv^{(k)}(\myclx)\bc_{k}\left(1-\frac{1}{1+\bc_{k}\tps\resolv^{(k)}(z)\bc_{k}}\right)+\nonumber \\
 & +p^{-1}\sum_{k=1}^{n}\bc_{k}\tps\resolv^{(k)}(z)\myXi \resolv^{(k)}(\myclx)\bc_{k}\frac{\bc_{k}\tps\resolv^{(k)}(z)\bc_{k}}{1+\bc_{k}\tps\resolv^{(k)}(z)\bc_{k}}\frac{\bc_{k}\tps\resolv^{(k)}(\myclx)\bc_{k}}{1+\bc_{k}\tps\resolv^{(k)}(\myclx)\bc_{k}}\nonumber \\
 & =p^{-1}\sum_{k=1}^{n}\frac{\bc_{k}\tps\resolv^{(k)}(z)\myXi \resolv^{(k)}(\myclx)\bc_{k}}{\left(1+\bc_{k}\tps\resolv^{(k)}(z)\bc_{k}\right)\left(1+\bc_{k}\tps\resolv^{(k)}(\myclx)\bc_{k}\right)}.\label{S-eq:expanded-trace}
\end{align}

By calculations similar to the ones in Lemma~\ref{lem:ip-rate},
\[
\frac{1}{p}\sum_{k=1}^p\left[\mathbf{c}_{k}\tps\resolv^{(k)}(z)\myXi \resolv^{(k)}(\myclx)\mathbf{c}_{k}-\frac{\mathrm{tr}(\bfr \resolv(z)\myXi \resolv(\myclx))}{n}\right] = \OPtilde \left(\frac{1}{\eta_{z}\eta_{\myclx}n}\right).
\]
This together with \eqref{S-eq:epsksqone} (and the technicality of extending and applying the boundedness clause of Section~\ref{app:bounded-delta} several times) gives the following estimate for (\ref{S-eq:expanded-trace}):
\[
\frac{1}{p}\frac{\mathrm{tr}(\bfr \resolv(z)\myXi \resolv(\myclx))}{\left(1+\frac{1}{n}\mathrm{tr}(\bfr \resolv(z))\right)\left(1+\frac{1}{n}\mathrm{tr}(\bfr \resolv(\myclx))\right)}+\OPtilde\left(\frac{1}{\eta_{z}\eta_{\myclx} n}\right).
\]
Comparing to (\ref{S-eq:alt-trace-init}) and \eqref{S-eq:S-trick} gives
\begin{align*}
 & \frac{1}{p}\sum_{i=1}^p \frac{\bu_i\tps \myXi \bu_i}{\lambda_i-\myclx}+z\left(p^{-1}\sum_{i=1}^{p}\frac{{\bu_i\tps \myXi \bu_i}}{(\lambda_{i}-z)(\lambda_{i}-\myclx)}\right)\\
  & =\frac{1}{p}\frac{\mathrm{tr}(\bfr \resolv(z)\myXi \resolv(\myclx))}{\left(1+\frac{1}{n}\mathrm{tr}(\bfr \resolv(z))\right)\left(1+\frac{1}{n}\mathrm{tr}(\bfr \resolv(\myclx))\right)}+\OPtilde\left(\frac{1}{\eta_{z}\eta_{\myclx} n}\right)
\end{align*}
so that, using the asymptotic boundedness of $1+\frac{1}{n}\mathrm{tr}(\bfr \resolv(z))$, which follows from the boundedness clause in Section~\ref{app:bounded-delta},
we have
\begin{align*}
 & \frac{1}{p}\mathrm{tr}(\bfr \resolv(z)\myXi \resolv(\myclx)) \nonumber \\
 & =\left(1+\frac{1}{n}\sum_{i=1}^{p}\frac{\bu_i\tps\bfr\bu_i}{\lambda_{i}-z}\right)\left(1+\frac{1}{n}\sum_{i=1}^{p}\frac{\bu_i\tps\bfr\bu_i}{\lambda_{i}-\myclx}\right)p^{-1}\sum_{i=1}^{p}\frac{\lambda_{i}\bu_i\tps \myXi\bu_i}{(\lambda_{i}-z)(\lambda_{i}-\myclx)}
 \\
 & +\OPtilde\left(\frac{1}{\eta_{z}\eta_{\myclx} n}\right),
 \nonumber
\end{align*}
as desired.
\end{proof}

We now find a deterministic limiting expression for \eqref{S-eq:tracio-denom-target-1}. Applying the argument of Section~\ref{sec:app-HS} to the real and imaginary parts of 
\[
n^{-1}\sum_i \frac{\bu_i\tps\bfr\bu_i}{\lambda_i-z}
\]
and using linearity allows us to replace this summation by \(z\mapsto\asp \int d\ulpmeas(\lambda)/(\lambda-z)\), with a negligible error of \(1/(n\min\{\eta_z,\eta_{\myclx}\})\).  Similarly, using partial fractions on the remaining sum and assuming \(\liminf_n |\eta_z/\eta_{\myclx} -1 | \ne 0\) allows us to replace this sum by 
\[
\int \frac{\lambda \xi(\lambda)}{(\lambda-z)(\lambda-\myclx)}w(\lambda)\, d\lambda + O_\prec\left(\frac{1}{\eta_{z}\eta_{\myclx} p}\right).
\]
Expanding the resulting product of three integral expressions and using Fubini's theorem and some algebra gives
\begin{equation}
    \begin{split}
& p^{-1}\tr\left(\myXi f(\bS) \bfr g(\bS)\right)
\label{S-eq:four-trace}
\\
& = \int  f(x)g(x) x \xi(x)\, d\umu(x) 
    \\
    & - \asp \int_{\mathbb{R}^2}   f(x) \frac{g(y)-g(x)}{y-x}\, x\xi(x)\, d\umu(x)d\ulpmeas(y) 
    \\
    & - \asp \int_{\mathbb{R}^2} g(x) \frac{f(y)-f(x)}{y-x}\, x\xi(x)\, d\umu(x)d\ulpmeas(y) 
    \\
    & + \asp^2 \int_{\mathbb{R}^3} \frac{g(u)-g(x)}{u-x} \frac{f(y)-f(x)}{y-x}\, x\xi(x)\, d\umu(x)d\ulpmeas(y)d\ulpmeas(u) + \OPtilde\left( \frac{1}{\eta_z\eta_{\myclx} n}\right). 
\end{split}
\end{equation}

Next, we claim that this identity holds for general regular functions \(\tilde{f}\equiv\tilde{f}_n\) and \(\tilde{g}\equiv\tilde{g}_n\).
For this, we simply integrate against \(\tilde{f}(\eta_z)\) and \(\tilde{g}(\eta_{\myclx})\) and let \(\eta_z\) and \(\eta_{\myclx}\) go to zero at a rate faster than \(n^{-1/3}\) but slower than \(n^{-1/2}\).  This results in the above equality with \(f\leftarrow \tilde{f}*\varphi_{\eta_z}\) and \(g \leftarrow \tilde{g}*\varphi_{\eta_{\myclx}}\), where \(\varphi\) is the Cauchy kernel.  Further, since \(\tilde{f}\) is regular, we have \(\tilde{f}*\varphi_\eta(x)  = \tilde{f}(x) + O_\prec(  n^{2/3}\eta^2)\) uniformly in \(x\), and so we replace these convolutions by \(\tilde{f}\) and \(\tilde{g}\), incurring only \(o_P(1)\) error.  The claim follows since our choice of \(\eta_z,\eta_{\myclx}\) implies that the error of \(O_P(n^{-1+\epsilon}\eta_z^{-1}\eta_{\myclx}^{-1})\) goes to zero in probability for sufficiently small \(\epsilon\).

Thus, we may assume \eqref{S-eq:four-trace} holds for general regular functions \(f\) and \(g\) and the right-hand side of this identity can be factored, giving
\begin{equation} \label{S-eq:def-sigxi}
\mysigma^2(f; \xi) := \int  [\myT f(x)]^2 x \xi(x)\, d\umu(x)
\end{equation}
since linearity of the Hilbert transform can be used to express \(T\) as
\[
\myT f(x) = f(x) - \asp \int \frac{f(y)-f(x)}{y-x} d\ulpmeas(y).
\]
In particular, \eqref{S-eq:four-trace} holds if the right-hand side is replaced by \(\mysigma^2(f)=\mysigma^2(f;\delta)\), as desired.


\section{Consistency of Variance Estimate} \label{app:consistent-variance}

\newcommand{\Topn}{T_n}

Let
 \begin{equation*} \label{S-eq:Hsubw}
\Hsubw f \equiv \Hsubw[f] := \fH[fw]
 \end{equation*}
 and 
\begin{equation*} \label{S-eq:Htildew}
   \Htildew f(x) \equiv  \Htildew[f](x) := p^{-1}\sum_{i=1}^p f(\lambda_i)K_{\Delta\lambda_i}(x-\lambda_i).
\end{equation*}
Intuitively, \(\Htildew f\) uses a kernel to approximate \(\Hsubw f\).
Rewriting \(\myT f\) as \((1+\asp \pi \Hw)f -\asp\pi \Hsubw[f\delta]\) and \(\myTn f\) as \((1+(p/n)\pi \htrans \tilde{w})f - (p/n)\pi \Htildew[f \lwEig]\) and expanding \((\myT f)^2\) and \((\myTn f)^2\) in the expressions for \(\sigma_\infty^2(f)\) and \((\sign)^2\), one of the main analytic steps is to show 
\[
\int \Htildew[f \lwEig](x)x \lwEig(x)\, d\mu_n(x) = \int \Hsubw[f\delta](x)x\delta(x)\, d\umu(x) + o_P(1).  
\]
The remaining terms in \(\sign(f)^2-\sigma_\infty^2(f)\) can be handled similarly.  

Using uniform convergence in probability of \(\lwEig(x)\) to \(\delta(x)\) on \(F\) guaranteed by one of Ledoit and Wolf's nonlinear shrinkage results (reproduced in Theorem~4.3),
all instances of \(\lwEig(x)\) can be replaced by \(\delta(x)\) without affecting the desired result.  \bdrdel{Since Section~\ref{app:bounded-delta} ensures \(\delta(x)\) is smooth and bounded on a neighborhood of \(F\),} The main task is essentially the same as proving
\[
\int \Htildew[f] \, d\mu_n = \int \Hsubw[f] \, d\umu+ o_P(1)
\]
for any regular function \(f\), with only slight modifications needed to account for the differently weighted measures.

We first note \(\int f\, d\mu_n\) is asymptotically equivalent to \(\int f\, d\mu_\infty\). Indeed, by Section~\ref{sec:app-HS}, the error in this approximation is of order
\[
\frac{\lv f \rv_1}{n} + \frac{\lv f'\rv_1}{n} + \frac{\lv f''\rv_1}{n^2} \prec \Delta^2
\]
since \(|f^{(k)}| \prec \Delta^k \) by regularity for \(k\le 3\).

We next show that \(\int \Htildew[f]\, d\mu_\infty\) approximates \(\int \Hsubw[f]\, d\mu_\infty\) in probability.
To simplify our analysis, as before, we make the replacement \( \Delta \lambda_{i} \leftarrow \Delta \) in \(K_{\Delta \lambda_{i}}\) above.  Similar to before, without this replacement \(\Htildew[f](x)\) is a convolution with respect to multiplication rather than addition, so the analyses do not significantly differ in character.

\begin{lem} \label{lem:Htildew-approx}
Assume \emph{[\textsc{Train}]} and that \(f\) is regular.  Then we have the following estimate
    \begin{align*}
        \lv \left(\Htildew[f] - \Hsubw[f]\right)w \rv_q^q \prec \Delta^2
    \end{align*}
    for any finite \(q \ge 1\).
\end{lem}
\begin{proof}
Using the notation \(K_\Delta(x) = \Delta^{-1}K(x/\Delta)\) and, we have \(\Htildew[f] = K_{\Delta}*(f\, d\mun)\).
Using the Helffer-Sj\"ostrand argument of Section~\ref{sec:app-HS} and the fact that  
\[
\lv(d/d\lambda)K_{\Delta}(x-\lambda)\rv_1 \lessapprox  1/\Delta
\qquad\text{and}\qquad \lv(d/d\lambda)^2 K_{\Delta}(x-\lambda)\rv_1 \lessapprox 1/\Delta^2,
\]
we get that \(\Htildew f\) can be approximated by
\begin{equation*}
\label{S-eq:continuous-expansion}
\int f(\lambda) K_{\Delta}(x-\lambda)w(\lambda)\, d\lambda = \int f(\lambda) K_{\Delta}(x-\lambda)w(\lambda) \varphi(\lambda)\, d\lambda
\end{equation*}
where \(\varphi\) is a smooth cutoff function with \(\varphi' \le \Delta^{-1}\) and \(\varphi'' \le \Delta^{-2}\) that takes the value 1 on \(\suppumu\) and 0 on \(\bbR\backslash (\suppumu+\Delta)\).
By Section~\ref{sec:app-HS},  then, the error in the above approximation of \(\Htildew f\)  can be written as
\begin{equation*}
    O_\prec\left( \frac{\lv fK\rv_{L^1(\suppumu)}}{n}
    + \frac{\lv f'K+ fK'/\Delta\rv_{L^1(\suppumu)}}{n} 
    + \frac{\lv f''K + f'K'/\Delta + f K''/\Delta^2\rv_{L^1(\suppumu)}}{n^2} \right),
\end{equation*}
which is \(O_\prec\left(\Delta^2\right)\) since \(f''\) is bounded and \(K, K',\) and \(K''\) are integrable on \(\suppumu\). 

We may now use the anti-self-adjointess of the Hilbert transform to obtain from the integral above
\[
\Hsubw[f] * k_{\Delta}(x) = \int  \htrans[fw](\lambda) k_{\Delta}(x-\lambda)\, d\lambda.
\]
Using continuity and/or the properties of approximate identities, this function converges in various modes to the desired function \(\Hsubw f\), but we have yet to establish the desired mode and rate of convergence.

Up to this point, we have shown that \(\Htildew [f]\) is an \(O_\prec(\Delta^2)\) approximation to \(\Hsubw[f]*k_\Delta\) with repect to the Lebesgue measure's infinity norm.  We next show that \(\Hsubw[f]*k_{\Delta}\) converges to \(\Hsubw f\) at a rate of \(O_\prec(\Delta^{2/q})\) in the \(L^q(\umu)\) norm for finite \(q\ge 1\).  That is, 
\begin{equation*} \label{S-eq:1-consistency}
\int |\Hsubw[f] * k_\Delta (\lambda) - \Hsubw f(\lambda)|^q w\, dx \prec \Delta^2.
\end{equation*}
This estimate follows from the fact that the integrand is eventually bounded and for \(q=1\) the left-hand side is bounded by \(O_\prec(\Delta^2|\log\Delta|) = O_\prec(\Delta^2)\) in exactly the same way as in Lemma~\ref{lem:hs_main_result}, since square-root behavior of \(\myw\) near the spectral edges implies square-root behavior of \(\htrans[fw]\) near the spectral edges.
\end{proof}

\section{Alternate significance level calculation} \label{sec:alt-sig-lev}

In this section, we expand upon the discussion at the end of Section~5.1
of the main text to formally demonstrate how to obtain an asymptotic significance level that is independent of sub-Gaussian constants.  



In order to assume as little as possible about the moment structure of $W$, we use a procedure from \cite[Section~3.3]{lopes2019bootstrapping} to estimate the only one which is needed: the fourth moment.  For notational tidiness, we denote the excess kurtosis $\bbE[W^4]-3$ by $\exkurt$.  A formula from \cite[Equation~9.8.6]{bai2010spectral} that holds under our independent-components and spectral assumptions, provided each component has a finite eighth moment,  is, in our notation,
\[
\exkurt = \frac{a_n - 2b_n}{c_n},
\]
where
\[
a_n = \mathrm{var}\left(\left\Vert \bx_1 \right\Vert_2^2 \right)
\]
and
\[
b_n = \left\Vert \bfr_n\right\Vert_{HS}^2
\]
and
\[
c_n = \sum_{j=1}^p \overline{\sigma}_j^4,
\]
where $(\overline{\sigma}_1^2, \overline{\sigma}_2^2, \dots \overline{\sigma}_p^2) = \mathrm{diag}(\bfr_n)$. As discussed in \cite[Section~3.3]{lopes2019bootstrapping}, rate-consistent estimates for the quantities appearing above are
\[
\hat{a}_n :=  \frac{1}{n-1}\sum_{i=1}^n \left( \left\Vert \bx_i\right\Vert^2 -\frac{1}{n}\sum_{j=1}^n \left\Vert \bx_j\right\Vert^2 \right)^2
\]
and
\[
\hat{b}_n :=  \mathrm{tr}(\bS_n^2) - \frac{1}{n} \mathrm{tr}(\bS_n)^2
\]
and
\[
\hat{c}_n = \sum_{j=1}^p \left(\frac{1}{n} \sum_{i=1}^n (\bx_i)_j^2 \right)^2,
\]
where $(\bx_i)_j$ is the $j^\text{th}$ component of $\bx_i$ in the standard observable basis. 
As a result, the authors define the following consistent estimator for $\exkurt$:
\[
\hat{\gamma}_2 := \max\left\{3+\frac{\hat{a}_n-2\hat{b}_n}{\hat{c}_n}, 1 \right\} -3
\]
 (If desired, this estimator can be set to 0 when infinite, but this choice is of no asymptotic importance.)

With this estimator in hand, we can establish the desired asymptotic significance level.  Specifically, as $n,p\to\infty$, the following significance level is asymptotically valid:
\[
    \Pr\left[ \Zanalyticn > \tau \mid \mathcal{H}_0 \right] \le \Phi\left(-\frac{\tau}{\sqrt{2+\max\{0,\hat{\gamma}_2\}}} \right) + o_P(1),
    \]
    where $\Phi$ is the standard normal c.d.f. and $\tau > 0$.  
To see why this holds, let $\bz_n = \bfr_n^{-1/2}\bz_n$. We recall that the variance of \[ \bz_n\tps\tilde{\bfr}_n\bz_n\] can be written as
\begin{equation} \label{S-eq:aaa}
 \exkurt\sum_{i=1}^p (\tilde{\bfr}_n)_{ii}^2 + 2\left\Vert\tilde{\bfr}_n\right\Vert_{HS}^2 ,
\end{equation}
where the matrix elements $(\tilde{\bfr}_n)_{ij}$ are evaluated in the basis with respect to which the noise components are independent.  

If $\exkurt > 0$, the summation multiplying $\exkurt$ in \eqref{S-eq:aaa} can be bounded above using the larger sum \[\
\sum_{i,j=1}^p (\tilde{\bfr}_n)_{ij}^2 =
\Vert \tilde{\bfr}_n\Vert_{HS}^2, \] allowing us to use the estimate $\hat{\gamma}_2$ in place of $\exkurt$.  On the other hand, if $\exkurt \le 0$, the variance in \eqref{S-eq:aaa} can be bounded coarsely by $2\Vert\tilde{\bfr}_n\Vert_{HS}^2$.  Thus, in both cases as $n\to\infty$, the variance or variance-upper-bound is asymptotic to $(2+\max\{0,\hat{\gamma}_2\})p(\sign)^2$.  By the asymptotic equivalence of $\sign$ and $\tilde{\sigma}_n$ (Lemma~5.3),
this is asymptotic to the proposed variance of $(2+\max\{0,\hat{\gamma}_2\})p\tilde{\sigma}_n^2$.

In particular, by observing that bounds just discussed only affect the coefficient of $\hat{\gamma}_2$,
we note that the distribution of $\Zanalyticn$ is close in total-variation norm to the distribution of a standard normal if  $\hat{\gamma}_2$ is small.  This happens, for example, if the first four moments of $W$ match a standard Gaussian's.

While the estimator $\hat{\gamma}_2$ provides an elegant method for data-driven size control in the proportional growth regime, we note two practical considerations relating to its implementation. First, because the variance of $\hat{\gamma}_2$ depends fundamentally on the eighth moment of $W$, $\hat{\gamma}_2$ may high variance and converge slowly.  Second, because consistency requires this eighth moment to be finite, $\hat{\gamma}_2$ may not be reliably computable for the heavy-tailed distributions (such as the Student-$t$ models $t_4$ and $t_6$) evaluated in Section~\ref{app:heavy-tail}.


\bdradd{
%
}

\section{Proof of Optimal Limiting Shrinkage (Theorem~5.6(b))}
\label{app:nonsingular-case}
We prove the theorem for a slightly more general objective function.  Let
\begin{equation*} \label{S-eq:Utilde-infty}
\fU(f; \xi_\infty) = \frac{\int f\, d\omegainfty}{\sigma_\infty(f; \xi_\infty)},    
\end{equation*}
where \(\sigma_\infty(\mydot; \xi_\infty)\) and \(\xi_\infty\) were defined in \eqref{S-eq:def-sigxi}.
This objective function arises, for instance, in the work of Ledoit and Wolf on financial portfolio optimization, in which case \(\xi_\infty = \omegainfty \equiv 1\), though we will not pursue this case in more detail.
More suggestively, we may re-express \(\fU(f; \xi_\infty)\)  as follows:
\begin{equation*}
    \frac{\la f, \omegafcn \ra}{\sqrt{\la M_{\mya}\myT f, \myT f \ra}},
\end{equation*}
where \(\la \mydot \ra\) is the inner product on \(\tilhilb \times \tilhilb\) and \(a(x)\) is generally \(x w(x) \xi_\infty(x)\).  Taking \(A = \myT \tps M_{\mya} \myT : \tilhilb\to \tilhilb\), where \(\myT \tps\) is the transpose of  \(\myT:\tilhilb\to\tilhilb\), Cauchy-Schwarz implies that this ratio is maximized precisely when
\[
f = A^{-1}\omegafcn .
\]

Inversion of the operator \(A\) can be performed in two main steps.  First, writing \(A^{-1} =T^{-1} M_{\mya}^{-1} (\myT \tps)^{-1} \), we would like to invert \(\myT \tps\).
A next, minor step is to multiply by the reciprocal of the function \(a(x)\).  The second major step is to invert \(T\), which can be done by writing \(T^{-1} = ((\myT \tps)^{-1})\tps\)---just the adjoint of the previously found inverse \((\myT \tps)^{-1}\).

For the first main step, observe that
\(T\) takes the following form:
\begin{align*}
\myT f(x) & =  b(x)f(x) - \asp\, \text{p.v.}\int \frac{f(y)}{y-x}\, d\ulpmeas(y) \\
&=  b(x)f(x) - \asp\pi\htrans[f\delta w](x) 
\end{align*}
where
\[
b(x):=1+\asp\pi\htrans[\delta w](x).
\]
Thus, 
\begin{align} \label{S-eq:bfBF}
    \myT \tps f(x) = 
    b(x)f(x) - B(x)\htrans f(x),
\end{align}
where 
\begin{equation*}
    B(x) := -\asp\pi\delta(x)w(x).
\end{equation*}
In order to invert \(T'\),
we need a lemma. 
\begin{lem} \label{lem:first-inversion}
The equation \(\myT \tps f = \varphi \in \tilhilb\) 
has the unique solution
\begin{equation*} \label{S-eq:spec-soln-Tprime}
f = g\varphi + G\htrans \varphi,
\end{equation*}
where \(g(x) = 1-\asp
-\asp\pi x\Hw(x)\) and \(G(x) = -\asp\pi x w(x)\).
\end{lem}
\begin{proof}
We show this in two steps.   First, we simplify \(b\) by considering the analytic signal \(b+iB\) and its extension to the complex upper half-plane.  Second, we invert 
\(T'\)
on \(\tilhilb\) by solving a singular integral equation.

First we simplify \(b\). 
Using the result from the Ledoit--P\'ech\'e theorem (Theorem~4.1)
that for \(x\in F\),
\[
\lim_{\eta\to 0^+} \mathrm{Im}\left[\uTheta(x+i\eta)\right] = \pi \delta(x)w(x) 
\]
Thus, by the properties of the Hilbert transform, \(\pi\htrans[\delta w]\) is given by the limiting real part of \(\uTheta(z)\) as \(\mathrm{Im}z\to 0^+\), so that \(b\) can be calculated using the following  identity 
(see (3.6)):
    \begin{equation*}
    \begin{split}
b(x) & = \lim_{\eta\to 0^+} \mathrm{Re}\left[1+\asp\uTheta(x+i\eta)\right] 
\\
& = \mathrm{Re}\lim_{\eta\to 0^+} \frac{1}{1-\asp-\asp z \um(z)}
\\
& = \mathrm{Re}\left[\frac{1}{1-\asp-\asp x \brevem(x)}\right]
\\
& = \frac{1-\asp-\asp x\mathrm{Re}[\brevem(x)]}{|1-\asp-\asp x\brevem(x)|^2}.
\end{split}
    \end{equation*}
    By the fact that \(\brevem(x) = \pi\Hw(x) + i\pi w(x)\), the above can be written 
    \begin{equation*}
            \frac{1-\asp -\asp \pi x \Hw(x)}{(1-\asp -\asp \pi x \Hw(x))^2 + \asp^2\pi^2 x^2 w(x)^2}.
    \end{equation*}

Assume \(\myT \tps f = \varphi \in \tilhilb\). 
Using the fact that 1 
has vanishing Hilbert transform
and \(\htrans\) is an anti-involution, we observe that \(B(x)=\htrans b(x)\). 
Thus, we may take the Hilbert transform of \eqref{S-eq:bfBF} and use the identity \( \htrans[b f - B\htrans f ]=B f + b\htrans f \) of \cite{carton1977product} to obtain a second equation \(\varphi\) satisfies:
\begin{equation*}
    B f + b\htrans f  = \htrans\varphi.
\end{equation*}
Together with \eqref{S-eq:bfBF}, this yields a two-by-two linear system for \(f\) and \(\htrans f\), which can be solved for \(f\) as follows:
\begin{equation*}
    \Tprimeinv\varphi = f = \frac{b\varphi + B\htrans\varphi}{b^2 + B^2}.
\end{equation*}
In order to simplify the coefficients \(b/(b^2+B^2)\) and \(B/(b^2 + B^2)\), observe that
\[
b(x)^2+B(x)^2 =  \frac{1}{(1-\asp-\asp \pi x \Hw(x))^2 + \asp^2\pi^2 x^2 w(x)^2},
\]
so that
\begin{equation*}
    \frac{b(x)}{b(x)^2+B(x)^2} =     1-\asp-\asp\pi x \Hw(x) = g(x)
\end{equation*}
and 
\begin{equation*}
    \frac{B(x)}{b(x)^2+B(x)^2} = -\asp\pi x w(x) = G(x),
\end{equation*}
by the definitions of \(g\) and \(G\). In other words, 
\begin{equation*} \label{S-eq:inverse-1}
    \Tprimeinv\varphi = g \varphi + G\htrans\varphi,
\end{equation*}
as desired.
\end{proof}

As a result of the above lemma, we may easily apply \(\Tprimeinv\) to \(\omegafcn\).  The next step is to apply in the inverse of \(M_{\mya}\) to the result, obtaining
\begin{align}
    \fstar & := M_{\mya}^{-1}\Tprimeinv \omegafcn \nonumber
    \\ & = \frac{g \omegafcn}{a} + \frac{ G H}{a}. \label{eq:f_sub_star}
\end{align}

The above lemma also enables us to easily compute  \((\Tprimeinv)\tps\), which coincides with \(T^{-1}\), the final operator we must apply.  Using the anti-self-adjointness of \(\htrans\) and the self-adjointness of multiplication operators, one obtains that for \(\varphi \in \tilhilb\)
\[
T^{-1}\varphi = (\Tprimeinv)\tps\varphi = g\varphi - \htrans\left[G\varphi\right].
\]
Thus,
\begin{align*}
    A^{-1}\omegafcn & = T^{-1}\fstar
    \\
    & = g\fstar - \htrans\left[G\fstar\right]
    \\
    & = g\frac{g\omegafcn + GH}{a} - \htrans\left[G\frac{g\omegafcn + GH}{a} \right],
\end{align*}
which has the desired form when \(\xi_\infty(x) = \delta(x)\).

\section{Consistency of Proposed Shrinker} \label{app:consistency}

As in previous appendices, we will use \(\mywt(x)\) and \(\Htildew f\) to denote the additive convolutions of \(\mu\) with \(k\) or \(f\,d\mu\) with \(K\) rather than the multiplicative ones, with the understanding that our results can easily be adapted to the multiplicative versions.

The main idea 
will then be to show in-probability convergence of \(p^{-3/2}\tr( \shrinker \myOmega_n)\) to \(\int f^*\, d\omegainfty \) and \(\sign(f_n)^2\) to \(\mysigma^2(f^*)\). 
For the former, [\textsc{Test}] and Cauchy-Schwarz imply that the identity \(\lv (f_n - f^*)w \rv_2 \to 0\) in probability  is sufficient.
For the latter, by the triangle inequality, it suffices to show that  \(\mysigma^2(f_n-f^*,f_n-f^*)\to  0\) in probability.  By using the expression of \(\mysigma^2(\mydot)\) in terms of the bounded operator \(\myT:\tilhilb\to\tilhilb\), we see that this also follows if \(\lv (f_n - f^*)w \rv_2 \to 0\) in probability.

As in previous appendices, we will now suppress subscripts of \(n\) wherever they are unnecessary.  To establish this norm's convergence, observe that from Section~\ref{app:nonsingular-case}, \(f^*\) takes the form of \(\psi_1\overline{h} + \psi_2\Hsubw[\psi_3\overline{h}] + \Hsubw[\psi_4\Hsubw[\overline{h}]],\) where  \(\psi_j\) are non-singular functions depending on \(\htrans w \), and that  \(f_n\) can be obtained by simply replacing each instance of \(\Hsubw\) with \(\Htildew\), \(\htrans w\) with \(\Hwt \), and   \(\overline{h}(\lambda_{nj})\) with \(\myhnj\) for all \(j\).  
However, we ignore this last step since our main approximation (see~(5.13) in Theorem~5.8)
and boundedness of \(\psi_j\) imply that identifying \(\overline{h}(\lambda_{nj}) =\myhnj\) can be done without loss of generality and without changing convergence in probability.   Thus, we must analyze terms of the form \(\Hsubw[\psi_4 \Hsubw[\overline{h}]] - \Htildew[\tilde{\psi}_4\Htildew[\overline{h}]]\), where \(\tilde{\psi}_4\) is obtained from \(\psi_4\) by substitution of \(\Hwt\) for \(\Hw\).  To shorten our argument, we will analyze only the term just given: the other terms can be analyzed similarly. 

Thus, for brevity's sake, we define \(\psi=\psi_4\Hsubw[\overline{h}]\) and \(\tilde{\psi} = \tilde{\psi}_4 \Htildew[ \overline{h}]\). Since \(\psi\) and \(\tilde{\psi}\)  are regular,  Lemma~\ref{lem:Htildew-approx} gives
\begin{align*}
    & \lv \left(\Htildew \tilde{\psi} -\Hsubw\psi\right) w \rv_2
    \\
    & \prec \lv \left(\Hsubw \tilde{\psi} - \Hsubw\psi\right) w\rv_2,  
    \end{align*}
This inequality can be continued as
    \begin{align*}
    & \prec 
     \lv \left( \tilde{\psi} - \psi\right)w \rv_2 
     \\
     & = \lv \left( \tilde{\psi}_4\Htildew \overline{h} - \psi_4 \Hsubw \overline{h} \right) w\rv_2
     \\
     &
     \prec
     \lv  \Htildew \overline{h} - \Hsubw \overline{h} \rv_2 
     \\
     & 
\prec
     \Delta,
\end{align*}
where we have used Lemma~\ref{lem:Htildew-approx} in the last step and Section~\ref{app:lw-rate} in the previous step, together with the fact that \(\psi_4\), \(\tilde{\psi}_4 \), and \(\overline{h}\) are almost surely bounded as \(n\to\infty\).  The proof is complete.


\section{Additional plots for heavy-tailed data} 
\label{app:heavy-tail}

While a major assumption of our approach is that data possess ``light-tailed'' behavior, such as sub-Gaussian, sub-exponential, or further relaxations, a reasonable question is how robust the paper's methods are against heavy-tailed data.  To investigate this question, we include in Figure~\ref{fig:isotrue_nu4} and \ref{fig:isofalse_nu4} simulations that mimic those in 
Figures~4 and 5,
but assume independent noise components with a Student-$t$ distribution $t_\nu$, having $\nu >2$ degrees of freedom.  In particular, we show the effect of decreasing the number of degrees of freedom from 8 to 6 to the challenging case of 4.  (Quadratic forms do not necessarily have finite variances for $\nu$ less than 4, challenging asymptotic normality and subsequent analysis.)

\begin{figure}[htbp]
\centering
\includegraphics[width=\columnwidth]{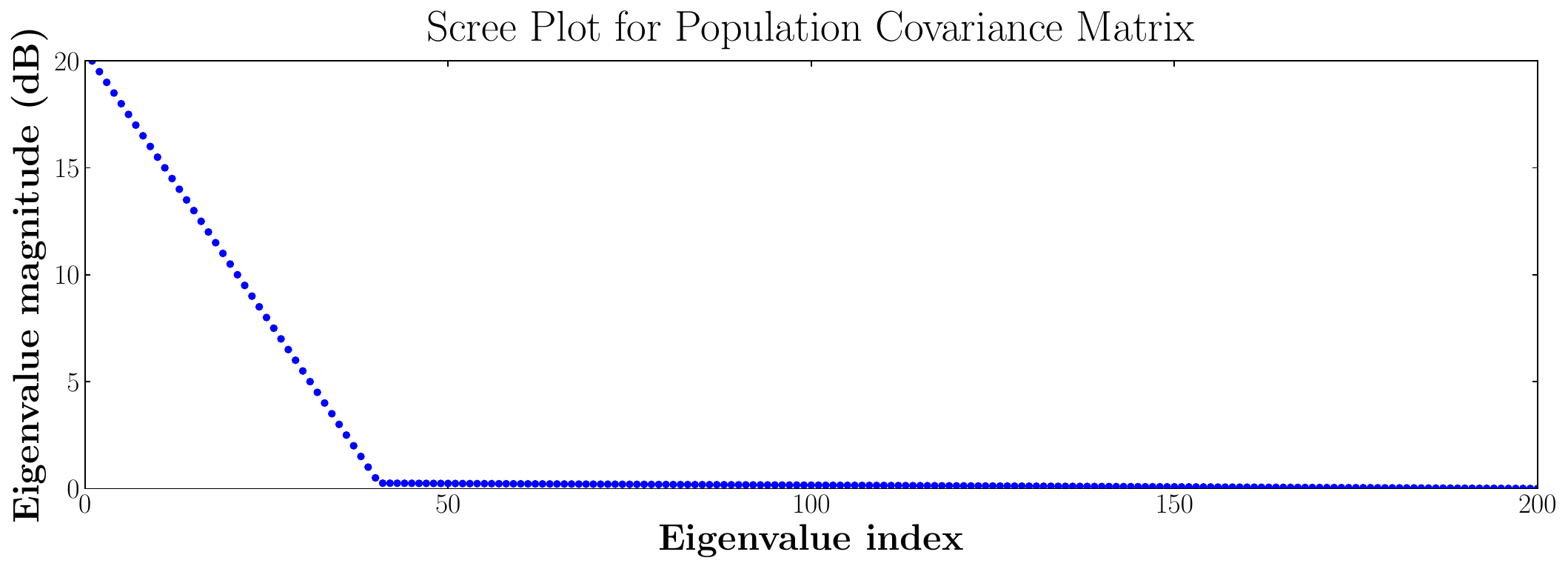}
\caption{Scree plot of population covariance matrix chosen to generate artificial data, \(\kappa=10^2\).
}
\label{S-fig:scree1e2}
\end{figure}

\begin{figure}[htbp]
    \centering
    {\large\textbf{Robustness to Heavy-Tailed Data, $\myOmega\propto \mathbf{I}$}}\\[1ex]
        \includegraphics[width=\linewidth]{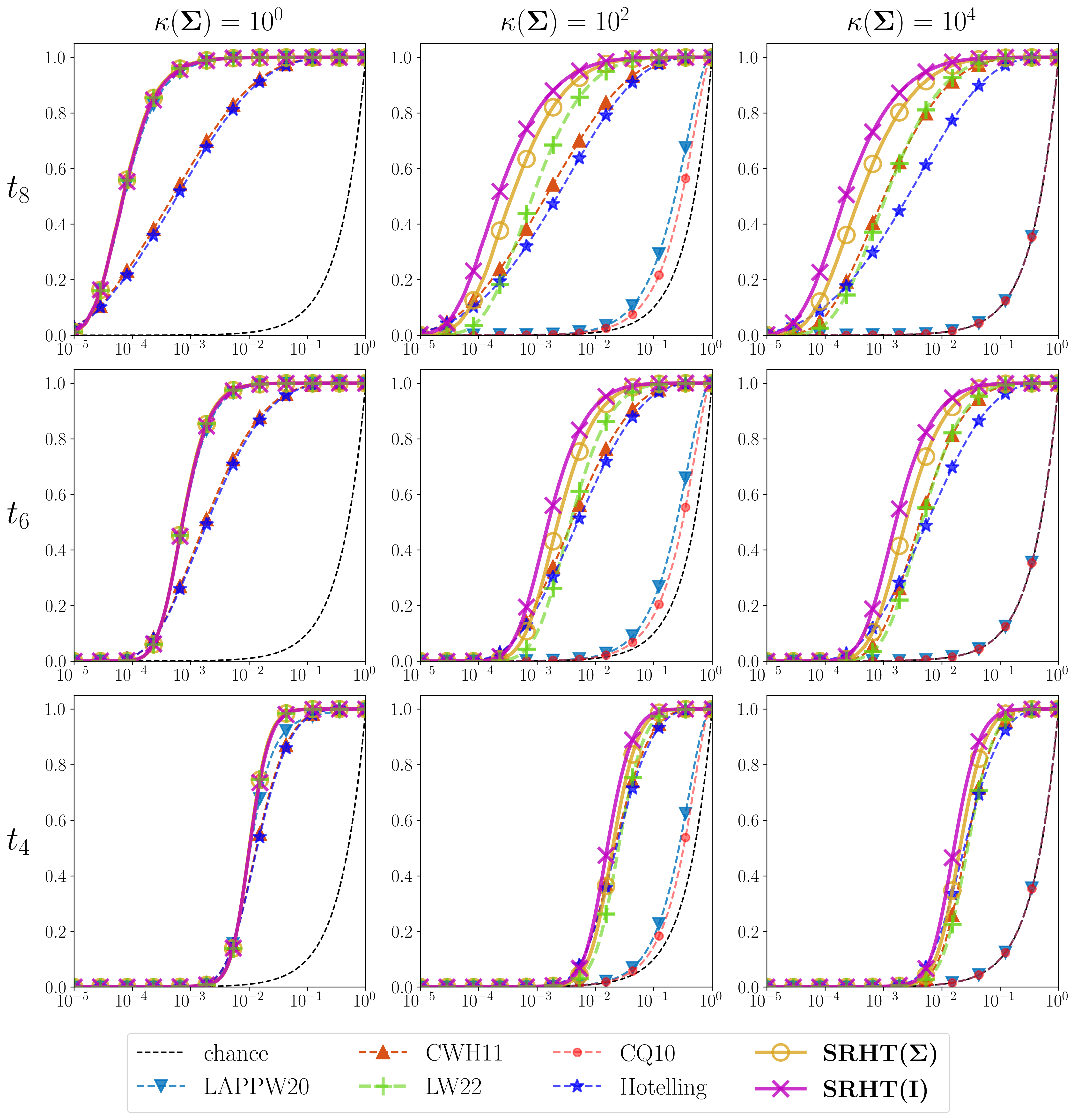}
        \caption{
        Size-adjusted empirical power of several methods, where $p=200$, $n=300$, the noise component $W\sim t_\nu$ (Student-$t$ model with $\nu$ degrees of freedom), and the true signal-dispersion matrix $\myOmega$ proportional to $\mathbf{I}$.
        \textit{From top to bottom}: The effect of increasing tail weight on algorithm performance.
        \textit{From left to right}: SRHT(\textbf{I}) increasingly outperforms the other baselines as the spectral complexity increases. SRHT($\bfr$) remains competitive throughout, though less so than SRHT(\textbf{I}). 
        }
    \label{fig:isotrue_nu4}
\end{figure}

\begin{figure}[htbp]
    \centering
    {\large\textbf{Robustness to Heavy-Tailed Data, $\myOmega\propto \bfr$}}\\[1ex]
        \includegraphics[width=\linewidth]{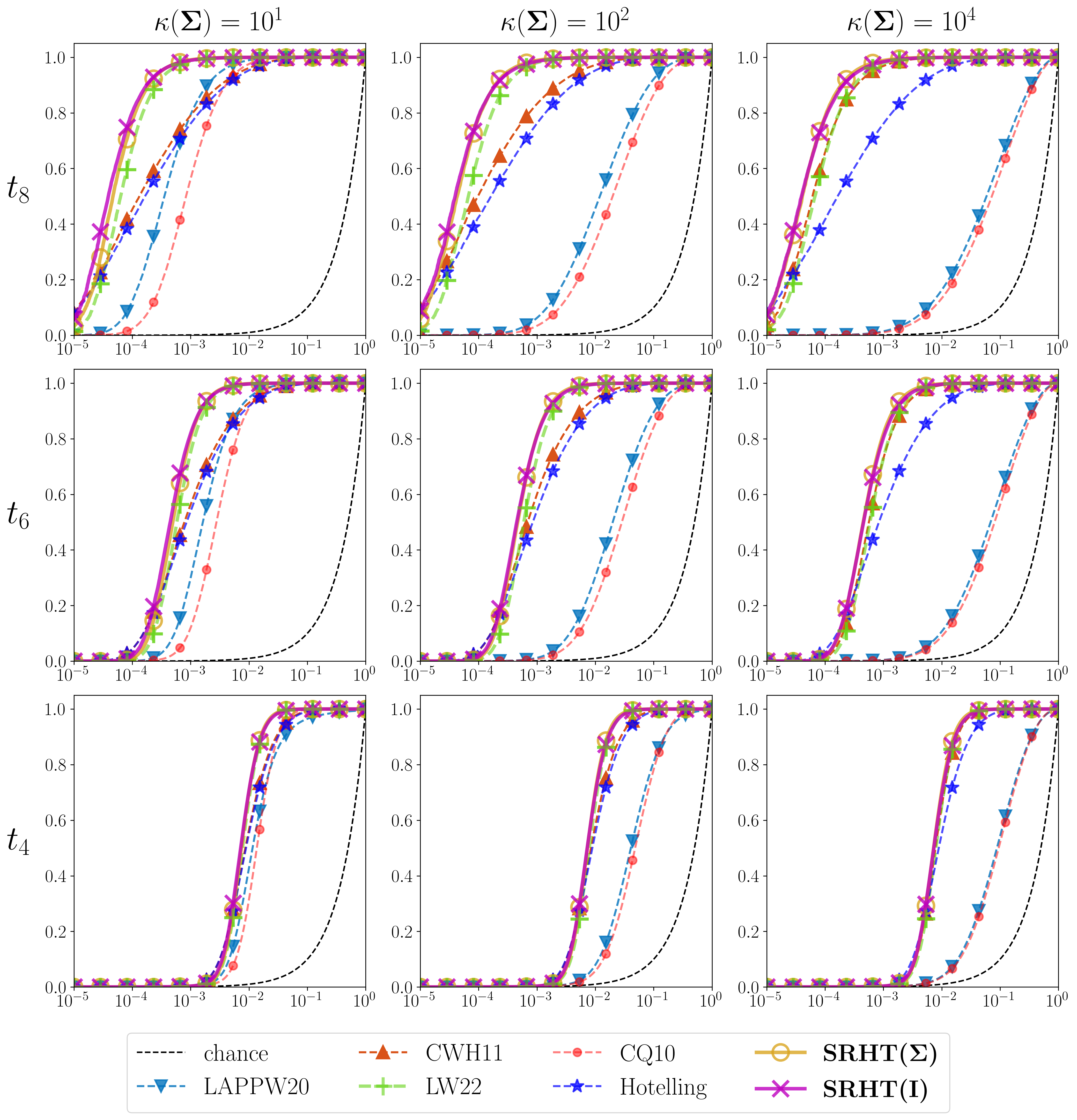}
        \caption{
        An analogous figure to Figure~\ref{fig:isotrue_nu4}, except that $\myOmega \propto \bfr$ and $\kappa(\bfr)$ begins at $10^1$ rather than $10^0$.  Our methods roughly match or outperform the competition, although SRHT($\bfr$) is slightly outperformed by SRHT(\textbf{I}) for $\kappa=10^1$.  This could indicate that if $\myOmega$ is well-conditioned enough, assuming $\myOmega=\mathbf{I}$ is a suitable choice.
        }
    \label{fig:isofalse_nu4}
\end{figure}

Recall that the population covariance matrices $\bfr$ are parameterized by a number $\kappa = \kappa(\bfr) \ge 1$.  When $\kappa = 1$, the matrix $\bfr$ is simply the identity.  When $\kappa  > 1$, the matrix $\bfr$ is taken to have piecewise log-linear spectrum of the form \(\{\kappa^{ i/40}\}_{i=1}^{40}\cup\{10^{ (i-1)/(40(p-41))}\}_{i=1}^{p-40}\).  (In both cases, $\kappa$ corresponds roughly to the condition number of $\bfr$.) We include an example of the latter type of spectrum in Figure~\ref{S-fig:scree1e2}, which can also be found in the paper's main body.  Otherwise, our data-generating procedure follows precisely the one outlined in Section~6, only choosing the Student-$t$ option rather than the sub-Gaussian one. 


  We present the result of our simulations in a grid, where condition number (and spectral complexity) increases from left to right, and $\nu$ decreases from top to bottom.  From these plots, we observe that our methods roughly match or exceed the competition even in the heavy-tailed regime.

For completeness, we present similar plots for the \textit{singular} case, where $p > n$, in Figures~\ref{fig:isotrue_nu4_n100} and \ref{fig:isofalse_nu4_n100}, but first we outline the formulae we propose to use in that case.

\section{Outline of shrinkage method for $p > n$} \label{app:sample-starved}

When $\asp > 1$, the limiting sample spectral distribution of $\bS_n = p^{-1}\bX_n\bX_n\tps$ becomes \textit{singular}, having a point mass at zero of weight $(1-1/\asp)$ and an absolutely continuous part supported on $E := F\backslash\{0\} \subset (0, \infty)$ with density function $w$. 

A shrinker in this regime may be defined as a pair of continuous functions $f = (f_c, f_s)$, labeled for ``continuous'' and ``singular,'' respectively. The singular component of this pair is only defined in a compact neighborhood of 0 disjoint from $E$, and the continuous component can be applied to any $\lambda\in E$.  As always, both functions are allowed to tacitly depend on global limiting spectral quantities ($F$, $\myw$, $\Hw$, \dots) as well as their main argument.

As in the non-degenerate case of Section~\ref{app:nonsingular-case}, the optimal shrinker is derived by maximizing the limiting utility functional $(\int f\, d\omega_\infty)^2 / \int [\Gamma f(x))]^2 x\xi(x)\, d\umu(x)$, which leads to another Cauchy--Schwarz argument. To accommodate the newly appearing point mass in $\nu_\infty$ in $\Gamma f$, the forward operator $T$ must be extended to an operator $T_+$ acting on pairs $f=(f_c, f_s)$:
\begin{align*}
T_+ f(x) = T f_c(x) + c_0 \frac{f_s(x)}{x},
\end{align*}
where $c_0 = (\asp-1)\delta(0)$, and $\delta(0)$ is the  Ledoit--P\'ech\'e $\delta$ shrinkage function at zero. Ignoring a full specification of its codomain, the adjoint operator $T_+\tps$ acting on a test function $\varphi$ maps to the pair
\[
T_+\tps \varphi = \left(  T\tps \varphi ,\, c_0 \int_E \frac{\varphi(x)}{x} \, dx \right).
\]

The optimal shrinker is the solution $f^o = (f^c, f^s)$ to the extended system $T_+\tps M_\alpha T_+ f = (h_c, h_s)$, where $h_s\equiv \la\delta_0, d\omega \ra$ and $h_c = d\omega/dx$, and $M_\alpha$ is again the multiplication operator by the function $\alpha(x) = x \xi(x) w(x)$.  Recall that for the purpose of this paper's hypothesis test $\xi(x)$ must be instantiated as $\delta(x)$.

While a fully rigorous inversion of this operator system requires a careful treatment of the relevant function spaces, the formal algebraic inversion follows steps analogous to those in Section~\ref{app:nonsingular-case}. Specifically, extending the inversion from Lemma~\ref{lem:first-inversion}, one finds that applying the generalized inverse to the $(h_c, h_s)$ yields $(T_+\tps)^{-1} (h_c, h_s) = (T\tps)^{-1}h_c + \asp h_s w$. 

Subsequent application of $M_\alpha^{-1}$, then $T_+^{-1}$ yields
\begin{equation} \label{eq:f_opt_s}
f^s(0) = \asp \int_E f_*(x)w(x)\, dx+ \asp^2 h_s \int_E \frac{w(x)}{x\xi(x)} \, dx,
\end{equation}
where $f_*$ was defined in Appendix~\ref{app:nonsingular-case}. To outline a special case, if $h_c \equiv 1$, $h_s = 1-1/\asp$, and $\xi \equiv 1$, the first integral in \eqref{eq:f_opt_s} reduces to 0, and the second integral evaluates to $\pi \mathcal{H}w(0)$. Using the relationship between the un-normalized density $w$ and the normalized one $\underline{w} = \asp w$, this simplifies exactly to $f^s(0) = 1/\delta(0)$, successfully recovering the standard Ledoit--P\'ech\'e shrinkage for degenerate eigenvalues.

With this singular component in hand, solving the system for the continuous component $f^c$ follows much as before, and yields a correction to the continuous component $f^*(x)$ appearing in the non-singular case:
\begin{equation} \label{eq:f_opt_c}
f^c(x) = f^*(x) + \asp^2 h_s \pi  \left( \mathcal{H}\left[ \frac{w}{\xi} \right](x) - \frac{\mathcal{H}w(x)}{\xi(x)} \right). \qquad (x > 0)
\end{equation}

Equations \eqref{eq:f_opt_s} and \eqref{eq:f_opt_c} provide the exact analytical formulas for which we now propose finite-sample analogues.

To approximate $f^c(\lambda_i)$, we must make a few finite-sample replacements.  For example, we replace $\htrans_w[f] = \htrans[fw]$ for sufficiently regular $f$, including $\htrans_w[1] = \htrans w$, with an extension of the definition of $\tilde{\htrans}_w[f]$:
\[
\overline{\htrans}_w[f] := p^{-1}\sum_{i=(p-n)^+ + 1}^p f(\lambda_i) K_{ni}(x-\lambda_i).
\]
Additionally, we replace $\asp$ with $p/n$, and estimate $h_s$ with the presumably given quantity $(1-1/\asp)\overline{h}_n(0)$, again equating $p/n$ with $\asp$.  This gives the approximation
\begin{equation} \label{eq:fc}
f^c_{n}(\lambda_i) = f_n(\lambda_i) + \frac{(p-n)^+}{n} \overline{h}_n(0) \pi
\left(
\overline{\htrans}_w 
\left[ \frac{1}{\xi}\right](\lambda_i) - \frac{\htrans \tilde{w}(\lambda_i)
}{\xi(\lambda_i)}\right), 
\end{equation}
where $f_n$ is given in our main approximation theorem (Theorem~5.8).
Next, to approximate $f^s(0)$, we simply use the fact that for a test function $\varphi$, the integral $\int \varphi(x) w(x)\, dx$ is the almost sure limit of $p^{-1} \sum_{i=(p-n+1)^+}^p \varphi(\lambda_i)$.  This yields the approximation:
\begin{equation} \label{eq:fs}
f^s_n(0) = \sum_{i=(p-n+1)^+}^p \left[ 
\frac{f_{*,n}(\lambda_i)}{n} + \frac{(p-n)^+}{n^2} \overline{h}_n(0) \frac{1}{\lambda_i \xi(\lambda_i)}
\right],
\end{equation}
where 
\begin{equation} \label{eq:f_substar_n}
f_{*,n}(x) := \frac{g_n(x)\overline{h}_n(x) + \overline{G}_n(x)H_n(x)}{x \tilde{d}_n(x)},
\end{equation}
 as inspired by \eqref{eq:f_sub_star}.

For finite sample sizes, the theoretically optimal shrinker $f^o$ can potentially take on negative values at zero or on a neighborhood of the limiting spectral support $F$. 
Thus, it is advantageous and makes no difference asymptotically if we threshold all eigenvalues $f^o(\lambda_j)$ from below by zero.  However, this threshold is rather conservative, and it is possible finite-sample performance could be further optimized by a tighter one.  We leave investigation of this possibility to future work.

The proposed approximations of $f^s$ and $f^c$ described above
are implemented in Python code in Listing~\ref{list:python}.  Figures~\ref{fig:isotrue_subg_lss} and \ref{fig:isofalse_subg_lss} give examples of empirical ROC curves for the case where $p=200$ and $n=100$.  A full convergence analysis would involve extensive functional-analytic complications and remains future work, but the simulations indicate that the proposed shrinkage methods achieve first-order optimality, remaining highly competitive despite possessing a potentially larger finite-sample variance than highly-tuned approaches like \lwalg{}.


\definecolor{codegreen}{rgb}{0,0.6,0}
\definecolor{codegray}{rgb}{0.5,0.5,0.5}
\definecolor{codepurple}{rgb}{0.58,0,0.82}
\definecolor{backcolour}{rgb}{0.96,0.96,0.96}

\lstdefinestyle{pythonstyle}{
    commentstyle=\color{codegreen},
    keywordstyle=\color{blue},
    numberstyle=\tiny\color{codegray},
    stringstyle=\color{codegreen},
    basicstyle=\ttfamily\footnotesize,
    breakatwhitespace=false,         
    breaklines=true,                 
    captionpos=b,                    
    keepspaces=true,                 
    numbers=left,                    
    numbersep=5pt,                  
    showspaces=false,                
    showstringspaces=false,
    showtabs=false,                  
    tabsize=4,
    mathescape=true
}
\lstset{style=pythonstyle}

\newtcolorbox{codecontainer}{
    colback=backcolour,
    boxrule=0.5pt,
    colframe=backcolour,
    sharp corners,
    enhanced,
    breakable,
    left=0pt, right=0pt, top=0pt, bottom=0pt
}

\include{lstlisting}

\begin{codecontainer} 
\begin{lstlisting}[language=Python, caption={Finite-sample implementation of the optimal shrinker, including for \(p > n\).}, label=list:python] 
import numpy as np

def optimal_precision_shrinkage(p, n, nz_lams, prior_type, hbar=None, xi=None):
    r"""
    $\color{codepurple}\textbf{\text{Description}}$: Proposed nonlinear precision shrinkage estimator, generalized to the $\color{codegreen}p > n$ regime but also valid for $\color{codegreen}n < p$.  
    $\color{codepurple}\textbf{\text{Output}}$: (p x 1) array $\color{codegreen}f=\color{codepurple}f^o(\lambda_i)$ (f_opt) that approximately optimizes
    $\color{codegreen}\mathrm{tr}(f(\bS)\myOmega)^2/\tr(\bfr f(\bS) \myXi f(\bS))$ (appropriately scaled) in the limit as $\color{codegreen}n,p\to\infty$ and $\color{codegreen}p/n\to\phi$, where $\color{codegreen}\lambda_i$ are the eigenvalues of the sample covariance matrix $\color{codegreen}\bS$, in descending order.
    $\color{codepurple}\textbf{\text{Inputs}}$: 
    - Dimension p, sample size n
    - nz_lams: descending nonzero sample eigenvalues.
    - prior_type (signal prior): 'isotropic', 'matched' (covariance-matched), or determined by hbar if any other argument is given
    - hbar[i]: estimated signal-prior spectral parameters of $\color{codegreen}\myOmega$: $\color{codegreen} \overline{h}(\lambda_i) = d\omega_\infty/d\mu_\infty(\lambda_i)$
    - xi[i]: estimated spectral parameters of $\color{codegreen}\myXi$: $\color{codegreen}\xi_\infty(\lambda_i)$, defined analogously for the matrix $\color{codegreen}\myXi$ to how $\color{codegreen}\overline{h}(\lambda_i)$ is defined for $\color{codegreen}\myOmega$
    """
    
    m = len(nz_lams)
    phi = p / n
    hn = p**(-1/3) # Bandwidth for kernel estimation

    # Semi-circular kernel and its Hilbert transform
    def k_func(x):
        return 1 / (2 * np.pi) * np.sqrt(np.maximum(4 - x**2, 0))
    def Hk_func(x):
        return (-x + np.sign(x) * np.sqrt(np.maximum(x**2 - 4, 0))) / (2 * np.pi)

    # Matrices for taking quick convolutions with k(x) and Hk(x)
    lam_diffs = nz_lams[:, None] - nz_lams[None, :]
    arg_array = lam_diffs / (nz_lams[None, :] * hn)
    Hk_arr = Hk_func(arg_array)
    k_arr = k_func(arg_array)

    # Estimate density $\color{codegreen}w(x)$ and its Hilbert transform $\color{codegreen}\mathcal{H}w(x)$ at each $\color{codegreen}\lambda_i \ne 0$
    Hw = 1 / (p * hn) * np.sum(Hk_arr / nz_lams[None, :], axis=1)
    w = 1 / (p * hn) * np.sum(k_arr / nz_lams[None, :], axis=1)

    # Construct $\color{codegreen}g(x)$ and denominator norm squared $\color{codegreen}b(x)^2 + B(x)^2$
    g_func = np.maximum(1 - phi, 0) - phi * np.pi * nz_lams * Hw
    norm_sq = g_func**2 + (phi * np.pi * nz_lams * w)**2

    # Nonlinear shrinkage function d, or $\color{codegreen}\tilde{d}_n(x)$, of (Ledoit-Wolf, 2020)
    d_nz = nz_lams / norm_sq # Non-degenerate sample eigenvalues d_nz
    Hw_zero = (1 - np.sqrt(max(1 - 4 * hn**2, 0))) / (2 * np.pi * n * hn**2) * np.sum(1.0 / nz_lams)
    d_zero = n / (np.pi * max(p - n, 1) * Hw_zero) # Degenerate ei-val
    d = np.concatenate([d_nz, np.repeat(d_zero, max(0, p - m))])
        
    # Assign xi and hbar (based on prior_type string)
    if prior_type == 'matched':
        hbar = d
    elif prior_type == 'isotropic':
        hbar = np.ones(p)
    elif hbar is None:
        exit("Error: Invalid prior type or no value given for hbar")
    if xi is None: # Default: In present setting $\color{codegreen}\myXi = \bfr$
        xi = d_nz # $\color{codegreen}\xi(x) =\tilde{d}_n(x)$

    # Estimate Hilbert transform of $\color{codegreen}h^c$
    H_hc = 1 / (p * hn) * np.sum(Hk_arr * (hbar[:m] / nz_lams)[None, :], axis=1)
    Gbar = - phi * np.pi * nz_lams

    # Function $\color{codegreen}f_{*,n}$ from $\eqref{eq:f_substar_n}$
    f_sub_star = (g_func * hbar[:m] + Gbar * H_hc) / (nz_lams * xi[:m])
    H_lambda_fsubstar = 1 / (p * hn) * np.sum(Hk_arr * f_sub_star[None, :], axis=1)

    # Approximate deterministic limiting solution $\color{codepurple}f^*$ from Section $\ref{app:nonsingular-case}$
    f_opt = g_func * f_sub_star + phi * np.pi * H_lambda_fsubstar
    # Note: 2nd term equivalent to -H_w[Gbar*f_sub_star] by def of Gbar

    # Corrections for the singular regime ($\color{codegreen}\phi = p/n > 1$)
    if phi > 1:
        h_s = 1/p * np.sum(hbar[n:p]) # Weight for null eigenvalues
        H_w_over_xi = 1 / (p * hn) * np.sum(Hk_arr / (xi * nz_lams)[None, :], axis=1)
        
        # Adjustment to obtain $\color{codegreen}f^c_n(\lambda_i)$, as in $\eqref{eq:fc}$: Note that $\color{codegreen}g(x)/x$ equals $\color{codegreen}-\phi \pi \mathcal{H}w(x)$ on $\color{codegreen}F$
        f_opt += phi * h_s * (g_func / nz_lams / xi + phi * np.pi * H_w_over_xi)
        
        # Equation $\eqref{eq:fs}$ implementation: For scalar component $\color{codegreen}f^s_n(0)$, use empirical averages np.mean() and np.sum()
        f_s = phi**2 * h_s * 1/p * np.sum(1 / (xi * nz_lams)) + 1 / n * np.sum(f_sub_star)
        f_opt = np.concatenate([f_opt, np.repeat(f_s, p - m)]) 
        
    # Return $\color{codegreen}\max\{f^o, 0\}$ to ensure positivity
    return np.maximum(f_opt, 0)

\end{lstlisting}
\end{codecontainer}

\begin{figure}[htbp!]
\centering
{\large\textbf{Singular Regime, sub-Gaussian Data, $\myOmega\propto\mathbf{I}$}}\\[1ex]
\includegraphics[width=\linewidth]{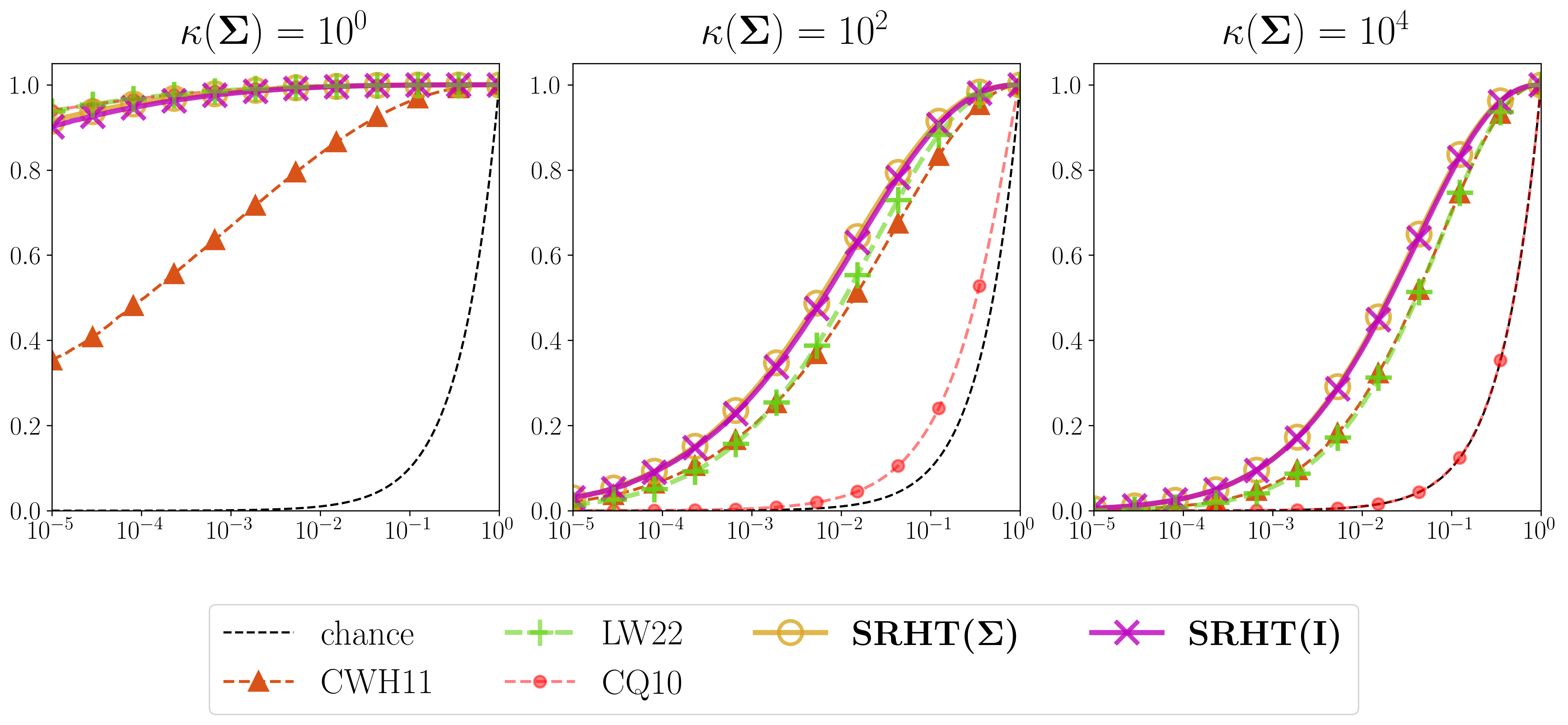}
        \caption{
        Size-adjusted empirical power of several methods, where $p=200$, $n=100$, the data are sub-Gaussian, and the true signal-dispersion matrix $\myOmega_n$ is   $\propto\mathbf{I}$.
        Our proposed methods SRHT($\mathbf{I}$) and SRHT($\bfr$) both 
        become increasingly competitive as the spectral complexity of $\bfr$ increases.
        }
    \label{fig:isotrue_subg_lss}
\end{figure}

\begin{figure}[htbp!]
\centering
{\large\textbf{Singular Regime, sub-Gaussian Data,  $\myOmega\propto\bfr$}}\\[1ex]
\includegraphics[width=\linewidth]{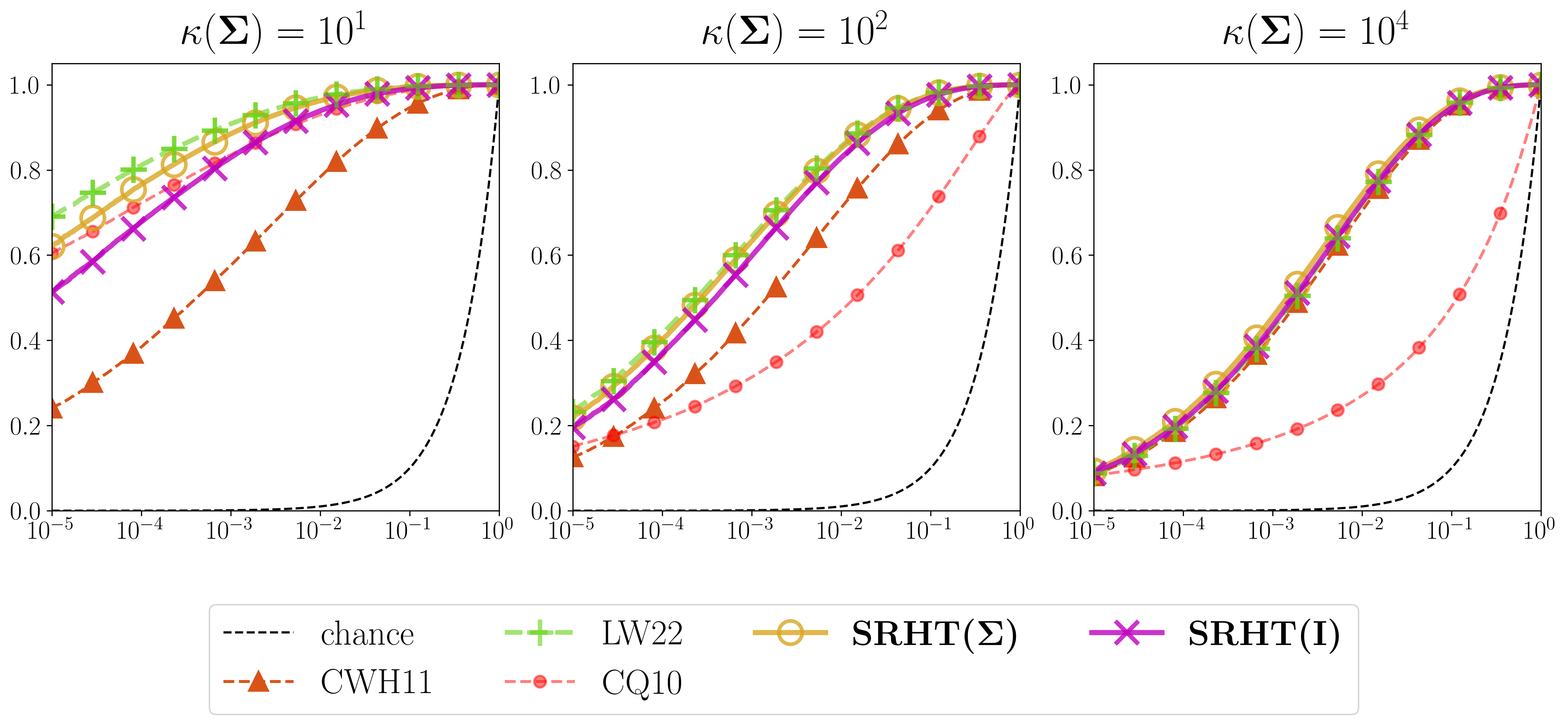}
        \caption{
        Size-adjusted empirical power of several methods, where $p=200$, $n=100$, the data are sub-Gaussian, and the true signal-dispersion matrix $\myOmega $ is $\propto\bfr$.
        While the performance of the low-variance \lwalg{} estimator is superior in the relatively well-conditioned case (left panel)---where the objective function $\mathcal{U}_{n,MV}$ it targets approximates our detection criterion $\mathcal{U}_n$---our proposed method SRHT($\bfr$) roughly matches or exceeds the top baseline as the spectral complexity of $\bfr$ increases (center and right panels).
        }
    \label{fig:isofalse_subg_lss}
\end{figure}

\begin{figure}[htbp]
    \centering
    {\large\textbf{Singular Regime, Heavy-Tailed Data, $\myOmega\propto \mathbf{I}$}}\\[1ex]
        \includegraphics[width=\linewidth]{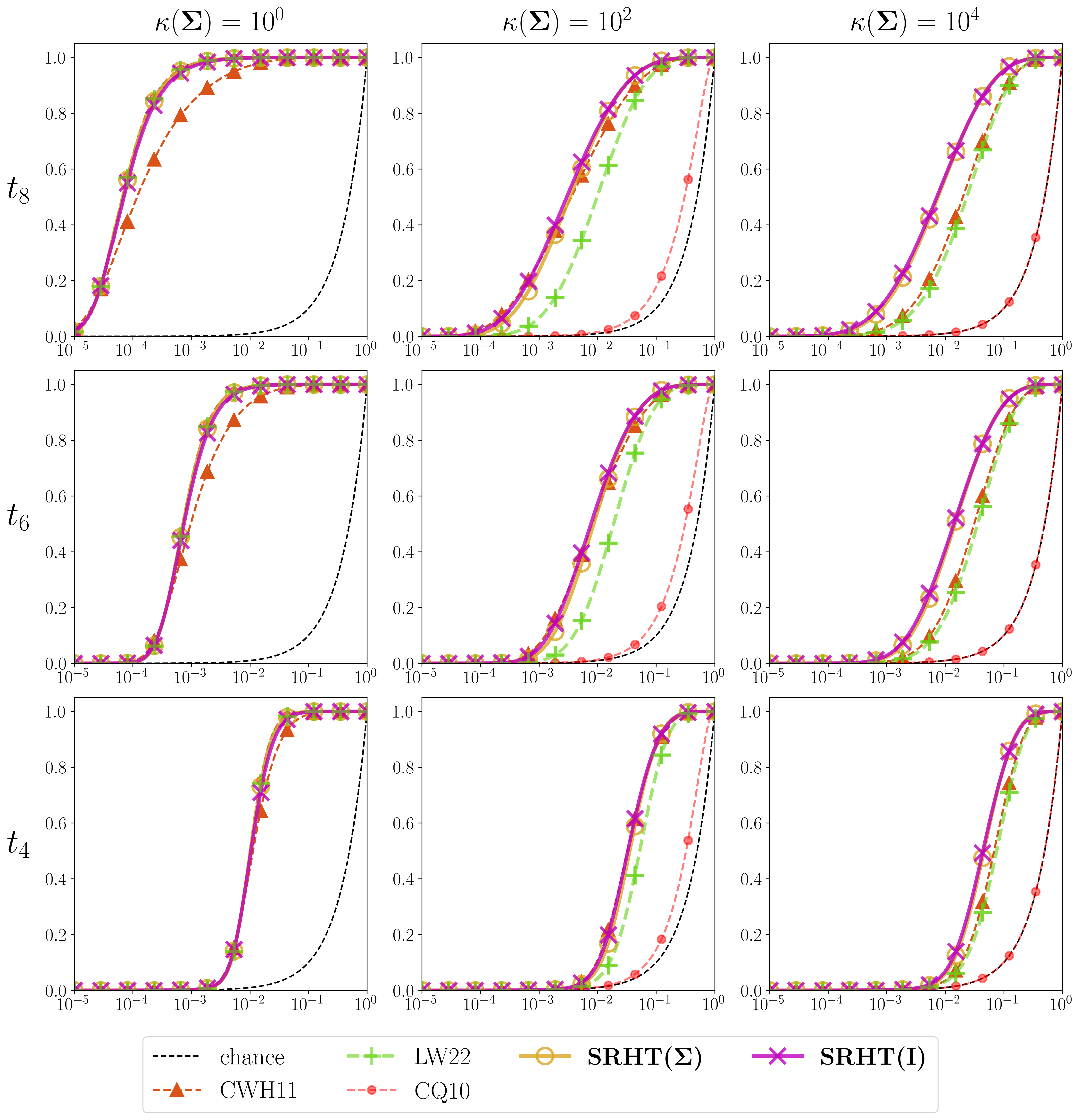}
        \caption{
        Size-adjusted empirical power of several methods, where $p=200$, $n=100$, the noise component $W\sim t_\nu$ (Student-$t$ model with $\nu$ degrees of freedom), and the true signal-dispersion matrix $\myOmega$ proportional to $\mathbf{I}$.
        In other words, an analogous figure to Figure~\ref{fig:isotrue_nu4} for $n=100$ rather than $n=300$.
        }
    \label{fig:isotrue_nu4_n100}
\end{figure}

\begin{figure}[htbp]
    \centering
    {\large\textbf{Singular Regime, Heavy-Tailed Data, $\myOmega\propto \bfr$}}\\[1ex]
        \includegraphics[width=\linewidth]{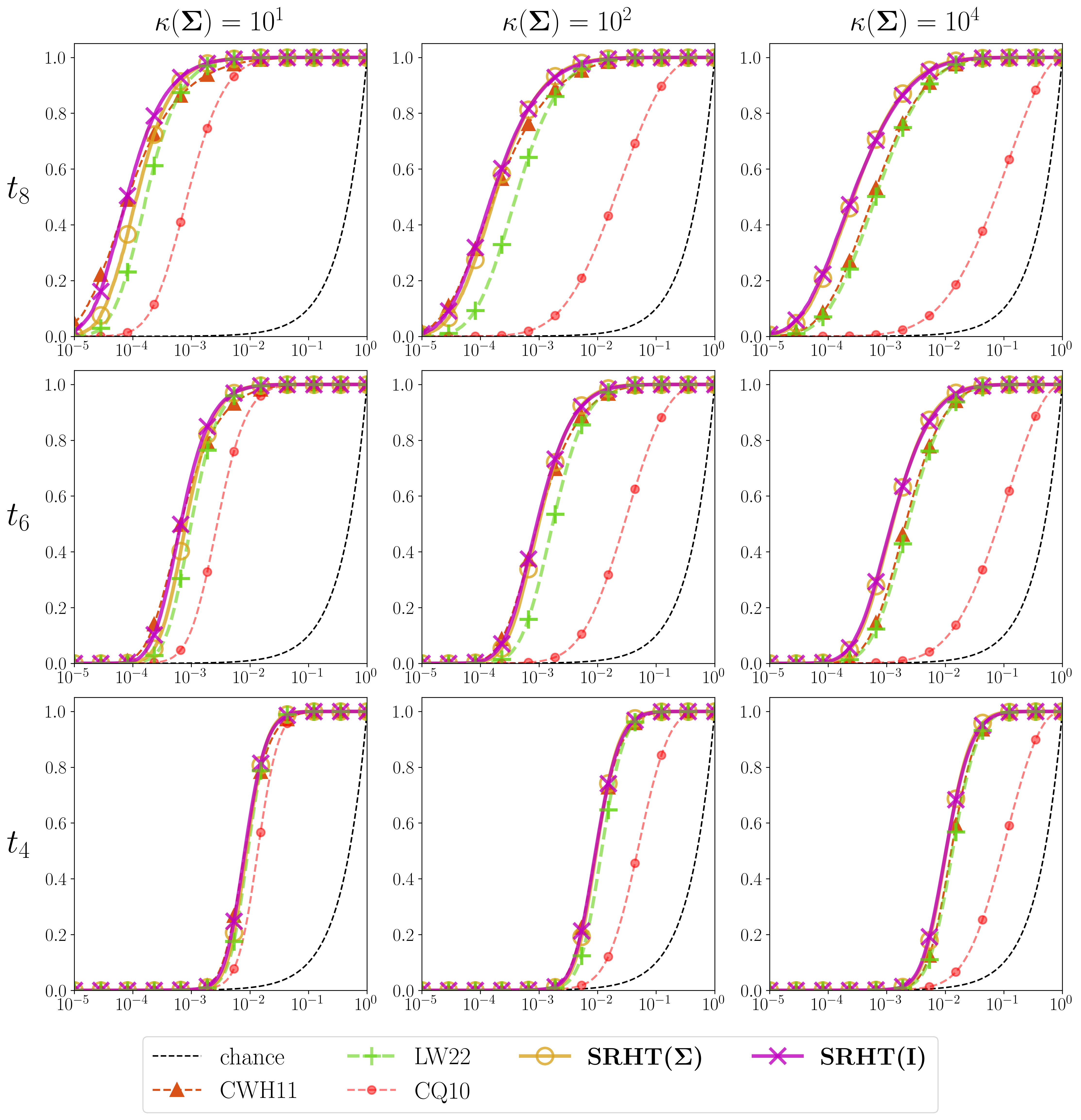}
        \caption{
        Size-adjusted empirical power of several methods, where $p=200$, $n=300$, the noise component $W\sim t_\nu$ (Student-$t$ model with $\nu$ degrees of freedom), and the true signal-dispersion matrix $\myOmega$ proportional to $\bfr$.
        In this figure, analogous to Figure~\ref{fig:isofalse_nu4} for the case of $n=100$ rather than $n=300$, our methods roughly match or outperform the competition.
        }
    \label{fig:isofalse_nu4_n100}
\end{figure}

\clearpage

\section{Data-driven selection of $\overline{h}_n(\lambda)$} \label{sec:prior-selection}

In this section, we briefly provide a possible framework for estimating important properties of $\myOmega$ when only certain structural properties of $\myOmega$ are known.  More specifically, we provide a framework for estimating the weights $ \overline{h}_n(\lambda)$ needed for our main result:  the asymptotically optimal shrinkage estimator we construct in Theorem~5.8.
As noted in our Remark on choosing a prior for $\myOmega$ (Remark~5.9),
this task requires us to go beyond the setting of our paper and assume that $\mynone \ge 2$ i.i.d. samples $\by_i \sim \by$ are available.

The structural assumption we make is that $\myOmega$ is a polynomial in $\bfr$ with unknown coefficients.  This assumption is restrictive, but since we do not explicitly assume $p/\mynone\to 0$, we have no access to consistent estimators of the eigen-basis of $\myOmega$, so the eigen-basis must be assumed.  Furthermore, polynomials in $\bfr$ are norm-dense in the continuous functional calculus of $\bfr$, and this structural assumption is not without precedent: in \cite{li2020adaptable} the authors make an assumption that, in our scaling regime, translates to $p^{-1/2}\myOmega = \pit_0 \mathbf{I} + \pit_1 \bfr + \pit_2 \bfr^2$.  This structure corresponds to a prior assumption that $p^{-1/2}\myOmega$ is equal to $\mathbf{I}, \bfr,$ or $\bfr^2$ with probabilities $\pit_0+\pit_1+ \pit_2=1$, which can in principle capture fairly ill-conditioned choices of $p^{-1/2}\myOmega$.

A natural approach is a moment-matching one.  Letting $\tilde{S}$ be the sample covariance matrix of the $\by_i$'s, a simple observation is that 
\[
\frac{1}{p^{3/2}}\tr(\tilde{\bS}) = \frac{1}{(\mynone-1)p}\sum_{i=1}^p\left\Vert \by_i - \overline{\by}\right\Vert^2,
\]
which concentrates by Hanson--Wright to
\[
\frac{1}{p^{3/2}}\tr(\myOmega) = 
w_0 +  w_1 \frac{\tr \bfr}{p} + w_2 \frac{\tr\bfr^2}{p}
\]
where $\overline{\by}$ is the mean of the $\by_i$'s.  Similarly, $p^{-3/2}\tr(\tilde{\bS}\bS)$ concentrates to
\begin{align*}
 w_0 \frac{\tr\bfr}{p} + w_1 \frac{\tr(\bfr^2)}{p} + w_2 \frac{\tr(\bfr^3)}{p}.
\end{align*}
Due to our assumption that $\pit_0+\pit_1+\pit_2=1$, it becomes possible to solve a $3\times 3$ system for the $\pit_j$'s, assuming the moments $p^{-1}\tr(\bfr^k)$ can be stably estimated for $1\le k \le 3$. From this, the ideal if unobservable value for 
$\overline{h}(\lambda_i)$ would be
\[
\tilde{w}_0 + \tilde{w}_1 \bu_i\tps \bfr\bu_i + \tilde{w}_2 \bu_i\tps \bfr^2\bu_i,
\]
where $\tilde{w}_j$ are the weights obtained by solving.  To enforce non-negativity for finite sample sizes, these weights may be thresholded from below by zero.

Fortunately, it is unnecessary to consistently estimate $\bu_i\tps \bfr \bu_i$ and $\bu_i\tps \bfr^2 \bu_i$ for each individual $i$.  The main condition on $\overline{h}_n(\lambda)$ 
is an averaged absolute error condition in Theorem~5.8.
However, establishing this condition for the special case of $\myOmega = \bfr$ takes up the bulk of this paper's theoretical work, and re-establishing a similar condition for the case of $\myOmega=\bfr^2$ (to handle a quadratic in $\bfr$) would require substantial theoretical effort to fuse the explicit spectral laws of \cite{ledoit2011eigenvectors} with the optimal eigenvector-overlap rates of \cite{lin2026eigenvector}.  A possible solution involves trading theoretical effort and expressivity: one could avoid these theoretical difficulties, by forcing $\pit_2$ to be zero, but this severely limits the expressivity of the prior for $\myOmega$.

A second practical trade-off lies in the estimation of the higher-order moments $p^{-1}\mathrm{tr}(\bfr^k)$. While consistent estimators for these tracial moments are available in the literature \cite[Theorem~2]{ledoit2011eigenvectors}, their finite-sample stability can degrade as $k$ increases, particularly when $\bfr$ is ill-conditioned. Consequently, a practical implementation must carefully balance the expressivity of the prior (\textit{i.e.}, the degree of the polynomial) against the statistical and computational stability of the moment estimates.

Ultimately, the moment-matching framework outlined above demonstrates that a data-dependent estimation of the prior weights $h_n(\lambda_i)$ is entirely feasible in principle when $\mynone\ge 2$ and the polynomial structural assumption is made. However, fully operationalizing this procedure and navigating the above trade-offs
represents a substantial theoretical undertaking. We leave further development of such data-driven priors to future work.

\end{document}